\documentclass[10pt]{article}

\usepackage[utf8]{inputenc}

\usepackage{epsf}
\usepackage{amsmath}

\allowdisplaybreaks

\usepackage[showframe=false]{geometry}
\usepackage{changepage}

\usepackage{epsfig}
\usepackage{amssymb}

\usepackage{amsthm}
\usepackage{setspace}
\usepackage{cite}
\usepackage{mcite}

\usepackage{algorithmic}  % This package provides an algorithmic environment fo describing algorithms.
\usepackage{algorithm}

\usepackage{shadow}
\usepackage{fancybox}
\usepackage{fancyhdr}

\usepackage{color}
\usepackage[usenames,dvipsnames,svgnames,table]{xcolor}

\definecolor{mypurple}{rgb}{.4,.0,.5}

\usepackage[hyphens]{url}

\usepackage[colorlinks=true,
            linkcolor=black,
            urlcolor=blue,
            citecolor=purple]{hyperref}

\usepackage{breakurl}
\def\y{{\bf y}}

\def\x{{\bf x}}

\def\x{{\mathbf x}}

\def\u{{\bf u}}

\def\x{{\bf x}}
\def\y{{\bf y}}

\def\q{{\bf q}}
\def\m{{\bf m}}

\def\h{{\bf h}}

\def\be{\begin{equation}}
\def\ee{\end{equation}}
\def\ba{\left[\begin{array}}
\def\ea{\end{array}\right]}

\def\u{{\bf u}}

\def\x{{\bf x}}
\def\y{{\bf y}}

\def\q{{\bf q}}

\def\p{{\bf p}}

\def\1{{\bf 1}}

\def\0{{\bf 0}}

\def\calX{{\mathcal X}}
\def\barcalX{ \bar{ {\mathcal X} } }
\def\calY{{\mathcal Y}}

\def\mR{{\mathbb R}}
\def\mN{{\mathbb N}}
\def\mE{{\mathbb E}}

\def\lp{\left (}
\def\rp{\right )}

\newtheorem{theorem}{Theorem}
\newtheorem{proposition}{Proposition}
\newtheorem{corollary}{Corollary}

\begin{document}

\begin{singlespace}

\title {A large deviation view of \emph{stationarized} fully lifted blirp interpolation %A tight variant of Gordon's escape through a mesh theorem
%\footnote{ This work was supported in
%part.}
}
\author{
\textsc{Mihailo Stojnic
\footnote{e-mail: {\tt flatoyer@gmail.com}} }}
\date{}
\maketitle

%%%%%%%%%%%%%%%%%%%%%%%%%%%%%%%%%%%%%%%%%%%%%%%%%%%%%%%%%%%%%%%%%%%%%%%%%%%%%%%%
%%%%%%%%%%%%%%%%%%%%%%%%%%%%%%%%%%%%%%%%%%%%%%%%%%%%%%%%%%%%%%%%%%%%%%%%%%%%%%%%
\centerline{{\bf Abstract}} \vspace*{0.1in}
%%%%%%%%%%%%%%%%%%%%%%%%%%%%%%%%%%%%%%%%%%%%%%%%%%%%%%%%%%%%%%%%%%%%%%%%%%%%%%%%
%%%%%%%%%%%%%%%%%%%%%%%%%%%%%%%%%%%%%%%%%%%%%%%%%%%%%%%%%%%%%%%%%%%%%%%%%%%%%%%%

We consider \emph{bilinearly indexed random processes} (blirp) and study their interpolating comparative mechanisms.   Generic introduction of the \emph{fully lifted} (fl) blirp interpolation in \cite{Stojnicnflgscompyx23} was followed by a corresponding stationarization counterpart  in \cite{Stojnicsflgscompyx23}.  A  \emph{large deviation} upgrade of  \cite{Stojnicnflgscompyx23} introduced in companion paper \cite{Stojnicnflldp25} is complemented here with the corresponding one of \cite{Stojnicsflgscompyx23}. Similarly to \cite{Stojnicnflldp25}, the mechanism that we introduce extends the range of \cite{Stojnicsflgscompyx23}'s applicability so that it encompasses  random structures \emph{atypical} features. Among others these include  the \emph{local entropies} (LE)  which explain atypical solutions clusterings in hard random optimization problems believed to be directly responsible for the presumable existence of the so-called \emph{computational gaps}.  Moreover (and similar to \cite{Stojnicnflgscompyx23}), despite on occasion somewhat involved  technical considerations, the final forms of the uncovered fundamental interpolating parameters relations are rather elegant and as such provide a valuable tool readily available for further use.

\vspace*{0.25in} \noindent {\bf Index Terms: Random processes; comparison, lifting, stationarization, large deviations}.

\end{singlespace}

%%%%%%%%%%%%%%%%%%%%%%%%%%%%%%%%%%%%%%%%%%%%%%%%%%%%%%%%%%%%%%%%%
\section{Introduction}
\label{sec:back}
%%%%%%%%%%%%%%%%%%%%%%%%%%%%%%%%%%%%%%%%%%%%%%%%%%%%%%%%%%%%%%%%%

Comparative mechanisms associated with random processes have played a key role in analysis of random structures and optimization problems over the last two decades. Statistical analyses critically powered by their use provided a strong progress in explaining a host of underlying phenomena in a variety of scientific and engineering fields including statistical physics, information theory, signal and image processing, compressed sensing, AI, machine learning, and neural networks (see,  e.g., \cite{Guerra03,Tal06,Pan10,Pan10a,Pan13,Pan13a,StojnicISIT2010binary,StojnicCSetam09,StojnicUpper10,StojnicICASSP10knownsupp,StojnicICASSP10block,StojnicICASSP10var,StojnicDiscPercp13,StojnicGardGen13}).
Concrete examples include precise analysis of phase transitions phenomena, characterizations of  standard optimization metrics (objective values, optimal solutions) typical behaviors, and very accurate explanations of algorithmic accuracy, speed, and convergence. Despite demonstrated success, simple comparison forms  solid upgrades  are often needed as more complex random structures are faced. A  typical such situation occurs in the famous SK quadratic/tensorial model \cite{SheKir72} where standard Slepian/Gordon max comparison forms \cite{Slep62,Gordon85,Sudakov71,Fernique74,Fernique75,Kahane86,Stojnicgscomp16,Adler90,Lifshits85,LedTal91,Tal05} come up short of achieving \emph{exact} free energy characterizations \cite{Guerra03,Tal06,Pan10,Pan10a,Pan13,Pan13a,Parisi80,Par79,Par80,Par83}. Many different random models including quadratic, bilinear, or minmax forms face similar difficulties (see, e.g., \cite{StojnicLiftStrSec13,StojnicMoreSophHopBnds10,StojnicRicBnds13,StojnicAsymmLittBnds11,StojnicGardSphNeg13,StojnicGardSphErr13} and references therein for a host of well known examples ranging from high-dimensional geometry to signal processing, machine learning and statistical mechanics).  As observations from \cite{Guerra03,Tal06,Pan10,Pan10a,Pan13,Pan13a} regarding the quadratic/tensorial max forms and those from \cite{Stojnicgscomp16,Stojnicgscompyx16,Stojnicnflgscompyx23,Stojnicsflgscompyx23} regarding forms  from \cite{StojnicLiftStrSec13,StojnicMoreSophHopBnds10,StojnicRicBnds13,StojnicAsymmLittBnds11,StojnicGardSphNeg13,StojnicGardSphErr13} hint, any progress in understanding underlying intrinsic structures is likely to critically depend on the upgrades in associated random processes' comparisons.

An extra caveat on top of all of the above is that the results of \cite{Guerra03,Tal06,Pan10,Pan10a,Pan13,Pan13a,Stojnicgscomp16,Stojnicgscompyx16,Stojnicnflgscompyx23,Stojnicsflgscompyx23} usually relate to the so-called \emph{typical} behavior prominently exhibited in the so-called ground state regime. A need to study large deviations counterparts emerged in recent years in large part due to the role they are predicated to play in explanation of hard optimization problems \emph{computational gaps} (for more -- on where such gaps currently exist and what their overall relevance in modern algorithmic theory is -- see, e.g., \cite{MMZ05,GamarSud14,GamarSud17,GamarSud17a,AchlioptasR06,AchlioptasCR11,GamMZ22,AlaouiMS22}). While generically resolving mysteries surrounding gaps remains an extraordinary challenge, a strong progress has been made towards their demystification in many particular instances \cite{GamarSud14,GamarSud17,GamarSud17a,Bald15,Bald16,Bald20,Bald21,BaldassiBBZ07,BaldMPZ23,BMPZ23,AbbLiSly21a,Barb24,BarbAKZ23,GamKizPerXu22}. Two approaches based on: \textbf{\emph{(i)}} the Overlap gap property (OGP) \cite{GamarSud14,GamarSud17,GamarSud17a,GamKizPerXu22} and  \textbf{\emph{(ii)}} the Local entropy (LE) \cite{Bald15,Bald16,Bald20,Bald21,BaldassiBBZ07,BaldMPZ23,BMPZ23,AbbLiSly21a,Barb24,BarbAKZ23} distinguished themselves as particularly fruitful avenues  in recent years (for other approachers, often tailored towards specific  algorithmic groups, see, e.g., \cite{GamarnikJW20,Wein23,BandeiraAHSWZ22,HopkinsSS15,HopkinsSSS16,HopkinsKPRSS17,DiakonikolasKS17,FeldmanPV18}). Due to their combinatorial or nonconvex nature both OGP and LE are analytically fairly difficult to handle. Nonetheless, a statistical mechanics large deviation approach  allows for particularly elegant LEs formulations \cite{Bald16,Bald15}.

Introduction of fully lifted (fl) blirp interpolating comparison and its stationarized variant in \cite{Stojnicnflgscompyx23} and \cite{Stojnicsflgscompyx23} allowed studying \emph{typical} features.  Companion paper \cite{Stojnicnflldp25} presents a large deviation fl upgrade thereby  extending the range of \cite{Stojnicnflgscompyx23}'s applicability and allowing for studying \emph{atypical} features as well. To complete the practical usability, we here follow into the footsteps of \cite{Stojnicsflgscompyx23} and provide a stationarized counterpart to  \cite{Stojnicnflldp25}.  A generic nature of the presented machinery allows for a rather wide range of concrete examples where it can be practically used. As discussed in \cite{Stojnicnflldp25}, these include various perceptrons forms, spherical \cite{FPSUZ17,FraHwaUrb19,FraPar16,FraSclUrb19,FraSclUrb20,AlaSel20,StojnicGardGen13,StojnicGardSphErr13,StojnicGardSphNeg13,GarDer88,Gar88,Schlafli,Cover65,Winder,Winder61,Wendel62,Cameron60,Joseph60,BalVen87,Ven86,SchTir02,SchTir03}, binary asymmetric \cite{StojnicGardGen13,GarDer88,Gar88,StojnicDiscPercp13,KraMez89,GutSte90,KimRoc98,TalBook11a,NakSun23,BoltNakSunXu22,PerkXu21,CXu21,DingSun19,Huang24,Stojnicbinperflrdt23,LiSZ24} or symmetric  \cite{AubPerZde19,AbbLiSly21a,AbbLiSly21b,Bald20,GamKizPerXu22,PerkXu21,ElAlGam24,SahSaw23,Barb24,djalt22,BarbAKZ23}, positive/negative Hopfield  \cite{Hop82,PasFig78,Hebb49,PasShchTir94,ShchTir93,BarGenGueTan10,BarGenGueTan12,Tal98,StojnicMoreSophHopBnds10,BovGay98,TalBook11a}
 or Little \cite{BruParRit92,Little74,BarGenGue11bip,CabMarPaoPar88,AmiGutSom85,StojnicAsymmLittBnds11} forms and many others.

To make the flow of the paper smoother, in Sections \ref{sec:gencon} we  discuss the key ingredients of the stationarization mechanism and how it works on th first level of full lifting. In Section \ref{sec:rthlev}, we then  proceed with formalization of the corresponding  $r$-th ($r\in\mN$) level extensions.

%%%%%%%%%%%%%%%%%%%%%%%%%%%%%%%%%%%%%%%%%%%%%%%%%%%%%%%%%%%%%%%%%
\section{$\p,\q$-derivatives -- first level of lifting}
\label{sec:gencon}
%%%%%%%%%%%%%%%%%%%%%%%%%%%%%%%%%%%%%%%%%%%%%%%%%%%%%%%%%%%%%%%%%

 For given sets $\calX=\{\x^{(1)},\x^{(2)},\dots,\x^{(l)}\}$, $\bar{\calX}=\{\bar{\x}^{(1)},\bar{\x}^{(2)},\dots,\bar{\x}^{(l)}\}$,  and  $\calY=\{\y^{(1)},\y^{(2)},\dots,\y^{(l)}\}$ with $\x^{(i)},\bar{\x}^{(i)}\in \mR^n$ and $\y^{(i)}\in \mR^m$, vectors $\p=[\p_0,\p_1,\p_2]$ and $\q=[\q_0,\q_1,\q_2]$ with $\p_0\geq \p_1\geq \p_2= 0$ and $\q_0\geq \q_1\geq \q_2= 0$, real parameters $\beta$, $p$, and $s$ with $\beta,p>0$, and function $f_{\bar{\x}^{(i)}}(\cdot):\mR^n\rightarrow\mR$,  we consider the following
{\small\begin{equation}\label{eq:genanal1}
 f(G,u^{(4,1)},u^{(4,2)},\calX,\barcalX,\calY,\p,\q,\beta,s,p,f_{\bar{\x}^{(i)}} (\cdot))
 =
 \frac{ \log\lp \sum_{i_3=1}^{l}
\lp  \sum_{i_1=1}^{l}\lp\sum_{i_2=1}^{l}e^{\beta  D_{00}^{(i_1,i_2,i_3)}} \rp^{s}\rp^p\rp  }{p|s|\sqrt{n}},
\end{equation}}
where
\begin{eqnarray}\label{eq:genanal1a}
 D_{00}^{(i_1,i_2,i_3)} & \triangleq &
  \lp (\y^{(i_2)})^T
 G\x^{(i_1)}+\|\x^{(i_1)}\|_2\|\y^{(i_2)}\|_2 (a_1u^{(4,1)}+a_2u^{(4,2)})  + f_{\bar{\x}^{(i_3)}} (\x^{(i_1)})    \rp.
 \end{eqnarray}

\noindent We are interested in random structures and assume that elements of  $G\in \mR^{m\times n}$,  $u^{(4,1)}$, and $u^{(4,2)}$ are independent standard normals. After setting  $a_1=\sqrt{\p_0\q_0-\p_1\q_1}$ and $a_2=\sqrt{\p_1\q_1}$, we for a scalar $\m=[\m_1]$ note that \cite{Stojnicnflldp25} recognized  a key role played by the following function when studying $ f(G,u^{(4,1)},u^{(4,2)},\calX,\barcalX,\calY,\p,\q,\beta,s,p,f_{\bar{\x}^{(i)}} (\cdot))
$
\begin{eqnarray}\label{eq:genanal2}
\xi(\calX,\barcalX,\calY,\p,\q,\m,\beta,s,p,f_{\bar{\x}^{(i)}} (\cdot))  \triangleq   \frac{ \mE_{G,u^{(4,2)}} \log
\lp
\sum_{i_3=1}^{l} \lp \mE_{u^{(4,1)}}\lp \sum_{i_1=1}^{l}\lp\sum_{i_2=1}^{l}e^{\beta  D_{00}^{(i_1,i_2,i_3)}  } \rp^{s}  \rp^{\m_1} \rp^p \rp  }{p|s|\sqrt{n}\m_1}.
\end{eqnarray}
In (\ref{eq:genanal2}) as well as everywhere else in the paper, we adopt standard convention that a subscript next to  $\mE$ specifies the randomness with repsect to which the expectation is taken (if the subscript is omitted, the expectation is with respect to all sources of randomness). To study properties of $\xi(\cdot)$ we follow into the footsteps of \cite{Stojnicgscomp16,Stojnicgscompyx16,Stojnicnflgscompyx23,Stojnicnflldp25,Stojnicsflgscompyx23} and consider the following so-called interpolating function $\psi(\cdot)$
\begin{equation}\label{eq:genanal3}
\psi(t)  =
  \frac{ \mE_{G,u^{(4,2)},\u^{(2,2)},\h^{(2)}} \log
\lp
\sum_{i_3=1}^{l} \lp \mE_{u^{(4,1)},\u^{(2,1)},\h^{(1)}}\lp \sum_{i_1=1}^{l}\lp\sum_{i_2=1}^{l}e^{\beta  D_{0}^{(i_1,i_2,i_3)}  } \rp^{s}  \rp^{\m_1} \rp^p \rp  }{p|s|\sqrt{n}\m_1},
\end{equation}
where
\begin{eqnarray}\label{eq:genanal3a}
 D_0^{(i_1,i_2,i_3)} & \triangleq & \sqrt{t}(\y^{(i_2)})^T
 G\x^{(i_1)}+\sqrt{1-t}\|\x^{(i_1)}\|_2 (\y^{(i_2)})^T(b_1\u^{(2,1)}+b_2\u^{(2,2)})\nonumber \\
 & & +\sqrt{t}\|\x^{(i_1)}\|_2\|\y^{(i_2)}\|_2 ( a_1 u^{(4,1)}+ a_2 u^{(4,2)}) +\sqrt{1-t}\|\y^{(i_2)}\|_2( c_1 \h^{(1)}+ c_2 \h^{(2)})^T\x^{(i_1)}
\nonumber \\
& &  + f_{\bar{\x}^{(i_3)}} (\x^{(i_1)}),
 \end{eqnarray}
and the elements of  $\u^{(2,1)},\u^{(2,2)}\in\mR^m$ and $\h^{(1)},\h^{(2)}\in \mR^n$ are independent standard normals. We set $b_1=\sqrt{\p_0-\p_1}$, $b_2=\sqrt{\p_1}$, $c_1=\sqrt{\q_0-\q_1}$, and $c_2=\sqrt{\q_1}$. As observed in \cite{Stojnicnflldp25}, $\xi(\calX,\barcalX,\calY,\p,\q,\m,\beta,s,p,f_{\bar{\x}^{(i)}} (\cdot)) =\psi(1)$ and since $\psi(0)$ is presumably easier to handle than $\psi(1)$, one would like to connect  $\psi(1)$ to $\psi(0)$ and thereby establish a direct connection between $\xi(\calX,\barcalX,\calY,\p,\q,\m,\beta,s,p,f_{\bar{\x}^{(i)}} (\cdot))$ and $\psi(0)$. Practically speaking, this would enable to connect the original blirp and its two decoupled  linear alternatives.

We find it convenient to also set
\begin{eqnarray}\label{eq:genanal4}
\bar{\u}^{(i_1,1)} & =  & \frac{G\x^{(i_1)}}{\|\x^{(i_1)}\|_2} \nonumber \\
\u^{(i_1,3,1)} & =  & \frac{(\h^{(1)})^T\x^{(i_1)}}{\|\x^{(i_1)}\|_2} \nonumber \\
\u^{(i_1,3,2)} & =  & \frac{(\h^{(2)})^T\x^{(i_1)}}{\|\x^{(i_1)}\|_2},
\end{eqnarray}
and after denoting the $j$-th component of $\bar{\u}^{(i_1,1)}$ by $\bar{\u}_j^{(i_1,1)}$ , we have
\begin{eqnarray}\label{eq:genanal5}
\bar{\u}_j^{(i_1,1)} & =  & \frac{G_{j,1:n}\x^{(i_1)}}{\|\x^{(i_1)}\|_2},1\leq j\leq m,
\end{eqnarray}
where  $G_{j,1:n}$  is the $j$-th row of $G$. Clearly, the elements of $\bar{\u}^{(i_1,1)}$, $\u^{(2,1)}$, $\u^{(2,2)}$, $\u^{(i_1,3,1)}$, and $\u^{(i_1,3,2)}$ are i.i.d. standard normals. Setting ${\mathcal U}_k=\{u^{(4,k)},\u^{(2,k)},\h^{(k)}\},k\in\{1,2\}$, allows to rewrite (\ref{eq:genanal3}) as
\begin{equation}\label{eq:genanal6}
\psi(t)  =   \frac{\mE_{G,{\mathcal U}_2}
\log
\lp
\sum_{i_3=1}^{l}
\lp \mE_{{\mathcal U}_1} \lp \sum_{i_1=1}^{l}\lp\sum_{i_2=1}^{l}A_{i_3}^{(i_1,i_2)} \rp^{s}\rp^{\m_1} \rp^p \rp  }{p|s|\sqrt{n}\m_1}
=
 \frac{\mE_{G,{\mathcal U}_2}
\log
\lp
\sum_{i_3=1}^{l}
\lp \mE_{{\mathcal U}_1} Z_{i_3}^{\m_1} \rp^p \rp  }{p|s|\sqrt{n}\m_1},
\end{equation}
where $\beta_{i_1}=\beta\|\x^{(i_1)}\|_2$ and
\begin{eqnarray}\label{eq:genanal7}
B^{(i_1,i_2)} & \triangleq &  \sqrt{t}(\y^{(i_2)})^T\bar{\u}^{(i_1,1)}+\sqrt{1-t} (\y^{(i_2)})^T( b_1\u^{(2,1)}+b_2\u^{(2,2)}) \nonumber \\
D^{(i_1,i_2,i_3)} & \triangleq &  (B^{(i_1,i_2)}+\sqrt{t}\|\y^{(i_2)}\|_2 ( a_1 u^{(4,1)}+ a_2 u^{(4,2)})+\sqrt{1-t}\|\y^{(i_2)}\|_2( c_1 \u^{(i_1,3,1)}+c_2\u^{(i_1,3,2)})
\nonumber \\
& &
+ f_{\bar{\x}^{(i_3)}} (\x^{(i_1)}) )
\nonumber \\
A_{i_3}^{(i_1,i_2)} & \triangleq &  e^{\beta_{i_1}D^{(i_1,i_2,i_3)}}\nonumber \\
C_{i_3}^{(i_1)} & \triangleq &  \sum_{i_2=1}^{l}A_{i_3}^{(i_1,i_2)}\nonumber \\
Z_{i_3} & \triangleq & \sum_{i_1=1}^{l} \lp \sum_{i_2=1}^{l} A_{i_3}^{(i_1,i_2)}\rp^s =\sum_{i_1=1}^{l}  (C_{i_3}^{(i_1)})^s.
\end{eqnarray}
Since we are particularly interested in studying the effect $\p$ and $\q$ have on the relation between $\psi(1)$ and $\psi(0)$, we start by considering the monotonicity of $\psi(t)$ with respect to these parameters. We have for its derivative
\begin{eqnarray}\label{eq:genanal9}
\frac{d\psi(t)}{d\p_1} & = &
\frac{d}{d\p_1}\lp
\frac{\mE_{G,{\mathcal U}_2}
\log
\lp
\sum_{i_3=1}^{l}
\lp \mE_{{\mathcal U}_1} Z_{i_3}^{\m_1} \rp^p \rp  }{p|s|\sqrt{n}\m_1}
\rp
\nonumber \\
& = &  \mE_{G,{\mathcal U}_2}
\sum_{i_3=1}^{l} \frac{
\lp \mE_{{\mathcal U}_1} Z_{i_3}^{\m_1} \rp^{p-1}    }{|s|\sqrt{n}\m_1  \lp
\sum_{i_3=1}^{l}
\lp \mE_{{\mathcal U}_1} Z_{i_3}^{\m_1} \rp^p \rp   }
\frac{d \mE_{{\mathcal U}_1} Z_{i_3}^{\m_1} }{d\p_1}\nonumber \\
& = &  \mE_{G,{\mathcal U}_2}
\sum_{i_3=1}^{l} \frac{\m_1
\lp \mE_{{\mathcal U}_1} Z_{i_3}^{\m_1} \rp^{p-1}    }{|s|\sqrt{n}\m_1  \lp
\sum_{i_3=1}^{l}
\lp \mE_{{\mathcal U}_1} Z_{i_3}^{\m_1} \rp^p \rp   }
\mE_{{\mathcal U}_1} \frac{1}{Z_{i_3}^{1-\m_1}}\frac{d Z_{i_3}}{d\p_1}\nonumber \\
& = &  \mE_{G,{\mathcal U}_2}
 \sum_{i_3=1}^{l} \frac{\m_1
\lp \mE_{{\mathcal U}_1} Z_{i_3}^{\m_1} \rp^{p-1}    }{|s|\sqrt{n}\m_1  \lp
\sum_{i_3=1}^{l}
\lp \mE_{{\mathcal U}_1} Z_{i_3}^{\m_1} \rp^p \rp   }
\mE_{{\mathcal U}_1} \frac{1}{Z_{i_3}^{1-\m_1}} \frac{d\lp \sum_{i_1=1}^{l} \lp \sum_{i_2=1}^{l} A_{i_3}^{(i_1,i_2)}\rp^s \rp }{d\p_1}\nonumber \\
& = &   \mE_{G,{\mathcal U}_2}
\sum_{i_3=1}^{l} \frac{s\m_1
\lp \mE_{{\mathcal U}_1} Z_{i_3}^{\m_1} \rp^{p-1}    }{|s|\sqrt{n}\m_1  \lp
\sum_{i_3=1}^{l}
\lp \mE_{{\mathcal U}_1} Z_{i_3}^{\m_1} \rp^p \rp   }
\mE_{{\mathcal U}_1} \frac{1}{Z_{i_3}^{1-\m_1}}  \sum_{i=1}^{l} (C_{i_3}^{(i_1)})^{s-1} \nonumber \\
& & \times \sum_{i_2=1}^{l}\beta_{i_1}A_{i_3}^{(i_1,i_2)}\frac{dD^{(i_1,i_2,i_3)}}{d\p_1},
\end{eqnarray}
where from \cite{Stojnicsflgscompyx23}'s (14) we also have
 \begin{eqnarray}\label{eq:genanal10b}
\frac{dD^{(i_1,i_2,i_3)}}{d\p_1} & = & \bar{T}_2+\bar{T}_1,
\end{eqnarray}
with
\begin{eqnarray}\label{eq:genanal10c}
 \bar{T}_2 & = & \sum_{j=1}^{m}\sqrt{1-t}\frac{\y_j^{(i_2)}\u_j^{(2,2)}}{2\sqrt{\p_1}}+\sqrt{t}\q_1\frac{\|\y^{(i_2)}\|_2 u^{(4,2)}}{2\sqrt{\p_1\q_1}} \nonumber\\
\bar{T}_1 & = &  -\sqrt{1-t}\sum_{j=1}^{m} \frac{\y_j^{(i_2)}\u_j^{(2,1)}}{2\sqrt{\p_0-\p_1}}  -\sqrt{t}\q_1\frac{\|\y^{(i_2)}\|_2 u^{(4,1)}}{2\sqrt{\p_0\q_0-\p_1\q_1}}.
\end{eqnarray}
A combination of (\ref{eq:genanal9}) and (\ref{eq:genanal10b}) then gives
\begin{equation}\label{eq:genanal10d}
\frac{d\psi(t)}{d\p_1}  =   \mE_{G,{\mathcal U}_2}
\sum_{i_3=1}^{l} \frac{s\m_1
\lp \mE_{{\mathcal U}_1} Z_{i_3}^{\m_1} \rp^{p-1}    }{|s|\sqrt{n}\m_1  \lp
\sum_{i_3=1}^{l}
\lp \mE_{{\mathcal U}_1} Z_{i_3}^{\m_1} \rp^p \rp   }
\mE_{{\mathcal U}_1} \frac{1}{Z_{i_3}^{1-\m_1}}  \sum_{i=1}^{l} (C_{i_3}^{(i_1)})^{s-1}
\sum_{i_2=1}^{l}\beta_{i_1}A_{i_3}^{(i_1,i_2)} \lp \bar{T}_2+ \bar{T}_1\rp.
\end{equation}
Moreover, analogously to \cite{Stojnicsflgscompyx23}'s (16)-(18) we first have
\begin{equation}\label{eq:genanal10e}
\frac{d\psi(t)}{d\p_1}  =       \frac{\mbox{sign}(s)}{2 \sqrt{n}} \sum_{i_1=1}^{l}  \sum_{i_2=1}^{l}
\beta_{i_1}\lp T_2^{\p}- T_1^{\p}\rp,
\end{equation}
where
\begin{eqnarray}\label{eq:genanal10f}
T_2^{\p} & = & \frac{\sqrt{1-t}}{\sqrt{\p_1}}\sum_{j=1}^{m} T_{2,1,j}^{\p} +\frac{\sqrt{t}\q_1}{\sqrt{\p_1\q_1}}\|\y^{(i_2)}\|_2T_{2,3}^{(\p,\q)} \nonumber\\
T_1^{\p} & = & \frac{\sqrt{1-t}}{\sqrt{\p_0-\p_1}}\sum_{j=1}^{m} T_{1,1,j}^{\p} +\frac{\sqrt{t}\q_1}{\sqrt{\p_0\q_0-\p_1\q_1}}\|\y^{(i_2)}\|_2T_{1,3}^{(\p,\q)},
\end{eqnarray}
and
\begin{eqnarray}\label{eq:genanal10g}
 T_{2,1,j}^{\p} & = &  \mE_{G,{\mathcal U}_2}\lp
\sum_{i_3=1}^{l} \frac{
\lp \mE_{{\mathcal U}_1} Z_{i_3}^{\m_1} \rp^{p-1}    }{  \lp
\sum_{i_3=1}^{l}
\lp \mE_{{\mathcal U}_1} Z_{i_3}^{\m_1} \rp^p \rp   }
\mE_{{\mathcal U}_1}\frac{(C_{i_3}^{(i_1)})^{s-1} A_{i_3}^{(i_1,i_2)} \y_j^{(i_2)}\u_j^{(2,2)}}{Z_{i_3}^{1-\m_1}} \rp \nonumber \\
 T_{2,3}^{(\p,\q)} & = &  \mE_{G,{\mathcal U}_2}\lp
 \sum_{i_3=1}^{l} \frac{
\lp \mE_{{\mathcal U}_1} Z_{i_3}^{\m_1} \rp^{p-1}    }{  \lp
\sum_{i_3=1}^{l}
\lp \mE_{{\mathcal U}_1} Z_{i_3}^{\m_1} \rp^p \rp   }
 \mE_{{\mathcal U}_1}\frac{(C_{i_3}^{(i_1)})^{s-1} A_{i_3}^{(i_1,i_2)} u^{(4,2)}}{Z_{i_3}^{1-\m_1}} \rp \nonumber \\
T_{1,1,j}^{\p} & = &  \mE_{G,{\mathcal U}_2} \lp
\sum_{i_3=1}^{l} \frac{
\lp \mE_{{\mathcal U}_1} Z_{i_3}^{\m_1} \rp^{p-1}    }{  \lp
\sum_{i_3=1}^{l}
\lp \mE_{{\mathcal U}_1} Z_{i_3}^{\m_1} \rp^p \rp   }
\mE_{{\mathcal U}_1}\frac{(C_{i_3}^{(i_1)})^{s-1} A_{i_3}^{(i_1,i_2)} \y_j^{(i_2)}\u_j^{(2,1)}}{Z_{i_3}^{1-\m_1}}\rp \nonumber \\
 T_{1,3}^{(\p,\q)} & = &  \mE_{G,{\mathcal U}_2}\lp
 \sum_{i_3=1}^{l} \frac{
\lp \mE_{{\mathcal U}_1} Z_{i_3}^{\m_1} \rp^{p-1}    }{  \lp
\sum_{i_3=1}^{l}
\lp \mE_{{\mathcal U}_1} Z_{i_3}^{\m_1} \rp^p \rp   }
 \mE_{{\mathcal U}_1}\frac{(C_{i_3}^{(i_1)})^{s-1} A_{i_3}^{(i_1,i_2)} u^{(4,1)}}{Z_{i_3}^{1-\m_1}}\rp,
\end{eqnarray}
and then analogously to \cite{Stojnicsflgscompyx23}'s (19-21)
\begin{equation}\label{eq:genanal10e1}
\frac{d\psi(t)}{d\q_1}  =       \frac{\mbox{sign}(s)}{2 \sqrt{n}} \sum_{i_1=1}^{l}  \sum_{i_2=1}^{l}
\beta_{i_1}\lp T_2^{\q}- T_1^{\q}\rp,
\end{equation}
where
\begin{eqnarray}\label{eq:genanal10f1}
T_2^{\q} & = & \frac{\sqrt{1-t}}{\sqrt{\q_1}}\|\y^{(i_2)}\|_2 T_{2,2}^{\q} +\frac{\sqrt{t}\p_1}{\sqrt{\p_1\q_1}}\|\y^{(i_2)}\|_2T_{2,3}^{(\p,\q)} \nonumber\\
T_1^{\q} & = & \frac{\sqrt{1-t}}{\sqrt{\q_0-\q_1}} \|\y^{(i_2)}\|_2T_{1,2}^{\q} +\frac{\sqrt{t}\p_1}{\sqrt{\p_0\q_0-\p_1\q_1}}\|\y^{(i_2)}\|_2T_{1,3}^{(\p,\q)},
\end{eqnarray}
and
\begin{eqnarray}\label{eq:genanal10g1}
 T_{2,2}^{\q} & = &  \mE_{G,{\mathcal U}_2}\lp
 \sum_{i_3=1}^{l} \frac{
\lp \mE_{{\mathcal U}_1} Z_{i_3}^{\m_1} \rp^{p-1}    }{  \lp
\sum_{i_3=1}^{l}
\lp \mE_{{\mathcal U}_1} Z_{i_3}^{\m_1} \rp^p \rp   }
 \mE_{{\mathcal U}_1}\frac{(C_{i_3}^{(i_1)})^{s-1} A_{i_3}^{(i_1,i_2)} \u^{(i_1,3,2)}}{Z_{i_3}^{1-\m_1}} \rp \nonumber \\
 T_{1,2}^{\q} & = &  \mE_{G,{\mathcal U}_2} \lp
 \sum_{i_3=1}^{l} \frac{
\lp \mE_{{\mathcal U}_1} Z_{i_3}^{\m_1} \rp^{p-1}    }{  \lp
\sum_{i_3=1}^{l}
\lp \mE_{{\mathcal U}_1} Z_{i_3}^{\m_1} \rp^p \rp   }
 \mE_{{\mathcal U}_1}\frac{(C_{i_3}^{(i_1)})^{s-1} A_{i_3}^{(i_1,i_2)} \u^{(i_1,3,1)}}{Z_{i_3}^{1-\m_1}}\rp.
 \end{eqnarray}
The above set of equations (\ref{eq:genanal10d})-(\ref{eq:genanal10g1}) is sufficient to determine $\p_1$ and $\q_1$ derivatives of $\psi(t)$. We handle separately each of the six key terms from (\ref{eq:genanal10g}) and  (\ref{eq:genanal10g1}). Also, since we to a large degree rely on  \cite{Stojnicgscompyx16,Stojnicnflgscompyx23,Stojnicsflgscompyx23,Stojnicnflldp25}  we parallel their flows of presentation as closely as possible.

%%%%%%%%%%%%%%%%%%%%%%%%%%%%%%%%%%%%%%%%%%%%%%%%%%%%%%%%%%%%%%%%%%%%%%%%
\subsection{Computing $\frac{d\psi(t)}{d\p_1}$ and $\frac{d\psi(t)}{d\q_1}$ -- first level  }
\label{sec:compderivative}
%%%%%%%%%%%%%%%%%%%%%%%%%%%%%%%%%%%%%%%%%%%%%%%%%%%%%%%%%%%%%%%%%%%%%%%%

As was the case in \cite{Stojnicnflgscompyx23,Stojnicsflgscompyx23,Stojnicnflldp25}, we carefully select the order in which the terms appearing in (\ref{eq:genanal10g}) are handled. In particular, we split the six terms into two groups of three. We first handle the  $T_1$--group ($T_{1,1,j}^{\p}$, $T_{1,2}^{\q}$, and $T_{1,3}^{(\p,\q)}$) and then $T_2$--group ($T_{2,1,j}^{\p}$,$ T_{2,2}^{\q}$, and $T_{2,3}^{(\p,\q)}$).

%%%%%%%%%%%%%%%%%%%%%%%%%%%%%%%%%%%%%%%%%%%%%%%%%%%%%%%%%%%%%%%%%%%%%%%%
\subsubsection{$T_1$--group -- first level}
\label{sec:handlT1}
%%%%%%%%%%%%%%%%%%%%%%%%%%%%%%%%%%%%%%%%%%%%%%%%%%%%%%%%%%%%%%%%%%%%%%%%

As stated above, each of the three $T_1$--group terms is handled separately.

%%%%%%%%%%%%%%%%%%%%%%%%%%%%%%%%%%%%%%%%%%%%%%%%%%%%%%%%%%%%%%%%%%%%%%%%
\underline{\textbf{\emph{Determining}} $T_{1,1,j}^{\p}$}
\label{sec:hand1T11}
%%%%%%%%%%%%%%%%%%%%%%%%%%%%%%%%%%%%%%%%%%%%%%%%%%%%%%%%%%%%%%%%%%%%%%%%

Comparing to  \cite{Stojnicnflldp25}'s Section  \ref{sec:hand1T11}, one notes that appearance of the scaling factors $b_1$ and $b_2$ is the only difference. Keeping in mind that this correspondingly affects the rescaling of $\u_j^{(2,1)}$ variances, we apply Gaussian integration by parts and write analogously to \cite{Stojnicnflldp25}'s (20)
 \begin{eqnarray}\label{eq:liftgenAanal19}
T_{1,1,j}^{\p} & = & \mE_{G,{\mathcal U}_2}\lp
 \sum_{i_3=1}^{l} \frac{
\lp \mE_{{\mathcal U}_1} Z_{i_3}^{\m_1} \rp^{p-1}    }{  \lp
\sum_{i_3=1}^{l}
\lp \mE_{{\mathcal U}_1} Z_{i_3}^{\m_1} \rp^p \rp   }
\mE_{{\mathcal U}_1}  \frac{(C_{i_3}^{(i_1)})^{s-1} A_{i_3}^{(i_1,i_2)}\y_j^{(i_2)} \u_j^{(2,1)} }{Z_{i_3}^{1-\m_1}}\rp \nonumber \\
& = & \mE_{G,{\mathcal U}_2} \lp
 \sum_{i_3=1}^{l} \frac{
\lp \mE_{{\mathcal U}_1} Z_{i_3}^{\m_1} \rp^{p-1}    }{  \lp
\sum_{i_3=1}^{l}
\lp \mE_{{\mathcal U}_1} Z_{i_3}^{\m_1} \rp^p \rp   }
\mE_{{\mathcal U}_1}\lp \mE_{{\mathcal U}_1}(\u_j^{(2,1)}\u_j^{(2,1)}) \frac{d}{d\u_j^{(2,1)}}\lp\frac{(C_{i_3}^{(i_1)})^{s-1} A_{i_3}^{(i_1,i_2)}\y_j^{(i_2)}}{Z_{i_3}^{1-\m_1}}\rp\rp\rp \nonumber \\
& = & \mE_{G,{\mathcal U}_2}\lp
 \sum_{i_3=1}^{l} \frac{
\lp \mE_{{\mathcal U}_1} Z_{i_3}^{\m_1} \rp^{p-1}    }{  \lp
\sum_{i_3=1}^{l}
\lp \mE_{{\mathcal U}_1} Z_{i_3}^{\m_1} \rp^p \rp   }
\mE_{{\mathcal U}_1}(\u_j^{(2,1)}\u_j^{(2,1)})\mE_{{\mathcal U}_1}\lp  \frac{d}{d\u_j^{(2,1)}}\lp\frac{(C_{i_3}^{(i_1)})^{s-1} A_{i_3}^{(i_1,i_2)}\y_j^{(i_2)}}{Z_{i_3}^{1-\m_1}}\rp\rp\rp \nonumber \\
& = & \mE_{G,{\mathcal U}_2} \lp
 \sum_{i_3=1}^{l} \frac{
\lp \mE_{{\mathcal U}_1} Z_{i_3}^{\m_1} \rp^{p-1}    }{  \lp
\sum_{i_3=1}^{l}
\lp \mE_{{\mathcal U}_1} Z_{i_3}^{\m_1} \rp^p \rp   }
\mE_{{\mathcal U}_1}\lp  \frac{d}{d\u_j^{(2,1)}}\lp\frac{(C_{i_3}^{(i_1)})^{s-1} A_{i_3}^{(i_1,i_2)}\y_j^{(i_2)}}{Z_{i_3}^{1-\m_1}}\rp\rp\rp.
\end{eqnarray}
One now notes that the inner expectation on the righthand side of the last equality structurally matches the corresponding one from the first part of \cite{Stojnicnflldp25}'s Section \ref{sec:hand1T11} with  $\u^{(2,1)}$ scaled by $\sqrt{\p_0-\p_1}$. Following \cite{Stojnicnflldp25} we can then write
\begin{eqnarray}\label{eq:liftgenAanal19a}
T_{1,1,j}^{\p} & = &    \sqrt{\p_0-\p_1}\mE_{G,{\mathcal U}_2} \lp
 \sum_{i_3=1}^{l} \frac{
\lp \mE_{{\mathcal U}_1} Z_{i_3}^{\m_1} \rp^{p-1}    }{  \lp
\sum_{i_3=1}^{l}
\lp \mE_{{\mathcal U}_1} Z_{i_3}^{\m_1} \rp^p \rp   }
\lp \Theta_1+\Theta_2 \rp\rp,
\end{eqnarray}
with $\Theta_1$ and $\Theta_2$ as in \cite{Stojnicnflldp25}'s (22)
{\small\begin{eqnarray}\label{eq:liftgenAanal19c}
\Theta_1 &  = &  \mE_{{\mathcal U}_1} \Bigg( \Bigg. \frac{\y_j^{(i_2)} \lp (C_{i_3}^{(i_1)})^{s-1}\beta_{i_1}A_{i_3}^{(i_1,i_2)}\y_j^{(i_2)}\sqrt{1-t} +A_{i_3}^{(i_1,i_2)}(s-1)(C_{i_3}^{(i_1)})^{s-2}\beta_{i_1}\sum_{p_2=1}^{l}A_{i_3}^{(i_1,p_2)}\y_j^{(p_2)}\sqrt{1-t}\rp}{Z_{i_3}^{1-\m_1}}\Bigg. \Bigg)\Bigg. \Bigg) \nonumber \\
\Theta_2 & = & -(1-\m_1)\mE_{{\mathcal U}_1} \lp\sum_{p_1=1}^{l}
\frac{(C_{i_3}^{(i_1)})^{s-1} A_{i_3}^{(i_1,i_2)}\y_j^{(i_2)}}{Z_{i_3}^{2-\m_1}}
s  (C_{i_3}^{(p_1)})^{s-1}\sum_{p_2=1}^{l}\beta_{p_1}A_{i_3}^{(p_1,p_2)}\y_j^{(p_2)}\sqrt{1-t}\rp\Bigg.\Bigg).
\end{eqnarray}}
The above then allows for the following observation
\begin{eqnarray}\label{eq:liftgenAanal19d}
\sum_{i_1=1}^{l}\sum_{i_2=1}^{l}\sum_{j=1}^{m} \lp
\sum_{i_3=1}^{l} \frac{
\lp \mE_{{\mathcal U}_1} Z_{i_3}^{\m_1} \rp^{p-1}    }{ \lp
\sum_{i_3=1}^{l}
\lp \mE_{{\mathcal U}_1} Z_{i_3}^{\m_1} \rp^p \rp   }
 \frac{\beta_{i_1}\Theta_1}{\sqrt{1-t}}\rp
 \hspace{-.02in} &  = & \hspace{-.0in} \Bigg ( \Bigg.
 \sum_{i_3=1}^{l} \frac{
\lp \mE_{{\mathcal U}_1} Z_{i_3}^{\m_1} \rp^{p}    }{ \lp
\sum_{i_3=1}^{l}
\lp \mE_{{\mathcal U}_1} Z_{i_3}^{\m_1} \rp^p \rp   }
 \mE_{{\mathcal U}_1}\frac{Z_{i_3}^{\m_1}}{\mE_{{\mathcal U}_1} Z_{i_3}^{\m_1}}
\nonumber  \\
& & \times
 \sum_{i_1=1}^{l}\frac{(C_{i_3}^{(i_1)})^s}{Z_{i_3}}\sum_{i_2=1}^{l}\frac{A_{i_3}^{(i_1,i_2)}}{C_{i_3}^{(i_1)}}\beta_{i_1}^2\|\y^{(i_2)}\|_2^2
 \Bigg ) \Bigg.
  \nonumber\\
& & +  \Bigg ( \Bigg.
 \sum_{i_3=1}^{l} \frac{
\lp \mE_{{\mathcal U}_1} Z_{i_3}^{\m_1} \rp^{p}    }{ \lp
\sum_{i_3=1}^{l}
\lp \mE_{{\mathcal U}_1} Z_{i_3}^{\m_1} \rp^p \rp   }
 \mE_{{\mathcal U}_1}\frac{Z_{i_3}^{\m_1}}{\mE_{{\mathcal U}_1} Z_{i_3}^{\m_1}}
 \nonumber \\
 & &
 \times
 \sum_{i_1=1}^{l}\frac{(s-1)(C_{i_3}^{(i_1)})^s}{Z_{i_3}}\sum_{i_2=1}^{l}\sum_{p_2=1}^{l}\frac{A_{i_3}^{(i_1,i_2)}A_{i_3}^{(i_1,p_2)}}{(C_{i_3}^{(i_1)})^2}
 \nonumber \\
 & &
 \times
 \beta_{i_1}^2(\y^{(p_2)})^T\y^{(i_2)}
 \Bigg ) \Bigg. .\nonumber \\
 \end{eqnarray}
 After introducing operator
\begin{eqnarray}\label{eq:genAanal19e}
 \Phi_{{\mathcal U}_1}^{(i_3)} & \triangleq &  \mE_{{\mathcal U}_{1}} \frac{Z_{i_3}^{\m_1}}{\mE_{{\mathcal U}_{1}}Z_{i_3}^{\m_1}}  \triangleq  \mE_{{\mathcal U}_{1}} \lp \frac{Z_{i_3}^{\m_1}}{\mE_{{\mathcal U}_{1}}Z_{i_3}^{\m_1}}\lp \cdot \rp\rp,
 \end{eqnarray}
and establishing measures
\begin{eqnarray}\label{eq:genAanal19e1}
  \gamma_{00}(i_3) & = &
  \frac{
\lp \mE_{{\mathcal U}_1} Z_{i_3}^{\m_1} \rp^{p}    }{ \lp
\sum_{i_3=1}^{l}
\lp \mE_{{\mathcal U}_1} Z_{i_3}^{\m_1} \rp^p \rp   }
\nonumber \\
  \gamma_0(i_1,i_2;i_3) & = &
\frac{(C_{i_3}^{(i_1)})^{s}}{Z_{i_3}}  \frac{A_{i_3}^{(i_1,i_2)}}{C_{i_3}^{(i_1)}} \nonumber \\
\gamma_{01}^{(1)}  & = &  \gamma_{00}(i_3)\Phi_{{\mathcal U}_1}^{(i_3)} (\gamma_0(i_1,i_2;i_3)) \nonumber \\
\gamma_{02}^{(1)}  & = &  \gamma_{00}(i_3)\Phi_{{\mathcal U}_1}^{(i_3)} (\gamma_0(i_1,i_2;i_3)\times \gamma_0(i_1,p_2;i_3)) \nonumber \\
\gamma_{1}^{(1)}   & = &  \gamma_{00}(i_3)\Phi_{{\mathcal U}_1}^{(i_3)}  \lp \gamma_0(i_1,i_2;i_3)\times \gamma_0(p_1,p_2;i_3) \rp \nonumber \\.
\gamma_{21}^{(1)}   & = &  \gamma_{00}(i_3)\Phi_{{\mathcal U}_1}^{(i_3)}   \gamma_0(i_1,i_2;i_3)  \times  \gamma_{00}(p_3) \Phi_{{\mathcal U}_1}^{(p_3)}  \gamma_0(p_1,p_2;p_3)
\nonumber  \\
\gamma_{22}^{(1)}   & = &  \gamma_{00}(i_3)\lp \Phi_{{\mathcal U}_1}^{(i_3)}   \gamma_0(i_1,i_2;i_3)  \times  \Phi_{{\mathcal U}_1}^{(i_3)}  \gamma_0(p_1,p_2;i_3) \rp.
\end{eqnarray}
one further writes
\begin{eqnarray}\label{eq:liftgenAanal19g}
\sum_{i_1=1}^{l}\sum_{i_2=1}^{l}\sum_{j=1}^{m} \lp
\sum_{i_3=1}^{l} \frac{
\lp \mE_{{\mathcal U}_1} Z_{i_3}^{\m_1} \rp^{p-1}    }{ \lp
\sum_{i_3=1}^{l}
\lp \mE_{{\mathcal U}_1} Z_{i_3}^{\m_1} \rp^p \rp   }
 \frac{\beta_{i_1}\Theta_1}{\sqrt{1-t}}\rp
&  = & \beta^2 \Bigg ( \Bigg. \langle \|\x^{(i_1)}\|_2^2\|\y^{(i_2)}\|_2^2\rangle_{\gamma_{01}^{(1)}}
\nonumber \\
& &
+  (s-1) \langle \|\x^{(i_1)}\|_2^2(\y^{(p_2)})^T\y^{(i_2)}\rangle_{\gamma_{02}^{(1)}} \Bigg ) \Bigg. .
 \end{eqnarray}
and analogously relying on (\ref{eq:liftgenAanal19c})
\begin{eqnarray}\label{eq:liftgenAanal19h}
\sum_{i_1=1}^{l}\sum_{i_2=1}^{l}\sum_{j=1}^{m} \lp
\sum_{i_3=1}^{l} \frac{
\lp \mE_{{\mathcal U}_1} Z_{i_3}^{\m_1} \rp^{p-1}    }{ \lp
\sum_{i_3=1}^{l}
\lp \mE_{{\mathcal U}_1} Z_{i_3}^{\m_1} \rp^p \rp   }
\frac{\beta_{i_1}\Theta_2}{\sqrt{1-t}}\rp
 \hspace{-.07in}
 & = &  \hspace{-.05in} -s(1-\m_1) \mE_{G,{\mathcal U}_2} \Bigg( \Bigg.
 \sum_{i_3=1}^{l} \frac{
\lp \mE_{{\mathcal U}_1} Z_{i_3}^{\m_1} \rp^{p}    }{ \lp
\sum_{i_3=1}^{l}
\lp \mE_{{\mathcal U}_1} Z_{i_3}^{\m_1} \rp^p \rp   }
\nonumber \\
& &
 \hspace{-.05in} \times
 \frac{Z_{i_3}^{\m_1}}{\mE_{{\mathcal U}_1} Z_{i_3}^{\m_1}} \sum_{i_1=1}^{l}\frac{(C_{i_3}^{(i_1)})^s}{Z_{i_3}}\sum_{i_2=1}^{l}
\frac{A_{i_3}^{(i_1,i_2)}}{C_{i_3}^{(i_1)}} \nonumber \\
& &
 \hspace{-.05in}
\times
 \sum_{p_1=1}^{l} \frac{(C_{i_3}^{(p_1)})^s}{Z_{i_3}}\sum_{p_2=1}^{l}\frac{A_{i_3}^{(p_1,p_2)}}{C_{i_3}^{(p_1)}} \beta_{i_1}\beta_{p_1}(\y^{(p_2)})^T\y^{(i_2)} \Bigg.\Bigg)\nonumber \\
& =&  \hspace{-.05in} -s\beta^2(1-\m_1) \mE_{G,{\mathcal U}_2} \langle \|\x^{(i_1)}\|_2\|\x^{(p_1)}\|_2(\y^{(p_2)})^T\y^{(i_2)} \rangle_{\gamma_{1}^{(1)}}.
\nonumber \\
\end{eqnarray}
Combining further  (\ref{eq:liftgenAanal19a}), (\ref{eq:liftgenAanal19g}), and (\ref{eq:liftgenAanal19h}) we also obtain
\begin{eqnarray}\label{eq:liftgenAanal19i}
\sum_{i_1=1}^{l}\sum_{i_2=1}^{l}\sum_{j=1}^{m} \beta_{i_1}\frac{\sqrt{1-t}}{\sqrt{\p_0-\p_1}}T_{1,1,j}^{\p}
& = &
 (1-t) \mE_{G,{\mathcal U}_2} \lp
\sum_{i_3=1}^{l} \frac{
\lp \mE_{{\mathcal U}_1} Z_{i_3}^{\m_1} \rp^{p-1}    }{ \lp
\sum_{i_3=1}^{l}
\lp \mE_{{\mathcal U}_1} Z_{i_3}^{\m_1} \rp^p \rp   }
\lp \frac{\beta_{i_1}\Theta_1}{\sqrt{1-t}}+\frac{\beta_{i_1}\Theta_2}{\sqrt{1-t}} \rp\rp\nonumber \\
& = &  (1-t) \beta^2  \Bigg ( \Bigg .
\mE_{G,{\mathcal U}_2} \langle \|\x^{(i_1)}\|_2^2\|\y^{(i_2)}\|_2^2\rangle_{\gamma_{01}^{(1)}}
 \nonumber \\
 & &
 +  (s-1)\mE_{G,{\mathcal U}_2}\langle \|\x^{(i_1)}\|_2^2(\y^{(p_2)})^T\y^{(i_2)}\rangle_{\gamma_{02}^{(1)}} \Bigg ) \Bigg . \nonumber \\
& & -  (1-t) s\beta^2(1-\m_1)\langle \|\x^{(i_1)}\|_2\|\x^{(p_1)}\|_2(\y^{(p_2)})^T\y^{(i_2)} \rangle_{\gamma_{1}^{(1)}}.
 \end{eqnarray}
Both the  $\gamma$ measures from (\ref{eq:genAanal19e1}) and the $\Phi(\cdot)$ operators from (\ref{eq:genAanal19e}) are  functions of $t$ which implies that all $\gamma$ dependent functions depend  on $t$ as well. To lighten notation, we skip explicitly stating  $t$-dependence.

%%%%%%%%%%%%%%%%%%%%%%%%%%%%%%%%%%%%%%%%%%%%%%%%%%%%%%%%%%%%%%%%%%%%%%%%
\underline{\textbf{\emph{Determining}} $T_{1,2}^{\q}$}
\label{sec:hand1T12}
%%%%%%%%%%%%%%%%%%%%%%%%%%%%%%%%%%%%%%%%%%%%%%%%%%%%%%%%%%%%%%%%%%%%%%%%

Relying again on the Gaussian integration by parts one finds
\begin{eqnarray}\label{eq:liftgenBanal20}
T_{1,2}^{\q} & = & \mE_{G,{\mathcal U}_2} \lp
 \sum_{i_3=1}^{l} \frac{
\lp \mE_{{\mathcal U}_1} Z_{i_3}^{\m_1} \rp^{p-1}    }{  \lp
\sum_{i_3=1}^{l}
\lp \mE_{{\mathcal U}_1} Z_{i_3}^{\m_1} \rp^p \rp   }
\mE_{{\mathcal U}_1} \frac{(C_{i_3}^{(i_1)})^{s-1} A_{i_3}^{(i_1,i_2)}\u^{(i_1,3,1)}}{Z_{i_3}^{1-\m_1}}\rp \nonumber \\
& = & \mE_{G,{\mathcal U}_2} \lp
 \sum_{i_3=1}^{l} \frac{
\lp \mE_{{\mathcal U}_1} Z_{i_3}^{\m_1} \rp^{p-1}    }{  \lp
\sum_{i_3=1}^{l}
\lp \mE_{{\mathcal U}_1} Z_{i_3}^{\m_1} \rp^p \rp   }
\mE_{{\mathcal U}_1} \sum_{p_1=1}^{l}\mE_{{\mathcal U}_1}(\u^{(i_1,3,1)}\u^{(p_1,3,1)}) \frac{d}{d\u^{(p_1,3,1)}}\lp\frac{(C_{i_3}^{(i_1)})^{s-1} A_{i_3}^{(i_1,i_2)}}{Z_{i_3}^{1-\m_1}}\rp\rp \nonumber \\
& = & \frac{\sqrt{\q_0-\q_1}}{\q_0-\q_1}\mE_{G,{\mathcal U}_2} \Bigg(\Bigg.
 \sum_{i_3=1}^{l} \frac{
\lp \mE_{{\mathcal U}_1} Z_{i_3}^{\m_1} \rp^{p-1}    }{  \lp
\sum_{i_3=1}^{l}
\lp \mE_{{\mathcal U}_1} Z_{i_3}^{\m_1} \rp^p \rp   }
\mE_{{\mathcal U}_1} \sum_{p_1=1}^{l}\frac{(\sqrt{\q_0-\q_1}\x^{(i_1)})^T\sqrt{\q_0-\q_1}\x^{(p_1)}}{\|\x^{(i_1)}\|_2\|\x^{(p_1)}\|_2} \nonumber \\
& & \times
\frac{d}{d\lp \sqrt{\q_0-\q_1} \u^{(p_1,3,1)}\rp}\lp\frac{(C_{i_3}^{(i_1)})^{s-1} A_{i_3}^{(i_1,i_2)}}{Z_{i_3}^{1-\m_1}}\rp\Bigg.\Bigg) \nonumber\\
& = & \frac{\sqrt{\q_0-\q_1}}{\q_0-\q_1} T_{1,2},
\end{eqnarray}
where  from  \cite{Stojnicnflgscompyx23}'s (32)
\begin{eqnarray}\label{eq:liftgenBanal20a0}
\sum_{i_1=1}^{l}\sum_{i_2=1}^{l} \beta_{i_1}\|\y^{(i_2)}\|_2 T_{1,2} & = & \sqrt{1-t}(\q_0-\q_1)\beta^2
\Bigg( \Bigg.\mE_{G,{\mathcal U}_2}\langle \|\x^{(i_1)}\|_2^2\|\y^{(i_2)}\|_2^2\rangle_{\gamma_{01}^{(1)}} \nonumber \\
& & +   (s-1)\mE_{G,{\mathcal U}_2}\langle \|\x^{(i_1)}\|_2^2 \|\y^{(i_2)}\|_2\|\y^{(p_2)}\|_2\rangle_{\gamma_{02}^{(1)}}\Bigg.\Bigg)  \nonumber \\
& & - \sqrt{1-t}(\q_0-\q_1)s\beta^2(1-\m_1)\mE_{G,{\mathcal U}_2}\langle (\x^{(p_1)})^T\x^{(i_1)}\|\y^{(i_2)}\|_2\|\y^{(p_2)}\|_2 \rangle_{\gamma_{1}^{(1)}}.\nonumber \\
\end{eqnarray}
A combination of (\ref{eq:liftgenBanal20}) and (\ref{eq:liftgenBanal20a0}) then immediately gives
\begin{eqnarray}\label{eq:liftgenBanal20b}
\sum_{i_1=1}^{l}\sum_{i_2=1}^{l} \beta_{i_1}\|\y^{(i_2)}\|_2 \frac{\sqrt{1-t}}{\sqrt{\q_0-\q_1}}T_{1,2}^{\q} & = & (1-t)\beta^2
\Bigg( \Bigg.\mE_{G,{\mathcal U}_2}\langle \|\x^{(i_1)}\|_2^2\|\y^{(i_2)}\|_2^2\rangle_{\gamma_{01}^{(1)}} \nonumber \\
& & +   (s-1)\mE_{G,{\mathcal U}_2}\langle \|\x^{(i_1)}\|_2^2 \|\y^{(i_2)}\|_2\|\y^{(p_2)}\|_2\rangle_{\gamma_{02}^{(1)}}\Bigg.\Bigg)  \nonumber \\
& & - (1-t)s\beta^2(1-\m_1)\mE_{G,{\mathcal U}_2}\langle (\x^{(p_1)})^T\x^{(i_1)}\|\y^{(i_2)}\|_2\|\y^{(p_2)}\|_2 \rangle_{\gamma_{1}^{(1)}}.\nonumber \\
\end{eqnarray}

%%%%%%%%%%%%%%%%%%%%%%%%%%%%%%%%%%%%%%%%%%%%%%%%%%%%%%%%%%%%%%%%%%%%%%%%
\underline{\textbf{\emph{Determining}} $T_{1,3}^{(\p,\q)}$}
\label{sec:hand1T13}
%%%%%%%%%%%%%%%%%%%%%%%%%%%%%%%%%%%%%%%%%%%%%%%%%%%%%%%%%%%%%%%%%%%%%%%%

Applying Gaussian integration by parts, we have the following analogue to \cite{Stojnicnflldp25}'s (33)
\begin{eqnarray}\label{eq:liftgenCanal21}
T_{1,3}^{(\p,\q)} \hspace{-.1in} & = & \mE_{G,{\mathcal U}_2} \lp
 \sum_{i_3=1}^{l} \frac{
\lp \mE_{{\mathcal U}_1} Z_{i_3}^{\m_1} \rp^{p-1}    }{  \lp
\sum_{i_3=1}^{l}
\lp \mE_{{\mathcal U}_1} Z_{i_3}^{\m_1} \rp^p \rp   }
\mE_{{\mathcal U}_1}  \frac{(C_{i_3}^{(i_1)})^{s-1} A_{i_3}^{(i_1,i_2)}u^{(4,1)}}{Z_{i_3}^{1-\m_1}} \rp \nonumber \\
& = & \mE_{G,{\mathcal U}_2} \lp
 \sum_{i_3=1}^{l} \frac{
\lp \mE_{{\mathcal U}_1} Z_{i_3}^{\m_1} \rp^{p-1}    }{  \lp
\sum_{i_3=1}^{l}
\lp \mE_{{\mathcal U}_1} Z_{i_3}^{\m_1} \rp^p \rp   }
\mE_{{\mathcal U}_1} \lp\mE_{{\mathcal U}_1} (u^{(4,1)}u^{(4,1)})\lp\frac{d}{du^{(4,1)}} \lp\frac{(C_{i_3}^{(i_1)})^{s-1} A_{i_3}^{(i_1,i_2)}u^{(4,1)}}{Z_{i_3}^{1-\m_1}}\rp \rp\rp\rp \nonumber \\
 & = & \frac{\sqrt{\p_0\q_0-\p_1\q_1}}{\p_0\q_0-\p_1\q_1}\mE_{G,{\mathcal U}_2} \Bigg( \Bigg.
  \sum_{i_3=1}^{l} \frac{
\lp \mE_{{\mathcal U}_1} Z_{i_3}^{\m_1} \rp^{p-1}    }{  \lp
\sum_{i_3=1}^{l}
\lp \mE_{{\mathcal U}_1} Z_{i_3}^{\m_1} \rp^p \rp   }
 \mE_{{\mathcal U}_1} (\sqrt{\p_0\q_0-\p_1\q_1}u^{(4,1)}\sqrt{\p_0\q_0-\p_1\q_1}u^{(4,1)}) \nonumber \\
 & & \times \mE_{{\mathcal U}_1} \lp\frac{d}{d \lp \sqrt{\p_0\q_0-\p_1\q_1} u^{(4,1)} \rp} \lp\frac{(C_{i_3}^{(i_1)})^{s-1} A_{i_3}^{(i_1,i_2)}u^{(4,1)}}{Z_{i_3}^{1-\m_1}}\rp\rp\Bigg.\Bigg) \nonumber \\
 & = & \frac{1}{\sqrt{\p_0\q_0-\p_1\q_1}}T_{1,3},
\end{eqnarray}
where   \cite{Stojnicnflldp25}'s (35) gives
\begin{eqnarray}\label{eq:liftgenCanal21b01}
\sum_{i_1=1}^{l}\sum_{i_2=1}^{l} \beta_{i_1}\|\y^{(i_2)}\|_2 T_{1,3} & = & \sqrt{t}(\p_0\q_0-\p_1\q_1)\beta^2 \Bigg( \Bigg. \mE_{G,{\mathcal U}_2}\langle \|\x^{(i_1)}\|_2^2\|\y^{(i_2)}\|_2^2\rangle_{\gamma_{01}^{(1)}} \nonumber \\
& & +   (s-1)\mE_{G,{\mathcal U}_2}\langle \|\x^{(i_1)}\|_2^2 \|\y^{(i_2)}\|_2\|\y^{(p_2)}\|_2\rangle_{\gamma_{02}^{(1)}}\Bigg.\Bigg) \nonumber \\
& & - \sqrt{t}(\p_0\q_0-\p_1\q_1)s\beta^2(1-\m_1)\mE_{G,{\mathcal U}_2}\langle \|\x^{(i_1)}\|_2\|\x^{(p_`)}\|_2\|\y^{(i_2)}\|_2\|\y^{(p_2)}\|_2 \rangle_{\gamma_{1}^{(1)}}. \nonumber \\
\end{eqnarray}
Combining (\ref{eq:liftgenCanal21}) and (\ref{eq:liftgenCanal21b01}) one then finds analogously to (\ref{eq:liftgenBanal20b})
\begin{eqnarray}\label{eq:liftgenCanal21b}
\sum_{i_1=1}^{l}\sum_{i_2=1}^{l} \beta_{i_1}\|\y^{(i_2)}\|_2 \frac{\sqrt{t}}{\sqrt{\p_0\q_0-\p_1\q_1}}T_{1,3}^{(\p,\q)} & = & t\beta^2 \Bigg( \Bigg. \mE_{G,{\mathcal U}_2}\langle \|\x^{(i_1)}\|_2^2\|\y^{(i_2)}\|_2^2\rangle_{\gamma_{01}^{(1)}} \nonumber \\
& & +   (s-1)\mE_{G,{\mathcal U}_2}\langle \|\x^{(i_1)}\|_2^2 \|\y^{(i_2)}\|_2\|\y^{(p_2)}\|_2\rangle_{\gamma_{02}^{(1)}}\Bigg.\Bigg) \nonumber \\
& & - ts\beta^2(1-\m_1)\mE_{G,{\mathcal U}_2}\langle \|\x^{(i_1)}\|_2\|\x^{(p_`)}\|_2\|\y^{(i_2)}\|_2\|\y^{(p_2)}\|_2 \rangle_{\gamma_{1}^{(1)}}. \nonumber \\
\end{eqnarray}

%%%%%%%%%%%%%%%%%%%%%%%%%%%%%%%%%%%%%%%%%%%%%%%%%%%%%%%%%%%%%%%%%%%%%%%%
%%%%%%%%%%%%%%%%%%%%%%%%%%%%%%%%%%%%%%%%%%%%%%%%%%%%%%%%%%%%%%%%%%%%%%%%
%%%%%%%%%%%%%%%%%%%%%%%%%%%%%%%%%%%%%%%%%%%%%%%%%%%%%%%%%%%%%%%%%%%%%%%%
%%%%%%%%%%%%%%%%%%%%%%%%%%%%%%%%%%%%%%%%%%%%%%%%%%%%%%%%%%%%%%%%%%%%%%%%
\subsubsection{$T_2$--group -- first level}
\label{sec:handlT2}
%%%%%%%%%%%%%%%%%%%%%%%%%%%%%%%%%%%%%%%%%%%%%%%%%%%%%%%%%%%%%%%%%%%%%%%%
%%%%%%%%%%%%%%%%%%%%%%%%%%%%%%%%%%%%%%%%%%%%%%%%%%%%%%%%%%%%%%%%%%%%%%%%
%%%%%%%%%%%%%%%%%%%%%%%%%%%%%%%%%%%%%%%%%%%%%%%%%%%%%%%%%%%%%%%%%%%%%%%%
%%%%%%%%%%%%%%%%%%%%%%%%%%%%%%%%%%%%%%%%%%%%%%%%%%%%%%%%%%%%%%%%%%%%%%%%

As in the previous subsection, each of the  $T_2$-group's three terms is handled separately.

%%%%%%%%%%%%%%%%%%%%%%%%%%%%%%%%%%%%%%%%%%%%%%%%%%%%%%%%%%%%%%%%%%%%%%%%
\underline{\textbf{\emph{Determining}} $T_{2,1,j}^{\p}$}
\label{sec:hand1T21}
%%%%%%%%%%%%%%%%%%%%%%%%%%%%%%%%%%%%%%%%%%%%%%%%%%%%%%%%%%%%%%%%%%%%%%%%

Gaussian integration by parts gives
\begin{eqnarray}\label{eq:genDanal19}
T_{2,1,j}^{\p}& = &  \mE_{G,{\mathcal U}_2}\lp
 \sum_{i_3=1}^{l}
 \frac{
\lp \mE_{{\mathcal U}_1} Z_{i_3}^{\m_1} \rp^{p-1}    }{  \lp
\sum_{i_3=1}^{l}
\lp \mE_{{\mathcal U}_1} Z_{i_3}^{\m_1} \rp^p \rp   }
\mE_{{\mathcal U}_1}\frac{(C_{i_3}^{(i_1)})^{s-1} A_{i_3}^{(i_1,i_2)} \y_j^{(i_2)}\u_j^{(2,2)}}{Z_{i_3}^{1-\m_1}} \rp \nonumber \\
 & = &
\mE_{G,{\mathcal U}_1}\lp
\sum_{i_3=1}^{l}
\mE_{{\mathcal U}_2}\lp\mE_{{\mathcal U}_2} (\u_j^{(2,2)}\u_j^{(2,2)})\frac{d}{d\u_j^{(2,2)}}\lp \frac{(C_{i_3}^{(i_1)})^{s-1} A_{i_3}^{(i_1,i_2)}\y_j^{(i_2)}}{Z_{i_3}^{1-\m_1} }
 \frac{
\lp \mE_{{\mathcal U}_1} Z_{i_3}^{\m_1} \rp^{p-1}    }{  \lp
\sum_{i_3=1}^{l}
\lp \mE_{{\mathcal U}_1} Z_{i_3}^{\m_1} \rp^p \rp   }
\rp\rp\rp \nonumber \\
& = &
\frac{\sqrt{\p_1}}{\p_1}
\Bigg( \Bigg.
\mE_{G,{\mathcal U}_2,{\mathcal U}_1}  \Bigg( \Bigg.
\sum_{i_3=1}^{l}
 \frac{
\lp \mE_{{\mathcal U}_1} Z_{i_3}^{\m_1} \rp^{p-1}    }{  \lp
\sum_{i_3=1}^{l}
\lp \mE_{{\mathcal U}_1} Z_{i_3}^{\m_1} \rp^p \rp   }
 \mE_{{\mathcal U}_2} (\sqrt{\p_1}\u_j^{(2,2)}\sqrt{\p_1}\u_j^{(2,2)})
 \nonumber \\
 & & \times
\frac{d}{d\lp \sqrt{\p_1}\u_j^{(2,2)}\rp}\lp \frac{(C_{i_3}^{(i_1)})^{s-1} A_{i_3}^{(i_1,i_2)}\y_j^{(i_2)}}{Z_{i_3}^{1-\m_1}}\rp
\Bigg .\Bigg )
\nonumber \\
& & + \mE_{G,{\mathcal U}_2,{\mathcal U}_1}   \Bigg. \Bigg(
\sum_{i_3=1}^{l}
\lp \mE_{{\mathcal U}_1} Z_{i_3}^{\m_1} \rp^{p-1}
\frac{\mE_{{\mathcal U}_2} (\sqrt{\p_1}\u_j^{(2,2)}\sqrt{\p_1}\u_j^{(2,2)})\lp(C_{i_3}^{(i_1)})^{s-1} A_{i_3}^{(i_1,i_2)}\y_j^{(i_2)} \rp}{Z_{i_3}^{1-\m_1}}
\nonumber \\
& & \times
\frac{d}{d\lp\sqrt{\p_1} \u_j^{(2,2)}\rp}\lp \frac{1}{   \lp
\sum_{i_3=1}^{l}
\lp \mE_{{\mathcal U}_1} Z_{i_3}^{\m_1} \rp^p \rp   }   \rp
  \Bigg. \Bigg)   \nonumber \\
& & + \mE_{G,{\mathcal U}_2,{\mathcal U}_1}   \Bigg. \Bigg(
\sum_{i_3=1}^{l}
 \frac{
1   }{  \lp
\sum_{i_3=1}^{l}
\lp \mE_{{\mathcal U}_1} Z_{i_3}^{\m_1} \rp^p \rp   }
\frac{\mE_{{\mathcal U}_2} (\sqrt{\p_1}\u_j^{(2,2)}\sqrt{\p_1}\u_j^{(2,2)})\lp(C_{i_3}^{(i_1)})^{s-1} A_{i_3}^{(i_1,i_2)}\y_j^{(i_2)} \rp}{Z_{i_3}^{1-\m_1}}
\nonumber \\
& & \times
\frac{d}{d\lp\sqrt{\p_1} \u_j^{(2,2)}\rp}\lp \mE_{{\mathcal U}_1} Z_{i_3}^{\m_1}  \rp^{p-1}
  \Bigg. \Bigg)   \Bigg. \Bigg) \nonumber \\
& = & \frac{\sqrt{\p_1}}{\p_1}T_{2,1,j},
\end{eqnarray}
where \cite{Stojnicnflgscompyx23}'s (49) gives
 \begin{eqnarray}\label{eq:genDanal25a01}
 \sum_{i_1=1}^{l}  \sum_{i_2=1}^{l} \sum_{j=1}^{m}  \beta_{i_1}T_{2,1,j}& = & \sqrt{1-t}\p_1\beta^2
 \Bigg(\Bigg. \mE_{G,{\mathcal U}_2}\langle \|\x^{(i_1)}\|_2^2\|\y^{(i_2)}\|_2^2\rangle_{\gamma_{01}^{(1)}} \nonumber \\
 & & +  (s-1)\mE_{G,{\mathcal U}_2}\langle \|\x^{(i_1)}\|_2^2(\y^{(p_2)})^T\y^{(i_2)}\rangle_{\gamma_{02}^{(1)}} \Bigg.\Bigg) \nonumber \\
& & - \sqrt{1-t}\p_1s\beta^2(1-\m_1)\mE_{G,{\mathcal U}_2}\langle \|\x^{(i_1)}\|_2\|\x^{(p_1)}\|_2(\y^{(p_2)})^T\y^{(i_2)} \rangle_{\gamma_{1}^{(1)}}\nonumber \\
 &   &
  -\sqrt{1-t}\p_1s\beta^2\m_1 p \mE_{G,{\mathcal U}_2} \langle \|\x^{(i_1)}\|_2\|\x^{(p_1)}\|_2(\y^{(p_2)})^T\y^{(i_2)} \rangle_{\gamma_{21}^{(1)}}
    \nonumber
  \\
   &   &
   +  \sqrt{1-t}s\beta^2\p_1\m_1 (p-1) \mE_{G,{\mathcal U}_2} \langle \|\x^{(i_1)}\|_2\|\x^{(p_1)}\|_2(\y^{(p_2)})^T\y^{(i_2)} \rangle_{\gamma_{22}^{(1)}}.
\end{eqnarray}
Combining (\ref{eq:genDanal19}) and (\ref{eq:genDanal25a01}) one then finds
\begin{eqnarray}\label{eq:genDanal25}
 \sum_{i_1=1}^{l}  \sum_{i_2=1}^{l} \sum_{j=1}^{m}  \beta_{i_1}\frac{\sqrt{1-t}}{\sqrt{\p_1}}T_{2,1,j}^{\p}& = & (1-t)\beta^2
 \Bigg(\Bigg. \mE_{G,{\mathcal U}_2}\langle \|\x^{(i_1)}\|_2^2\|\y^{(i_2)}\|_2^2\rangle_{\gamma_{01}^{(1)}} \nonumber \\
 & & +  (s-1)\mE_{G,{\mathcal U}_2}\langle \|\x^{(i_1)}\|_2^2(\y^{(p_2)})^T\y^{(i_2)}\rangle_{\gamma_{02}^{(1)}} \Bigg.\Bigg) \nonumber \\
& & - (1-t)s\beta^2(1-\m_1)\mE_{G,{\mathcal U}_2}\langle \|\x^{(i_1)}\|_2\|\x^{(p_1)}\|_2(\y^{(p_2)})^T\y^{(i_2)} \rangle_{\gamma_{1}^{(1)}}\nonumber \\
 &   &
  -(1-t)s\beta^2\m_1 p \mE_{G,{\mathcal U}_2} \langle \|\x^{(i_1)}\|_2\|\x^{(p_1)}\|_2(\y^{(p_2)})^T\y^{(i_2)} \rangle_{\gamma_{21}^{(1)}}
      \nonumber
  \\
   &   &
   +  (1-t) s\beta^2\m_1 (p-1) \mE_{G,{\mathcal U}_2} \langle \|\x^{(i_1)}\|_2\|\x^{(p_1)}\|_2(\y^{(p_2)})^T\y^{(i_2)} \rangle_{\gamma_{22}^{(1)}}.
\end{eqnarray}

%%%%%%%%%%%%%%%%%%%%%%%%%%%%%%%%%%%%%%%%%%%%%%%%%%%%%%%%%%%%%%%%%%%%%%%%
\underline{\textbf{\emph{Determining}} $T_{2,2}^{\q}$}
\label{sec:hand1T22}
%%%%%%%%%%%%%%%%%%%%%%%%%%%%%%%%%%%%%%%%%%%%%%%%%%%%%%%%%%%%%%%%%%%%%%%%

Another application of Gaussian integration by parts gives
{\small \begin{eqnarray}\label{eq:liftgenEanal20}
T_{2,2}^{\q} & = &  \mE_{G,{\mathcal U}_2} \lp
 \sum_{i_3=1}^{l}
 \frac{
\lp \mE_{{\mathcal U}_1} Z_{i_3}^{\m_1} \rp^{p-1}    }{  \lp
\sum_{i_3=1}^{l}
\lp \mE_{{\mathcal U}_1} Z_{i_3}^{\m_1} \rp^p \rp   }
\mE_{{\mathcal U}_1}\frac{(C^{(i_1)})^{s-1} A^{(i_1,i_2)} \u^{(i_1,3,2)}}{Z^{1-\m_1}} \rp \nonumber \\
 & = & \mE_{G,{\mathcal U}_1} \lp
\sum_{i_3=1}^{l}
\mE_{{\mathcal U}_2} \lp \sum_{p_1=1}^{l}\mE_{{\mathcal U}_2}(\u^{(i_1,3,2)}\u^{(p_1,3,2)}) \frac{d}{d\u^{(p_1,3,2)}}\lp\frac{(C^{(i_1)})^{s-1} A^{(i_1,i_2)}}{Z^{1-\m_1} }
 \frac{
\lp \mE_{{\mathcal U}_1} Z_{i_3}^{\m_1} \rp^{p-1}    }{  \lp
\sum_{i_3=1}^{l}
\lp \mE_{{\mathcal U}_1} Z_{i_3}^{\m_1} \rp^p \rp   }
\rp\rp\rp \nonumber \\
& = & \frac{\sqrt{\q_1}}{\q_1}
\Bigg( \Bigg.
\mE_{G,{\mathcal U}_2,{\mathcal U}_1} \Bigg (\Bigg .
  \sum_{i_3=1}^{l}
 \frac{
\lp \mE_{{\mathcal U}_1} Z_{i_3}^{\m_1} \rp^{p-1}    }{  \lp
\sum_{i_3=1}^{l}
\lp \mE_{{\mathcal U}_1} Z_{i_3}^{\m_1} \rp^p \rp   }
 \sum_{p_1=1}^{l}\mE_{{\mathcal U}_2}(\sqrt{\q_1}\u^{(i_1,3,2)}\sqrt{\q_1}\u^{(p_1,3,2)})
  \nonumber \\
  & & \times
  \frac{d}{d \lp \sqrt{\q_1}\u^{(p_1,3,2)}\rp}\lp\frac{(C^{(i_1)})^{s-1} A^{(i_1,i_2)}}{Z^{1-\m_1}}\rp
  \Bigg )\Bigg .
   \nonumber \\
& & + \mE_{G,{\mathcal U}_2,{\mathcal U}_1} \Bigg ( \Bigg .
  \sum_{i_3=1}^{l}
\lp \mE_{{\mathcal U}_1} Z_{i_3}^{\m_1} \rp^{p-1}
\frac{(C^{(i_1)})^{s-1} A^{(i_1,i_2)}}{Z^{1-\m_1}}
  \sum_{p_1=1}^{l}\mE_{{\mathcal U}_2}(\sqrt{\q_1}\u^{(i_1,3,2)}\sqrt{\q_1}\u^{(p_1,3,2)})
  \nonumber \\
  & & \times
      \frac{d}{d \lp\sqrt{\q_1}  \u^{(p_1,3,2)}\rp}\lp\frac{1}{   \lp
\sum_{i_3=1}^{l}
\lp \mE_{{\mathcal U}_1} Z_{i_3}^{\m_1} \rp^p \rp    }
      \rp\Bigg )\Bigg .  \Bigg. \Bigg) \nonumber \\
& & + \mE_{G,{\mathcal U}_2,{\mathcal U}_1} \Bigg ( \Bigg .
  \sum_{i_3=1}^{l}
 \frac{ 1  }{  \lp
\sum_{i_3=1}^{l}
\lp \mE_{{\mathcal U}_1} Z_{i_3}^{\m_1} \rp^p \rp   }
\frac{(C^{(i_1)})^{s-1} A^{(i_1,i_2)}}{Z^{1-\m_1}}
  \sum_{p_1=1}^{l}\mE_{{\mathcal U}_2}(\sqrt{\q_1}\u^{(i_1,3,2)}\sqrt{\q_1}\u^{(p_1,3,2)})
  \nonumber \\
  & & \times
      \frac{d}{d \lp\sqrt{\q_1}  \u^{(p_1,3,2)}\rp}\lp \mE_{{\mathcal U}_1} Z^{\m_1}   \rp^{p-1}  \Bigg )\Bigg .  \Bigg. \Bigg) \nonumber \\
  & = & \frac{\sqrt{\q_1}}{\q_1} T_{2,2},
 \end{eqnarray}}

\noindent where in \cite{Stojnicnflgscompyx23}'s (64) one finds
 \begin{eqnarray}\label{eq:genEanal25a01}
\sum_{i_1=1}^{l}\sum_{i_2=1}^{l} \beta_{i_1}\|\y^{(i_2)}\|_2 T_{2,2} &  = &
\sqrt{1-t}\q_1\beta^2 \Bigg( \Bigg. \mE_{G,{\mathcal U}_2}\langle \|\x^{(i_1)}\|_2^2\|\y^{(i_2)}\|_2^2\rangle_{\gamma_{01}^{(1)}} \nonumber \\
& & +  (s-1)\mE_{G,{\mathcal U}_2}\langle \|\x^{(i_1)}\|_2^2 \|\y^{(i_2)}\|_2\|\y^{(p_2)}\|_2\rangle_{\gamma_{02}^{(1)}}\Bigg.\Bigg) \nonumber \\
& & - \sqrt{1-t}\q_1s\beta^2(1-\m_1)\mE_{G,{\mathcal U}_2}\langle (\x^{(p_1)})^T\x^{(i_1)}\|\y^{(i_2)}\|_2\|\y^{(p_2)}\|_2 \rangle_{\gamma_{1}^{(1)}} \nonumber \\
&  & -\sqrt{1-t}\q_1s\beta^2\m_1 p \mE_{G,{\mathcal U}_2} \langle \|\y^{(i_2)}\|_2\|\y^{(p_2)}\|_2(\x^{(i_1)})^T\x^{(p_1)}\rangle_{\gamma_{21}^{(1)}}
\nonumber \\
& &
+ \sqrt{1-t}s\beta^2\q_1\m_1 (p-1) \mE_{G,{\mathcal U}_2} \langle \|\x^{(i_1)}\|_2\|\x^{(p_1)}\|_2(\y^{(p_2)})^T\y^{(i_2)} \rangle_{\gamma_{22}^{(1)}}.
\end{eqnarray}
A combination of (\ref{eq:liftgenEanal20}) and (\ref{eq:genEanal25a01}) then gives
  \begin{eqnarray}\label{eq:genEanal25}
\sum_{i_1=1}^{l}\sum_{i_2=1}^{l} \beta_{i_1}\|\y^{(i_2)}\|_2\frac{\sqrt{1-t}}{\sqrt{\q_1}}T_{2,2}^{\q} &  = &
(1-t)\beta^2 \Bigg( \Bigg. \mE_{G,{\mathcal U}_2}\langle \|\x^{(i_1)}\|_2^2\|\y^{(i_2)}\|_2^2\rangle_{\gamma_{01}^{(1)}} \nonumber \\
& & +  (s-1)\mE_{G,{\mathcal U}_2}\langle \|\x^{(i_1)}\|_2^2 \|\y^{(i_2)}\|_2\|\y^{(p_2)}\|_2\rangle_{\gamma_{02}^{(1)}}\Bigg.\Bigg) \nonumber \\
& & - (1-t)s\beta^2(1-\m_1)\mE_{G,{\mathcal U}_2}\langle (\x^{(p_1)})^T\x^{(i_1)}\|\y^{(i_2)}\|_2\|\y^{(p_2)}\|_2 \rangle_{\gamma_{1}^{(1)}} \nonumber \\
&  & -(1-t)s\beta^2\m_1 p \mE_{G,{\mathcal U}_2} \langle \|\y^{(i_2)}\|_2\|\y^{(p_2)}\|_2(\x^{(i_1)})^T\x^{(p_1)}\rangle_{\gamma_{21}^{(1)}}
\nonumber \\
& &
+ (1-t)s\beta^2\m_1 (p-1) \mE_{G,{\mathcal U}_2} \langle \|\x^{(i_1)}\|_2\|\x^{(p_1)}\|_2(\y^{(p_2)})^T\y^{(i_2)} \rangle_{\gamma_{22}^{(1)}}.
\nonumber \\
\end{eqnarray}

%%%%%%%%%%%%%%%%%%%%%%%%%%%%%%%%%%%%%%%%%%%%%%%%%%%%%%%%%%%%%%%%%%%%%%%%
\underline{\textbf{\emph{Determining}} $T_{2,3}^{(\p,\q)}$}
\label{sec:hand1T23}
%%%%%%%%%%%%%%%%%%%%%%%%%%%%%%%%%%%%%%%%%%%%%%%%%%%%%%%%%%%%%%%%%%%%%%%%

After Gaussian integration by parts we have
\begin{eqnarray}\label{eq:genFanal21}
T_{2,3}^{(\p,\q)} & = &  \mE_{G,{\mathcal U}_2}\lp
 \sum_{i_3=1}^{l}
 \frac{
\lp \mE_{{\mathcal U}_1} Z_{i_3}^{\m_1} \rp^{p-1}    }{  \lp
\sum_{i_3=1}^{l}
\lp \mE_{{\mathcal U}_1} Z_{i_3}^{\m_1} \rp^p \rp   }
\mE_{{\mathcal U}_1}\frac{(C^{(i_1)})^{s-1} A^{(i_1,i_2)} u^{(4,2)}}{Z^{1-\m_1}} \rp \nonumber \\
& = & \mE_{G,{\mathcal U}_1} \lp
  \sum_{i_3=1}^{l}
 \mE_{{\mathcal U}_2} \lp\mE_{{\mathcal U}_2} (u^{(4,2)}u^{(4,2)})\lp\frac{d}{du^{(4,2)}} \lp\frac{(C^{(i_1)})^{s-1} A^{(i_1,i_2)}}{Z^{1-\m_1} }
 \frac{
\lp \mE_{{\mathcal U}_1} Z_{i_3}^{\m_1} \rp^{p-1}    }{  \lp
\sum_{i_3=1}^{l}
\lp \mE_{{\mathcal U}_1} Z_{i_3}^{\m_1} \rp^p \rp   }
 \rp \rp\rp\rp \nonumber \\
& = & \frac{\sqrt{\p_1\q_1}}{\p_1\q_1}\mE_{G,{\mathcal U}_2,{\mathcal U}_1} \Bigg ( \Bigg .
 \sum_{i_3=1}^{l}
 \frac{
\lp \mE_{{\mathcal U}_1} Z_{i_3}^{\m_1} \rp^{p-1}    }{  \lp
\sum_{i_3=1}^{l}
\lp \mE_{{\mathcal U}_1} Z_{i_3}^{\m_1} \rp^p \rp   }
 \mE_{{\mathcal U}_2} (\sqrt{\p_1\q_1}u^{(4,2)}\sqrt{\p_1\q_1}u^{(4,2)})
 \nonumber \\
 & & \times
\lp\frac{d}{d \lp \sqrt{\p_1\q_1}u^{(4,2)}\rp} \lp\frac{(C^{(i_1)})^{s-1} A^{(i_1,i_2)}}{Z^{1-\m_1}}\rp\rp\Bigg )\Bigg . \nonumber \\
& & + \frac{\sqrt{\p_1\q_1}}{\p_1\q_1}
\times \mE_{G,{\mathcal U}_2,{\mathcal U}_1} \Bigg (\Bigg.
 \sum_{i_3=1}^{l}
\lp \mE_{{\mathcal U}_1} Z_{i_3}^{\m_1} \rp^{p-1}
\frac{\mE_{{\mathcal U}_2} (\sqrt{\p_1\q_1}u^{(4,2)}\sqrt{\p_1\q_1}u^{(4,2)})(C^{(i_1)})^{s-1} A^{(i_1,i_2)}}{Z^{1-\m_1}}
\nonumber \\
& & \times
\Bigg ( \Bigg. \frac{d}{d\lp \sqrt{\p_1\q_1} u^{(4,2)}\rp}
\Bigg ( \Bigg.
\frac{1}{   \lp
\sum_{i_3=1}^{l}
\lp \mE_{{\mathcal U}_1} Z_{i_3}^{\m_1} \rp^p \rp    }
 \Bigg ) \Bigg. \Bigg ) \Bigg. \Bigg ) \Bigg. \nonumber \\
& & + \frac{\sqrt{\p_1\q_1}}{\p_1\q_1}
\times \mE_{G,{\mathcal U}_2,{\mathcal U}_1} \Bigg (\Bigg.
 \sum_{i_3=1}^{l}
 \frac{
1  }{  \lp
\sum_{i_3=1}^{l}
\lp \mE_{{\mathcal U}_1} Z_{i_3}^{\m_1} \rp^p \rp   }
\frac{\mE_{{\mathcal U}_2} (\sqrt{\p_1\q_1}u^{(4,2)}\sqrt{\p_1\q_1}u^{(4,2)})(C^{(i_1)})^{s-1} A^{(i_1,i_2)}}{Z^{1-\m_1}}
\nonumber \\
& & \times
\Bigg ( \Bigg. \frac{d}{d\lp \sqrt{\p_1\q_1} u^{(4,2)}\rp}
\Bigg ( \Bigg.
 \lp \mE_{{\mathcal U}_1} Z_{i_3}^{\m_1} \rp^{p-1}
 \Bigg ) \Bigg. \Bigg ) \Bigg. \Bigg ) \Bigg. \nonumber \\
& = & \frac{\sqrt{\p_1\q_1}}{\p_1\q_1} T_{2,3},
\end{eqnarray}
where \cite{Stojnicnflldp25}'s (78) determined
 \begin{eqnarray}\label{eq:genFanal29a01}
\sum_{i_1=1}^{l}\sum_{i_2=1}^{l} \beta_{i_1}\|\y^{(i_2)}\|_2 T_{2,3}
& = &
\sqrt{t} \p_1\q_1\beta^2 \Bigg( \Bigg. \mE_{G,{\mathcal U}_2}\langle \|\x^{(i_1)}\|_2^2\|\y^{(i_2)}\|_2^2\rangle_{\gamma_{01}^{(1)}} \nonumber \\
& & +   (s-1)\mE_{G,{\mathcal U}_2}\langle \|\x^{(i_1)}\|_2^2 \|\y^{(i_2)}\|_2\|\y^{(p_2)}\|_2\rangle_{\gamma_{02}^{(1)}}\Bigg.\Bigg) \nonumber \\
& & - \sqrt{t}\p_1\q_1 s\beta^2(1-\m_1)\mE_{G,{\mathcal U}_2}\langle \|\x^{(i_1)}\|_2\|\x^{(p_`)}\|_2\|\y^{(i_2)}\|_2\|\y^{(p_2)}\|_2 \rangle_{\gamma_{1}^{(1)}} \nonumber \\
&  & -\sqrt{t}s\beta^2\p_1\q_1\m_1 p \mE_{G,{\mathcal U}_2} \langle\|\x^{(i_2)}\|_2\|\x^{(p_2)}\|_2\|\y^{(i_2)}\|_2\|\y^{(p_2)} \|_2 \rangle_{\gamma_{21}^{(1)}}
\nonumber \\
& & +
\sqrt{t} s\beta^2 \p_1 \q_1\m_1 (p-1) \mE_{G,{\mathcal U}_2} \langle \|\x^{(i_1)}\|_2\|\x^{(p_1)}\|_2  \|\y^{(p_2)}\|_2    \|\y^{(i_2)}\|_2 \rangle_{\gamma_{22}^{(1)}}.
\end{eqnarray}
A combination of (\ref{eq:genFanal21}) and (\ref{eq:genFanal29a01}) then gives
 \begin{eqnarray}\label{eq:genFanal29}
\sum_{i_1=1}^{l}\sum_{i_2=1}^{l} \beta_{i_1}\|\y^{(i_2)}\|_2\frac{\sqrt{t}}{\sqrt{\p_1\q_1}}T_{2,3}^{(\p,\q)}
& = &
t\beta^2 \Bigg( \Bigg. \mE_{G,{\mathcal U}_2}\langle \|\x^{(i_1)}\|_2^2\|\y^{(i_2)}\|_2^2\rangle_{\gamma_{01}^{(1)}} \nonumber \\
& & +   (s-1)\mE_{G,{\mathcal U}_2}\langle \|\x^{(i_1)}\|_2^2 \|\y^{(i_2)}\|_2\|\y^{(p_2)}\|_2\rangle_{\gamma_{02}^{(1)}}\Bigg.\Bigg) \nonumber \\
& & - t s\beta^2(1-\m_1)\mE_{G,{\mathcal U}_2}\langle \|\x^{(i_1)}\|_2\|\x^{(p_`)}\|_2\|\y^{(i_2)}\|_2\|\y^{(p_2)}\|_2 \rangle_{\gamma_{1}^{(1)}} \nonumber \\
&  & -ts\beta^2\m_1 p \mE_{G,{\mathcal U}_2} \langle\|\x^{(i_2)}\|_2\|\x^{(p_2)}\|_2\|\y^{(i_2)}\|_2\|\y^{(p_2)} \|_2 \rangle_{\gamma_{21}^{(1)}}
\nonumber \\
& & +
 t s\beta^2 \p_1 \q_1\m_1 (p-1) \mE_{G,{\mathcal U}_2} \langle \|\x^{(i_1)}\|_2\|\x^{(p_1)}\|_2  \|\y^{(p_2)}\|_2    \|\y^{(i_2)}\|_2 \rangle_{\gamma_{22}^{(1)}}. \nonumber \\
\end{eqnarray}

%%%%%%%%%%%%%%%%%%%%%%%%%%%%%%%%%%%%%%%%%%%%%%%%%%%%%%%%%%%%%%%%%%%%%%%%
\subsubsection{Connecting everything together -- first level}
\label{sec:conalt}
%%%%%%%%%%%%%%%%%%%%%%%%%%%%%%%%%%%%%%%%%%%%%%%%%%%%%%%%%%%%%%%%%%%%%%%%

To combine all the results obtained above, we first based on (\ref{eq:genanal10e}) and (\ref{eq:genanal10f})  write
\begin{eqnarray}\label{eq:ctp1}
\frac{d\psi(t)}{d\p_1}  & = &       \frac{\mbox{sign}(s)}{2 \sqrt{n}} \lp \Omega_1+\q_1\Omega_3\rp \nonumber \\
\frac{d\psi(t)}{d\q_1}  & = &       \frac{\mbox{sign}(s)}{2 \sqrt{n}} \lp \Omega_2+\p_1\Omega_3\rp,
\end{eqnarray}
where
\begin{eqnarray}\label{eq:ctp2}
\Omega_1 & = & \sum_{i_1=1}^{l}  \sum_{i_2=1}^{l} \sum_{j=1}^{m}\beta_{i_1}\frac{\sqrt{1-t}}{\sqrt{\p_1}}T_{2,1,j}^{\p}-\sum_{i_1=1}^{l}  \sum_{i_2=1}^{l} \sum_{j=1}^{m}\beta_{i_1}\frac{\sqrt{1-t}}{\sqrt{\p_0-\p_1}}T_{1,1,j}^{\p} \nonumber\\
\Omega_2 & = & \sum_{i_1=1}^{l}  \sum_{i_2=1}^{l}\beta_{i_1}\|\y^{(i_2)}\|_2\frac{\sqrt{1-t}}{\sqrt{\q_1}}T_{2,2}^{\q}-\sum_{i_1=1}^{l}  \sum_{i_2=1}^{l}\beta_{i_1}\|\y^{(i_2)}\|_2\frac{\sqrt{1-t}}{\q_0-\q_1}T_{1,2}^{\q} \nonumber\\
\Omega_3 & = & \sum_{i_1=1}^{l}  \sum_{i_2=1}^{l}\beta_{i_1}\|\y^{(i_2)}\|_2\frac{\sqrt{t}}{\sqrt{\p_1\q_1}}T_{2,3}^{(\p,\q)}- \sum_{i_1=1}^{l}  \sum_{i_2=1}^{l}\beta_{i_1}\|\y^{(i_2)}\|_2\frac{\sqrt{t}}{\p_0\q_0-\p_1\q_1}T_{1,3}^{(\p,\q)}.
\end{eqnarray}
We then from (\ref{eq:liftgenAanal19i}) and (\ref{eq:genDanal25}) find
\begin{eqnarray}\label{eq:cpt4}
-\Omega_1& = & (1-t)\beta^2 \lp \mE_{G,{\mathcal U}_2}\langle \|\x^{(i_1)}\|_2^2\|\y^{(i_2)}\|_2^2\rangle_{\gamma_{01}^{(1)}} +   (s-1)\mE_{G,{\mathcal U}_2}\langle \|\x^{(i_1)}\|_2^2(\y^{(p_2)})^T\y^{(i_2)}\rangle_{\gamma_{02}^{(1)}} \rp \nonumber \\
& & - (1-t)s\beta^2(1-\m_1)\mE_{G,{\mathcal U}_2}\langle \|\x^{(i_1)}\|_2\|\x^{(p_1)}\|_2(\y^{(p_2)})^T\y^{(i_2)} \rangle_{\gamma_{1}^{(1)}} \nonumber\\
& & -(1-t)\beta^2 \lp \mE_{G,{\mathcal U}_2}\langle \|\x^{(i_1)}\|_2^2\|\y^{(i_2)}\|_2^2\rangle_{\gamma_{01}^{(1)}} +   (s-1)\mE_{G,{\mathcal U}_2}\langle \|\x^{(i_1)}\|_2^2(\y^{(p_2)})^T\y^{(i_2)}\rangle_{\gamma_{02}^{(1)}} \rp \nonumber \\
& & +(1-t)s\beta^2(1-\m_1)\mE_{G,{\mathcal U}_2}\langle \|\x^{(i_1)}\|_2\|\x^{(p_1)}\|_2(\y^{(p_2)})^T\y^{(i_2)} \rangle_{\gamma_{1}^{(1)}}\nonumber \\
 &   &
  +(1-t)s\beta^2\m_1 p \mE_{G,{\mathcal U}_2} \langle \|\x^{(i_1)}\|_2\|\x^{(p_1)}\|_2(\y^{(p_2)})^T\y^{(i_2)} \rangle_{\gamma_{21}^{(1)}} \nonumber \\
 &   &
  -(1-t)s\beta^2\m_1 (p-1) \mE_{G,{\mathcal U}_2} \langle \|\x^{(i_1)}\|_2\|\x^{(p_1)}\|_2(\y^{(p_2)})^T\y^{(i_2)} \rangle_{\gamma_{22}^{(1)}} \nonumber \\
& = &
  (1-t)s\beta^2\m_1 p \mE_{G,{\mathcal U}_2} \langle \|\x^{(i_1)}\|_2\|\x^{(p_1)}\|_2(\y^{(p_2)})^T\y^{(i_2)} \rangle_{\gamma_{21}^{(1)}}
\nonumber \\
   &   &
  -(1-t)s\beta^2\m_1 (p-1) \mE_{G,{\mathcal U}_2} \langle \|\x^{(i_1)}\|_2\|\x^{(p_1)}\|_2(\y^{(p_2)})^T\y^{(i_2)} \rangle_{\gamma_{22}^{(1)}}.
\end{eqnarray}
From (\ref{eq:liftgenBanal20b}) and (\ref{eq:genEanal25}), we analogously find
\begin{eqnarray}\label{eq:cpt5}
-\Omega_2 & = & (1-t)\beta^2 \lp\mE_{G,{\mathcal U}_2}\langle \|\x^{(i_1)}\|_2^2\|\y^{(i_2)}\|_2^2\rangle_{\gamma_{01}^{(1)}} +   (s-1)\mE_{G,{\mathcal U}_2}\langle \|\x^{(i_1)}\|_2^2 \|\y^{(i_2)}\|_2\|\y^{(p_2)}\|_2\rangle_{\gamma_{02}^{(1)}}\rp\nonumber \\
& & - (1-t) s\beta^2(1-\m_1)\mE_{G,{\mathcal U}_2}\langle (\x^{(p_1)})^T\x^{(i_1)}\|\y^{(i_2)}\|_2\|\y^{(p_2)}\|_2 \rangle_{\gamma_{1}^{(1)}}\nonumber \\
&  & -
(1-t)\beta^2\lp\mE_{G,{\mathcal U}_2}\langle \|\x^{(i_1)}\|_2^2\|\y^{(i_2)}\|_2^2\rangle_{\gamma_{01}^{(1)}} +   (s-1)\mE_{G,{\mathcal U}_2}\langle \|\x^{(i_1)}\|_2^2 \|\y^{(i_2)}\|_2\|\y^{(p_2)}\|_2\rangle_{\gamma_{02}^{(1)}}\rp\nonumber \\
& & + (1-t) s\beta^2(1-\m_1)\mE_{G,{\mathcal U}_2}\langle (\x^{(p_1)})^T\x^{(i_1)}\|\y^{(i_2)}\|_2\|\y^{(p_2)}\|_2 \rangle_{\gamma_{1}^{(1)}} \nonumber \\
&  & + (1-t) s\beta^2 \m_1 p \mE_{G,{\mathcal U}_2} \langle \|\y^{(i_2)}\|_2\|\y^{(p_2)}\|_2(\x^{(i_1)})^T\x^{(p_1)}\rangle_{\gamma_{21}^{(1)}} \nonumber \\
&  & - (1-t) s\beta^2 \m_1 (p-1) \mE_{G,{\mathcal U}_2} \langle \|\y^{(i_2)}\|_2\|\y^{(p_2)}\|_2(\x^{(i_1)})^T\x^{(p_1)}\rangle_{\gamma_{22}^{(1)}} \nonumber \\
&  = &
   (1-t) s\beta^2\m_1 p\mE_{G,{\mathcal U}_2} \langle \|\y^{(i_2)}\|_2\|\y^{(p_2)}\|_2(\x^{(i_1)})^T\x^{(p_1)}\rangle_{\gamma_{21}^{(1)}}
   \nonumber \\
&  & - (1-t) s\beta^2 \m_1 (p-1) \mE_{G,{\mathcal U}_2} \langle \|\y^{(i_2)}\|_2\|\y^{(p_2)}\|_2(\x^{(i_1)})^T\x^{(p_1)}\rangle_{\gamma_{22}^{(1)}},
\end{eqnarray}
and from (\ref{eq:liftgenCanal21b}) and (\ref{eq:genFanal29})
  \begin{eqnarray}\label{eq:cpt6}
-\Omega_3 & = & t \beta^2 \lp \mE_{G,{\mathcal U}_2}\langle \|\x^{(i_1)}\|_2^2\|\y^{(i_2)}\|_2^2\rangle_{\gamma_{01}^{(1)}} +   (s-1)\mE_{G,{\mathcal U}_2}\langle \|\x^{(i_1)}\|_2^2 \|\y^{(i_2)}\|_2\|\y^{(p_2)}\|_2\rangle_{\gamma_{02}^{(1)}}\rp\nonumber \\
& & - ts\beta^2(1-\m_1)\mE_{G,{\mathcal U}_2}\langle \|\x^{(i_1)}\|_2\|\x^{(p_`)}\|_2\|\y^{(i_2)}\|_2\|\y^{(p_2)}\|_2 \rangle_{\gamma_{1}^{(1)}}\nonumber \\
&  & -
t\beta^2\lp\mE_{G,{\mathcal U}_2}\langle \|\x^{(i_1)}\|_2^2\|\y^{(i_2)}\|_2^2\rangle_{\gamma_{01}^{(1)}} +   (s-1)\mE_{G,{\mathcal U}_2}\langle \|\x^{(i_1)}\|_2^2 \|\y^{(i_2)}\|_2\|\y^{(p_2)}\|_2\rangle_{\gamma_{02}^{(1)}}\rp\nonumber \\
& & + t s\beta^2(1-\m_1)\mE_{G,{\mathcal U}_2}\langle \|\x^{(i_1)}\|_2\|\x^{(p_`)}\|_2\|\y^{(i_2)}\|_2\|\y^{(p_2)}\|_2 \rangle_{\gamma_{1}^{(1)}} \nonumber \\
&  & +ts\beta^2\m_1 p \mE_{G,{\mathcal U}_2} \mE_{G,{\mathcal U}_2}\langle\|\x^{(i_2)}\|_2\|\x^{(p_2)}\|_2\|\y^{(i_2)}\|_2\|\y^{(p_2)}    \|_2 \rangle_{\gamma_{21}^{(1)}} \nonumber \\
&  & -ts\beta^2\m_1 (p-1) \mE_{G,{\mathcal U}_2} \mE_{G,{\mathcal U}_2}\langle\|\x^{(i_2)}\|_2\|\x^{(p_2)}\|_2\|\y^{(i_2)}\|_2\|\y^{(p_2)}  \|_2 \rangle_{\gamma_{22}^{(1)}} \nonumber \\
& = &
 ts\beta^2
\m_1 p \mE_{G,{\mathcal U}_2} \langle\|\x^{(i_1)}\|_2\|\x^{(p_1)}\|_2\|\y^{(i_2)}\|_2\|\y^{(p_2)}   \|_2 \rangle_{\gamma_{21}^{(1)}}
\nonumber \\
&  & -ts\beta^2\m_1 (p-1) \mE_{G,{\mathcal U}_2} \mE_{G,{\mathcal U}_2}\langle\|\x^{(i_2)}\|_2\|\x^{(p_2)}\|_2\|\y^{(i_2)}\|_2\|\y^{(p_2)} \|_2 \rangle_{\gamma_{22}^{(1)}}.
\end{eqnarray}
Finally, a combination of  (\ref{eq:ctp1}) and (\ref{eq:cpt4})-(\ref{eq:cpt6}) gives
\begin{eqnarray}\label{eq:cpt7}
\frac{d\psi(t)}{d\p_1}  & = &       \frac{\mbox{sign}(s)\beta^2  }{2\sqrt{n}} \phi^{(1,\p)}\nonumber \\
\frac{d\psi(t)}{d\q_1}  & = &       \frac{\mbox{sign}(s)\beta^2  }{2\sqrt{n}} \phi^{(1,\q)},
 \end{eqnarray}
where
\begin{eqnarray}\label{eq:cpt8}
\phi^{(1,\p)} & = &
  -(1-t)s \m_1 p \mE_{G,{\mathcal U}_2} \langle \|\x^{(i_1)}\|_2\|\x^{(p_1)}\|_2(\y^{(p_2)})^T\y^{(i_2)} \rangle_{\gamma_{21}^{(1)}} \nonumber \\
& &   - ts \q_1
\m_1 p \mE_{G,{\mathcal U}_2} \langle\|\x^{(i_1)}\|_2\|\x^{(p_1)}\|_2\|\y^{(i_2)}\|_2\|\y^{(p_2)} \|_2 \rangle_{\gamma_{21}^{(1)}}\nonumber \\
&  &
  +(1-t)s \m_1 (p-1) \mE_{G,{\mathcal U}_2} \langle \|\x^{(i_1)}\|_2\|\x^{(p_1)}\|_2(\y^{(p_2)})^T\y^{(i_2)} \rangle_{\gamma_{22}^{(1)}} \nonumber \\
& &   + ts \q_1
\m_1 (p-1) \mE_{G,{\mathcal U}_2} \langle\|\x^{(i_1)}\|_2\|\x^{(p_1)}\|_2\|\y^{(i_2)}\|_2\|\y^{(p_2)} \|_2 \rangle_{\gamma_{22}^{(1)}}\nonumber \\
\phi^{(1,\q)} & = &
   -(1-t) s \m_1 p \mE_{G,{\mathcal U}_2} \langle \|\y^{(i_2)}\|_2\|\y^{(p_2)}\|_2(\x^{(i_1)})^T\x^{(p_1)}\rangle_{\gamma_{21}^{(1)}} \nonumber \\
& &   - ts \p_1
\m_1 p \mE_{G,{\mathcal U}_2} \langle\|\x^{(i_1)}\|_2\|\x^{(p_1)}\|_2\|\y^{(i_2)}\|_2\|\y^{(p_2)}\rangle_{\gamma_{21}^{(1)}}
\nonumber \\
&  &
   + (1-t) s \m_1 (p-1) \mE_{G,{\mathcal U}_2} \langle \|\y^{(i_2)}\|_2\|\y^{(p_2)}\|_2(\x^{(i_1)})^T\x^{(p_1)}\rangle_{\gamma_{22}^{(1)}} \nonumber \\
& &   + ts \p_1
\m_1 (p-1) \mE_{G,{\mathcal U}_2} \langle\|\x^{(i_1)}\|_2\|\x^{(p_1)}\|_2\|\y^{(i_2)}\|_2\|\y^{(p_2)}\rangle_{\gamma_{22}^{(1)}}.
\end{eqnarray}

The following proposition summarizes the above results.
\begin{proposition}
\label{thm:thm1} Let vectors $\m = [\m_1 ]$, $\p=[\p_0,\p_1,\p_2]$,  and $\q=[\q_0,\q_1,\q_2]$ be such that  $\p_0\geq \p_1\geq \p_2=0$ and $\q_0\geq \q_1\geq \q_2=0$. For $k\in\{1,2\}$  let the elements of $G\in\mR^{m \times n},u^{(4,k)}\in\mR^1,\u^{(2,k)}\in\mR^{m\times 1}$, and $\h^{(k)}\in\mR^{n\times 1}$ be independent standard normals. Set $a_k=\sqrt{\p_{k-1}\q_{k-1}-\p_k\q_k}$, $b_k=\sqrt{\p_{k-1}-\p_k}$, $c_k=\sqrt{\q_{k-1}-\q_k}$, and  ${\mathcal U}_k=[u^{(4,k)},\u^{(2,k)},\h^{(k)}]$. Let sets ${\mathcal X}=\{\x^{(1)},\x^{(2)},\dots,\x^{(l)}\}$, $\bar{{\mathcal X}}=\{\bar{\x}^{(1)},\bar{\x}^{(2)},\dots,\bar{\x}^{(l)}\}$, and ${\mathcal Y}=\{\y^{(1)},\y^{(2)},\dots,\y^{(l)}\}$  (with $\x^{(i)},\bar{\x}^{(i)}\in \mR^{n}$ and $\y^{(i)}\in \mR^{m}$), real scalars $\beta$, $p$, and $s$ (with $\beta,p\geq 0$), and function $f_{\bar{\x}^{(i_3)}} (\cdot): \mR^n\rightarrow \mR$ be given. One then has for function
\begin{equation}\label{eq:prop1eq1}
\psi( t)  =  \mE_{G,{\mathcal U}_2} \frac{1}{p|s|\sqrt{n}\m_1}
\log \lp  \sum_{i_3=1}^{l}  \lp  \mE_{{\mathcal U}_1} \lp \sum_{i_1=1}^{l}\lp\sum_{i_2=1}^{l}e^{\beta D_0^{(i_1,i_2,i_3 )}} \rp^{s}\rp^{\m_1} \rp^p \rp,
\end{equation}
with
\begin{eqnarray}\label{eq:prop1eq2}
 D_0^{(i_1,i_2,i_3)} & \triangleq & \sqrt{t}(\y^{(i_2)})^T
 G\x^{(i_1)}+\sqrt{1-t}\|\x^{(i_2)}\|_2 (\y^{(i_2)})^T(b_2\u^{(2,1)}+b_2\u^{(2,2)})\nonumber \\
 & & +\sqrt{t}\|\x^{(i_1)}\|_2\|\y^{(i_2)}\|_2(a_1u^{(4,1)}+a_2u^{(4,2)}) +\sqrt{1-t}\|\y^{(i_2)}\|_2(c_1\h^{(1)}+c_2\h^{(2)})^T\x^{(i)}
\nonumber \\
& &
   + f_{\bar{\x}^{(i_3)}  } (\x^{(i_1)}) ,
 \end{eqnarray}
that
\begin{eqnarray}\label{eq:prop1eq3}
\frac{d\psi(t)}{d\p_1}  & = &       \frac{\mbox{sign}(s)\beta^2}{2\sqrt{n}} \phi^{(1,\p)}\nonumber \\
\frac{d\psi(t)}{d\q_1}  & = &       \frac{\mbox{sign}(s)\beta^2}{2\sqrt{n}} \phi^{(1,\q)},
 \end{eqnarray}
where $\phi$'s are as in (\ref{eq:cpt8}).
\end{proposition}
\begin{proof}
  Follows automatically from the above discussion.
\end{proof}

%%%%%%%%%%%%%%%%%%%%%%%%%%%%%%%%%%%%%%%%%%%%%%%%%%%%%%%%%%%%%%%%%
%%%%%%%%%%%%%%%%%%%%%%%%%%%%%%%%%%%%%%%%%%%%%%%%%%%%%%%%%%%%%%%%%
%%%%%%%%%%%%%%%%%%%%%%%%%%%%%%%%%%%%%%%%%%%%%%%%%%%%%%%%%%%%%%%%%
%%%%%%%%%%%%%%%%%%%%%%%%%%%%%%%%%%%%%%%%%%%%%%%%%%%%%%%%%%%%%%%%%
%%%%%%%%%%%%%%%%%%%%%%%%%%%%%%%%%%%%%%%%%%%%%%%%%%%%%%%%%%%%%%%%%
\section{$\p,\q$-derivatives  -- $r$-th level}
\label{sec:rthlev}
%%%%%%%%%%%%%%%%%%%%%%%%%%%%%%%%%%%%%%%%%%%%%%%%%%%%%%%%%%%%%%%%%
%%%%%%%%%%%%%%%%%%%%%%%%%%%%%%%%%%%%%%%%%%%%%%%%%%%%%%%%%%%%%%%%%
%%%%%%%%%%%%%%%%%%%%%%%%%%%%%%%%%%%%%%%%%%%%%%%%%%%%%%%%%%%%%%%%%
%%%%%%%%%%%%%%%%%%%%%%%%%%%%%%%%%%%%%%%%%%%%%%%%%%%%%%%%%%%%%%%%%
%%%%%%%%%%%%%%%%%%%%%%%%%%%%%%%%%%%%%%%%%%%%%%%%%%%%%%%%%%%%%%%%%
%%%%%%%%%%%%%%%%%%%%%%%%%%%%%%%%%%%%%%%%%%%%%%%%%%%%%%%%%%%%%%%%%

We now show how the results of Section \ref{sec:gencon}  extend to hold for an arbitrary lifting level, $r\in\mN$.  All conceptual ingredients needed for such an extension are already present in Section \ref{sec:gencon}. We below connect them so that the extension can be fully formalized.

For $r\geq 2$ let vectors $\m=[\m_1,\m_2,...,\m_r]$,
$\p=[\p_0,\p_1,...,\p_r,\p_{r+1}]$, and $\q=[\q_0,\q_1,\q_2,\dots,\q_r,\q_{r+1}]$ be such that $\p_0\geq \p_1\geq \p_2\geq \dots \geq \p_r\geq\p_{r+1} = 0$  and $\q_0\geq \q_1\geq \q_2\geq \dots \geq \q_r\geq \q_{r+1} = 0$. For any $k\in\{1,2,\dots,r+1\}$ let the components of $G\in\mR^{m \times n}$,   $u^{(4,k)}\in\mR^1,\u^{(2,k)}\in\mR^{m\times 1}$, and $\h^{(k)}\in\mR^{n\times 1}$ be independent standard normals and set $a_k=\sqrt{\p_{k-1}\q_{k-1}-\p_k\q_k}$, $b_k=\sqrt{\p_{k-1}-\p_k}$, $c_k=\sqrt{\q_{k-1}-\q_k}$ and ${\mathcal U}_k\triangleq [u^{(4,k)},\u^{(2,k)},\h^{(k)}]$.  Let sets ${\mathcal X}$, $\bar{{\mathcal X}}$, and ${\mathcal Y}$, scalars  $\beta$, $p$, and $s$, and function $f_{\bar{\x}^{(i_3)} } (\cdot)$ be as in Proposition \ref{thm:thm1}. Our focus is then the following function
\begin{equation}\label{eq:rthlev2genanal3}
\psi(t)  =  \mE_{G,{\mathcal U}_{r+1}} \frac{1}{p|s|\sqrt{n}\m_r} \log
\lp \mE_{{\mathcal U}_{r}} \lp \dots \lp \mE_{{\mathcal U}_2}\lp\lp \sum_{i_3=1}^{l} \lp \mE_{{\mathcal U}_1}  Z_{i_3}^{\m_1}\rp^p \rp^{\frac{\m_2}{\m_1}}\rp\rp^{\frac{\m_3}{\m_2}} \dots \rp^{\frac{\m_{r}}{\m_{r-1}}}\rp,
\end{equation}
where
\begin{eqnarray}\label{eq:rthlev2genanal3a}
Z_{i_3} & \triangleq & \sum_{i_1=1}^{l}\lp\sum_{i_2=1}^{l}e^{\beta D_0^{(i_1,i_2,i_3)}} \rp^{s} \nonumber \\
 D_0^{(i_1,i_2,i_3)} & \triangleq & \sqrt{t}(\y^{(i_2)})^T
 G\x^{(i_1)}+\sqrt{1-t}\|\x^{(i_1)}\|_2 (\y^{(i_2)})^T\lp\sum_{k=1}^{r+1}b_k\u^{(2,k)}\rp\nonumber \\
 & & +\sqrt{t}\|\x^{(i_1)}\|_2\|\y^{(i_2)}\|_2\lp\sum_{k=1}^{r+1}a_ku^{(4,k)}\rp +\sqrt{1-t}\|\y^{(i_2)}\|_2\lp\sum_{k=1}^{r+1}c_k\h^{(k)}\rp^T\x^{(i_1)}
+ f_{\bar{\x}^{(i_3 )} } ( \x^{(i_1)} ). \nonumber \\
 \end{eqnarray}
Analogously to (\ref{eq:genanal4}), we set
\begin{eqnarray}\label{eq:rthlev2genanal4}
\u^{(i_1,1)} & =  & \frac{G\x^{(i_1)}}{\|\x^{(i_1)}\|_2} \nonumber \\
\u^{(i_1,3,k)} & =  & \frac{(\h^{(k)})^T\x^{(i_1)}}{\|\x^{(i_1)}\|_2},
\end{eqnarray}
recall on
\begin{eqnarray}\label{eq:rthlev2genanal5}
\bar{\u}_j^{(i_1,1)} & =  & \frac{G_{j,1:n}\x^{(i_1)}}{\|\x^{(i_1)}\|_2},1\leq j\leq m.
\end{eqnarray}
and note that the elements of $\u^{(i_1,1)}$, $\u^{(2,k)}$, and $\u^{(i_1,3,k)}$ are standard normals. This allows to rewrite (\ref{eq:rthlev2genanal3}) as
\begin{equation}\label{eq:rthlev2genanal6}
\psi(t)  =  \mE_{G,{\mathcal U}_{r+1}} \frac{1}{p|s|\sqrt{n}\m_r} \log
\lp \mE_{{\mathcal U}_{r}} \lp \dots \lp \mE_{{\mathcal U}_2}\lp\lp  \sum_{i_3=1}^{l} \lp \mE_{{\mathcal U}_1} Z_{i_3}^{\m_1}\rp^p \rp^{\frac{\m_2}{\m_1}}\rp\rp^{\frac{\m_3}{\m_2}} \dots \rp^{\frac{\m_{r}}{\m_{r-1}}}\rp,
\end{equation}
where $\beta_{i_1}=\beta\|\x^{(i_1)}\|_2$ and now
\begin{eqnarray}\label{eq:rthlev2genanal7}
B^{(i_1,i_2)} & \triangleq &  \sqrt{t}(\y^{(i_2)})^T\bar{\u}^{(i_1,1)}+\sqrt{1-t} (\y^{(i_2)})^T\lp \sum_{k=1}^{r+1}b_k\u^{(2,k)} \rp \nonumber \\
D^{(i_1,i_2,i_3)} & \triangleq &  B^{(i_1,i_2)}+\sqrt{t}\|\y^{(i_2)}\|_2 \lp \sum_{k=1}^{r+1}a_ku^{(4,k)}\rp+\sqrt{1-t}\|\y^{(i_2)}\|_2 \lp \sum_{k=1}^{r+1}c_k\u^{(i_1,3,k)}  \rp
+ f_{\bar{\x}^{(i_3 )} } ( \x^{(i_1)} )  \nonumber \\
A_{i_3}^{(i_1,i_2)} & \triangleq &  e^{\beta_{i_1}D^{(i_1,i_2,i_3)}}\nonumber \\
C_{i_3}^{(i_1)} & \triangleq &  \sum_{i_2=1}^{l}A_{i_3}^{(i_1,i_2)}\nonumber \\
Z_{i_3} & \triangleq & \sum_{i_1=1}^{l} \lp \sum_{i_2=1}^{l} A_{i_3}^{(i_1,i_2)}\rp^s =\sum_{i_1=1}^{l}  (C_{i_3}^{(i_1)})^s.
\end{eqnarray}
We set $\m_0=1$ and analogously to \cite{Stojnicnflldp25}'s (230).
\begin{eqnarray}\label{eq:rthlev2genanal7a}
\zeta_r\triangleq \mE_{{\mathcal U}_{r}} \lp \dots \lp \mE_{{\mathcal U}_2}\lp\lp  \sum_{i_3=1}^{l} \lp \mE_{{\mathcal U}_1}  Z_{i_3}^{\frac{\m_1}{\m_0}}\rp^p \rp^{\frac{\m_2}{\m_1}}\rp\rp^{\frac{\m_3}{\m_2}} \dots \rp^{\frac{\m_{r}}{\m_{r-1}}}, r\geq 2.
\end{eqnarray}
From \cite{Stojnicnflldp25}'s (231) one then also has
\begin{eqnarray}\label{eq:rthlev2genanal7b}
\zeta_k = \mE_{{\mathcal U}_{k}} \lp  \zeta_{k-1} \rp^{\frac{\m_{k}}{\m_{k-1}}}, k\geq 2,\quad \mbox{and} \quad
\zeta_1=  \sum_{i_3=1}^{l} \lp \mE_{{\mathcal U}_1}  Z_{i_3}^{\frac{\m_1}{\m_0}}\rp^p.
\end{eqnarray}
Relying on \cite{Stojnicsflgscompyx23}'s (64), we then have for the derivative
\begin{eqnarray}\label{eq:rthlev2genanal9}
\frac{d\psi( t)}{d\p_{k_1}}
 & = &  \mE_{G,{\mathcal U}_{r+1}} \frac{1}{p |s|\sqrt{n}\m_1\zeta_r}
\prod_{k=r}^{2}\mE_{{\mathcal U}_{k}} \zeta_{k-1}^{\frac{\m_k}{\m_{k-1}}-1}
 \frac{d\zeta_{1}}{d\p_{k_1}} \nonumber \\
& = &  \mE_{G,{\mathcal U}_{r+1}} \frac{1}{ p |s|\sqrt{n}\m_1\zeta_r}
\prod_{k=r}^{2}\mE_{{\mathcal U}_{k}} \zeta_{k-1}^{\frac{\m_k}{\m_{k-1}}-1}
\sum_{i_3=1}^{l} \lp \mE_{{\mathcal U}_1} Z_{i_3}^{\m_1}  \rp^{p-1}
 \frac{d \mE_{{\mathcal U}_1} Z_{i_3}^{\m_1} }{d\p_{k_1}} \nonumber \\
 & = &
\mE_{G,{\mathcal U}_{r+1}} \frac{1}{ |s|\sqrt{n}\zeta_r}
\prod_{k=r}^{2}\mE_{{\mathcal U}_{k}} \zeta_{k-1}^{\frac{\m_k}{\m_{k-1}}-1}
\sum_{i_3=1}^{l} \lp \mE_{{\mathcal U}_1} Z_{i_3}^{\m_1}  \rp^{p-1}
 \mE_{{\mathcal U}_1} \frac{1}{Z_{i_3}^{1-\m_1}}  \sum_{i=1}^{l} (C_{i_3}^{(i_1)})^{s-1} \nonumber \\
& & \times \sum_{i_2=1}^{l}\beta_{i_1}A_{i_3}^{(i_1,i_2)}\frac{dD^{(i_1,i_2,i_3)}}{d\p_{k_1}},
\end{eqnarray}
where we note that the product is running in an \emph{index decreasing order}. Moreover, from \cite{Stojnicsflgscompyx23}'s (65)-(66) we also have
\begin{eqnarray}\label{eq:rthlev2genanal9a}
\frac{dD^{(i_1,i_2,i_3)}}{d\p_{k_1}} & = & \lp \frac{dB^{(i_1,i_2)}}{d\p_{k_1}}+
\sqrt{t} \frac{d \lp\|\y^{(i_2)}\|_2 (\sum_{k=k_1}^{k_1+1} a_ku^{(4,k)}) \rp}{d\p_{k_1}}\rp \nonumber \\
& = & \lp \frac{dB^{(i_1,i_2)}}{d\p_{k_1}}
-\sqrt{t} \q_{k_1}\frac{\|\y^{(i_2)}\|_2 u^{(4,k_1)}}{2\sqrt{\p_{k_1-1}\q_{k_1-1}-\p_{k_1}\q_{k_1}}}+\sqrt{t} \q_{k_1}\frac{\|\y^{(i_2)}\|_2 u^{(4,k_1+1)}}{2\sqrt{\p_{k_1}\q_{k_1}-\p_{k_1+1}\q_{k_1+1}}}\rp,
\end{eqnarray}
and
\begin{eqnarray}\label{eq:rthlev2genanal10}
\frac{dB^{(i_1,i_2)}}{d\p_{k_1}} & = &   \frac{d\lp\sqrt{t}(\y^{(i_2)})^T\u^{(i_1,1)}+\sqrt{1-t} (\y^{(i_2)})^T(\sum_{k=1}^{r+1} b_k\u^{(2,k)})\rp}{d\p_{k_1}} \nonumber \\
& = &   \frac{d\lp\sqrt{1-t} (\y^{(i_2)})^T(\sum_{k=k_1}^{k_1+1} b_k\u^{(2,k)})\rp}{d\p_{k_1}} \nonumber \\
 & = &
-\sqrt{1-t}\sum_{j=1}^{m}\lp  \frac{\y_j^{(i_2)} \u_j^{(2,k_1)}}{2\sqrt{\p_{k_1-1}-\p_{k_1}}}\rp
+\sqrt{1-t}\sum_{j=1}^{m}\lp  \frac{\y_j^{(i_2)} \u_j^{(2,k_1+1)}}{2\sqrt{\p_{k_1}-\p_{k_1+1}}}\rp.
\end{eqnarray}
This then allows to write analogously to (\ref{eq:genanal10e})
\begin{equation}\label{eq:rthlev2genanal10e}
\frac{d\psi(t)}{d\p_{k_1}}  =       \frac{\mbox{sign}(s)}{2 \sqrt{n}} \sum_{i_1=1}^{l}  \sum_{i_2=1}^{l}
\beta_{i_1}\lp T_{k_1+1}^{\p}-T_{k_1}^{\p}\rp,
\end{equation}
where
\begin{eqnarray}\label{eq:rthlev2genanal10f}
 T_k^{\p} & = & \frac{\sqrt{1-t}}{\sqrt{\p_{k-1}-\p_{k}}}\sum_{j=1}^{m}T_{k,1,j}^{\p}
 +\frac{\sqrt{t}\q_{k_1}\|\y^{(i_2)}\|_2}{\sqrt{\p_{k-1}\q_{k-1}-\p_{k}\q_{k}}}T_{k,3}^{(\p,\q)}, k\in\{k_1,k_1+1\},
 \end{eqnarray}
 and
 \begin{eqnarray}\label{eq:rthlev2genanal10g}
 T_{k,1,j}^{\p} & = &   \mE_{G,{\mathcal U}_{r+1}} \lp
\zeta_r^{-1}\prod_{v=r}^{2}\mE_{{\mathcal U}_{v}} \zeta_{v-1}^{\frac{\m_v}{\m_{v-1}}-1}
\sum_{i_3=1}^{l} \lp \mE_{{\mathcal U}_1} Z_{i_3}^{\m_1}  \rp^{p-1}
  \mE_{{\mathcal U}_1}\frac{(C_{i_3}^{(i_1)})^{s-1} A_{i_3}^{(i_1,i_2)} \y_j^{(i_2)}\u_j^{(2,k)}}{Z_{i_3}^{1-\m_1}} \rp \nonumber \\
%T_{k,2}^{(\p)} & = &   \mE_{G,{\mathcal U}_{r+1}} \lp
%\zeta_r^{-1}\prod_{v=r}^{2}\mE_{{\mathcal U}_{v}} \zeta_{v-1}^{\frac{\m_v}{\m_{v-1}}-1}
%  \mE_{{\mathcal U}_1}\frac{(C^{(i_1)})^{s-1} A^{(i_1,i_2)} \u^{(i_1,3,k)}}{Z^{1-\m_1}} \rp \nonumber \\
T_{k,3}^{(\p,\q)} & = &  \mE_{G,{\mathcal U}_{r+1}} \lp
\zeta_r^{-1}\prod_{v=r}^{2}\mE_{{\mathcal U}_{v}} \zeta_{v-1}^{\frac{\m_v}{\m_{v-1}}-1}
\sum_{i_3=1}^{l} \lp \mE_{{\mathcal U}_1} Z_{i_3}^{\m_1}  \rp^{p-1}
  \mE_{{\mathcal U}_1}\frac{(C_{i_3}^{(i_1)})^{s-1} A_{i_3}^{(i_1,i_2)} u^{(4,k)}}{Z_{i_3}^{1-\m_1}} \rp.
\end{eqnarray}
Making the analogy with the first level of lifting we immediately write  analogously to (\ref{eq:rthlev2genanal10e})-(\ref{eq:rthlev2genanal10g})
\begin{equation}\label{eq:rthlev2genanal10e1}
\frac{d\psi(t)}{d\q_{k_1}}  =       \frac{\mbox{sign}(s)}{2 \sqrt{n}} \sum_{i_1=1}^{l}  \sum_{i_2=1}^{l}
\beta_{i_1}\lp T_{k_1+1}^{\q}-T_{k_1}^{\q}\rp,
\end{equation}
where
\begin{eqnarray}\label{eq:rthlev2genanal10f1}
 T_k^{\q} & = & \frac{\sqrt{1-t}}{\sqrt{\q_{k-1}-\q_{k}}} \|\y^{(i_2)}\|_2 T_{k,2}^{\q}
 +\frac{\sqrt{t}\p_{k_1}\|\y^{(i_2)}\|_2}{\sqrt{\p_{k-1}\q_{k-1}-\p_{k}\q_{k}}}T_{k,3}^{(\p,\q)}, k\in\{k_1,k_1+1\},
 \end{eqnarray}
 and
 \begin{eqnarray}\label{eq:rthlev2genanal10g1}
% T_{k,1,j}^{\p} & = &   \mE_{G,{\mathcal U}_{r+1}} \lp
%\zeta_r^{-1}\prod_{v=r}^{2}\mE_{{\mathcal U}_{v}} \zeta_{v-1}^{\frac{\m_v}{\m_{v-1}}-1}
%  \mE_{{\mathcal U}_1}\frac{(C^{(i_1)})^{s-1} A^{(i_1,i_2)} \y_j^{(i_2)}\u_j^{(2,k)}}{Z^{1-\m_1}} \rp \nonumber \\
T_{k,2}^{\p} & = &   \mE_{G,{\mathcal U}_{r+1}} \lp
\zeta_r^{-1}\prod_{v=r}^{2}\mE_{{\mathcal U}_{v}} \zeta_{v-1}^{\frac{\m_v}{\m_{v-1}}-1}
\sum_{i_3=1}^{l} \lp \mE_{{\mathcal U}_1} Z_{i_3}^{\m_1}  \rp^{p-1}
  \mE_{{\mathcal U}_1}\frac{(C_{i_3}^{(i_1)})^{s-1} A_{i_3}^{(i_1,i_2)} \u^{(i_1,3,k)}}{Z_{i_3}^{1-\m_1}} \rp \nonumber \\
T_{k,3}^{(\p,\q)} & = &  \mE_{G,{\mathcal U}_{r+1}} \lp
\zeta_r^{-1}\prod_{v=r}^{2}\mE_{{\mathcal U}_{v}} \zeta_{v-1}^{\frac{\m_v}{\m_{v-1}}-1}
\sum_{i_3=1}^{l} \lp \mE_{{\mathcal U}_1} Z_{i_3}^{\m_1}  \rp^{p-1}
  \mE_{{\mathcal U}_1}\frac{(C_{i_3}^{(i_1)})^{s-1} A_{i_3}^{(i_1,i_2)} u^{(4,k)}}{Z_{i_3}^{1-\m_1}} \rp.
\end{eqnarray}
Overall, there are  three sequences that need to be determined $\lp T_{k,1,j}^{\p}\rp_{k=1:r+1}$, $\lp T_{k,2}^{\q}\rp_{k=1:r+1}$, and $\lp T_{k,3}^{(\p,\q)}\rp_{k=1:r+1}$. However, in an analogy with what was observed in \cite{Stojnicnflgscompyx23} and \cite{Stojnicnflldp25}, the components within each of these sequences are connected in identical way. This allows to handle one of them, say, $\lp T_{k,1,j}^{\p}\rp_{k=1:r+1}$, and then automatically deduce the final results for the remaining two.

%%%%%%%%%%%%%%%%%%%%%%%%%%%%%%%%%%%%%%%%%%%%%%%%%%%%%%%%%%%%%%%%%%%%%%%%%%%%%%%%%%
%%%%%%%%%%%%%%%%%%%%%%%%%%%%%%%%%%%%%%%%%%%%%%%%%%%%%%%%%%%%%%%%%%%%%%%%%%%%%%%%%%
%%%%%%%%%%%%%%%%%%%%%%%%%%%%%%%%%%%%%%%%%%%%%%%%%%%%%%%%%%%%%%%%%%%%%%%%%%%%%%%%%%
%%%%%%%%%%%%%%%%%%%%%%%%%%%%%%%%%%%%%%%%%%%%%%%%%%%%%%%%%%%%%%%%%%%%%%%%%%%%%%%%%%
\subsection{Scaling and canceling out}
\label{sec:scaledcanout}
%%%%%%%%%%%%%%%%%%%%%%%%%%%%%%%%%%%%%%%%%%%%%%%%%%%%%%%%%%%%%%%%%%%%%%%%%%%%%%%%%%
%%%%%%%%%%%%%%%%%%%%%%%%%%%%%%%%%%%%%%%%%%%%%%%%%%%%%%%%%%%%%%%%%%%%%%%%%%%%%%%%%%
%%%%%%%%%%%%%%%%%%%%%%%%%%%%%%%%%%%%%%%%%%%%%%%%%%%%%%%%%%%%%%%%%%%%%%%%%%%%%%%%%%
%%%%%%%%%%%%%%%%%%%%%%%%%%%%%%%%%%%%%%%%%%%%%%%%%%%%%%%%%%%%%%%%%%%%%%%%%%%%%%%%%%

The way two successive sequence elements, $T_{k_1,1,j}$ and $T_{k_1+1,1,j}$ (for $k_1\geq 2$) are related to each other is shown below. Similarly to what was observed in \cite{Stojnicnflgscompyx23,Stojnicsflgscompyx23,Stojnicnflldp25}, there are two key components of the recursive connection: 1) \emph{successive scaling}; and 2) appearance of \emph{reweightedly averaged  new terms}.

First we find via Gaussian integration by parts
 \begin{eqnarray}\label{eq:rthlev2genanal11}
 T_{k_1,1,j}^{\p} \hspace{-.05in}  & = &   \mE_{G,{\mathcal U}_{r+1}} \lp
\zeta_r^{-1}\prod_{v=r}^{2}\mE_{{\mathcal U}_{v}} \zeta_{v-1}^{\frac{\m_v}{\m_{v-1}}-1}
\sum_{i_3=1}^{l} \lp \mE_{{\mathcal U}_1} Z_{i_3}^{\m_1}  \rp^{p-1}
  \mE_{{\mathcal U}_1}\frac{(C_{i_3}^{(i_1)})^{s-1} A_{i_3}^{(i_1,i_2)} \y_j^{(i_2)}\u_j^{(2,k_1)}}{Z_{i_3}^{1-\m_1}} \rp \nonumber \\
 & = &   \mE_{G,{\mathcal U}_{r+1}} \Bigg ( \Bigg .
\zeta_r^{-1}\prod_{v=r}^{k_1+1}\mE_{{\mathcal U}_{v}} \zeta_{v-1}^{\frac{\m_v}{\m_{v-1}}-1}
\prod_{v=k_1}^{2}\mE_{{\mathcal U}_{v}} \zeta_{v-1}^{\frac{\m_v}{\m_{v-1}}-1}
\sum_{i_3=1}^{l} \lp \mE_{{\mathcal U}_1} Z_{i_3}^{\m_1}  \rp^{p-1}
\nonumber \\
& & \times
  \mE_{{\mathcal U}_1}\frac{(C_{i_3}^{(i_1)})^{s-1} A_{i_3}^{(i_1,i_2)} \y_j^{(i_2)}\u_j^{(2,k_1)}}{Z_{i_3}^{1-\m_1}} \Bigg ) \Bigg.
  \nonumber \\
    & = &   \mE_{G,{\mathcal U}_{r+1}} \Bigg(\Bigg.
\zeta_r^{-1}\prod_{v=r}^{k_1+1}\mE_{{\mathcal U}_{v}} \zeta_{v-1}^{\frac{\m_v}{\m_{v-1}}-1}
\mE_{{\mathcal U}_{k_1}} \mE_{{\mathcal U}_{k_1}} \lp \u_j^{(2,k_1)}\u_j^{(2,k_1)}\rp \nonumber \\
& & \times
\frac{d}{d\u_j^{(2,k_1)}} \lp \zeta_{k_1-1}^{\frac{\m_{k_1}}{\m_{k_1-1}}-1}\prod_{v=k_1-1}^{2}\mE_{{\mathcal U}_{v}} \zeta_{v-1}^{\frac{\m_v}{\m_{v-1}}-1}
  \sum_{i_3=1}^{l} \lp \mE_{{\mathcal U}_1} Z_{i_3}^{\m_1}  \rp^{p-1}
  \mE_{{\mathcal U}_1}\frac{(C_{i_3}^{(i_1)})^{s-1} A_{i_3}^{(i_1,i_2)} \y_j^{(i_2)}\u_j^{(2,k_1)}}{Z_{i_3}^{1-\m_1}} \rp \Bigg.\Bigg)
  \nonumber \\
    & = &  \frac{\sqrt{\p_{k_1-1}-\p_{k_1}}}{\p_{k_1-1}-\p_{k_1}} \nonumber \\
    & & \times \mE_{G,{\mathcal U}_{r+1}} \Bigg(\Bigg.
\zeta_r^{-1}\prod_{v=r}^{k_1+1}\mE_{{\mathcal U}_{v}} \zeta_{v-1}^{\frac{\m_v}{\m_{v-1}}-1}
\mE_{{\mathcal U}_{k_1}} \mE_{{\mathcal U}_{k_1}} \lp \sqrt{\p_{k_1-1}-\p_{k_1}}\u_j^{(2,k_1)}\sqrt{\p_{k_1-1}-\p_{k_1}}\u_j^{(2,k_1)}\rp \nonumber \\
& & \times
\frac{d}{d \lp \sqrt{\p_{k_1-1}-\p_{k_1}} \u_j^{(2,k_1)} \rp}
\Bigg (\Bigg .
\zeta_{k_1-1}^{\frac{\m_{k_1}}{\m_{k_1-1}}-1}\prod_{v=k_1-1}^{2}\mE_{{\mathcal U}_{v}} \zeta_{v-1}^{\frac{\m_v}{\m_{v-1}}-1}
\sum_{i_3=1}^{l} \lp \mE_{{\mathcal U}_1} Z_{i_3}^{\m_1}  \rp^{p-1}
\nonumber \\
& & \times
  \mE_{{\mathcal U}_1}\frac{(C_{i_3}^{(i_1)})^{s-1} A_{i_3}^{(i_1,i_2)} \y_j^{(i_2)}\u_j^{(2,k_1)}}{Z_{i_3}^{1-\m_1}}
  \Bigg )\Bigg.
  \Bigg.\Bigg) \nonumber \\
   & = &  \frac{\sqrt{\p_{k_1-1}-\p_{k_1}}}{\p_{k_1-1}-\p_{k_1}}  T_{k_1,1,j},
\end{eqnarray}
where \cite{Stojnicnflldp25} determined $T_{k_1,1,j}$.

Also,
 \begin{eqnarray}\label{eq:rthlev2genanal12}
 T_{k_1+1,1,j}^{\p}
  & = &   \mE_{G,{\mathcal U}_{r+1}} \Bigg (\Bigg.
\zeta_r^{-1}\prod_{v=r}^{k_1+2}\mE_{{\mathcal U}_{v}} \zeta_{v-1}^{\frac{\m_v}{\m_{v-1}}-1}
\prod_{v=k_1+1}^{2}\mE_{{\mathcal U}_{v}} \zeta_{v-1}^{\frac{\m_v}{\m_{v-1}}-1}
\sum_{i_3=1}^{l} \lp \mE_{{\mathcal U}_1} Z_{i_3}^{\m_1}  \rp^{p-1}
\nonumber \\
& & \times
  \mE_{{\mathcal U}_1}\frac{(C^{(i_1)})^{s-1} A^{(i_1,i_2)} \y_j^{(i_2)}\u_j^{(2,k_1+1)}}{Z^{1-\m_1}} \Bigg ) \Bigg.
   \nonumber \\
    & = &   \mE_{G,{\mathcal U}_{r+1}} \Bigg( \Bigg.
\zeta_r^{-1}\prod_{v=r}^{k_1+2}\mE_{{\mathcal U}_{v}} \zeta_{v-1}^{\frac{\m_v}{\m_{v-1}}-1}
\mE_{{\mathcal U}_{k_1+1}}   \nonumber \\
& & \times  \zeta_{k_1}^{\frac{\m_{k_1+1}}{\m_{k_1}}-1}
\prod_{v=k_1}^{2}\mE_{{\mathcal U}_{v}} \zeta_{v-1}^{\frac{\m_v}{\m_{v-1}}-1}
\sum_{i_3=1}^{l} \lp \mE_{{\mathcal U}_1} Z_{i_3}^{\m_1}  \rp^{p-1}
  \mE_{{\mathcal U}_1}\frac{(C^{(i_1)})^{s-1} A^{(i_1,i_2)} \y_j^{(i_2)}\u_j^{(2,k_1+1)}}{Z^{1-\m_1}} \Bigg. \Bigg) \nonumber \\
      & = &   \mE_{G,{\mathcal U}_{r+1}} \Bigg( \Bigg.
\zeta_r^{-1}\prod_{v=r}^{k_1+2}\mE_{{\mathcal U}_{v}} \zeta_{v-1}^{\frac{\m_v}{\m_{v-1}}-1}
\mE_{{\mathcal U}_{k_1+1}} \mE_{{\mathcal U}_{k_1+1}} (\u_j^{(2,k_1+1)}\u_j^{(2,k_1+1)}) \nonumber \\
& & \times  \frac{d}{d\u_j^{(2,k_1+1)}}  \lp \zeta_{k_1}^{\frac{\m_{k_1+1}}{\m_{k_1}}-1}
\prod_{v=k_1}^{2}\mE_{{\mathcal U}_{v}} \zeta_{v-1}^{\frac{\m_v}{\m_{v-1}}-1}
\sum_{i_3=1}^{l} \lp \mE_{{\mathcal U}_1} Z_{i_3}^{\m_1}  \rp^{p-1}
  \mE_{{\mathcal U}_1}\frac{(C^{(i_1)})^{s-1} A^{(i_1,i_2)} \y_j^{(i_2)}}{Z^{1-\m_1}} \rp   \Bigg. \Bigg)
   \nonumber \\
      & = &  \frac{\sqrt{\p_{k_1} -\p_{k_1+1}}
}{\p_{k_1} -\p_{k_1+1}} \nonumber \\
& & \times \mE_{G,{\mathcal U}_{r+1}} \Bigg( \Bigg.
\zeta_r^{-1}\prod_{v=r}^{k_1+2}\mE_{{\mathcal U}_{v}} \zeta_{v-1}^{\frac{\m_v}{\m_{v-1}}-1}
\nonumber \\
& & \times
\mE_{{\mathcal U}_{k_1+1}} \mE_{{\mathcal U}_{k_1+1}} (\sqrt{\p_{k_1} -\p_{k_1+1}}
\u_j^{(2,k_1+1)}\sqrt{\p_{k_1} -\p_{k_1+1}}
\u_j^{(2,k_1+1)}) \nonumber \\
& & \times  \frac{d}{d \lp \sqrt{\p_{k_1} -\p_{k_1+1}}
 \u_j^{(2,k_1+1)}\rp}
 \Bigg ( \Bigg.
 \zeta_{k_1}^{\frac{\m_{k_1+1}}{\m_{k_1}}-1}
\prod_{v=k_1}^{2}\mE_{{\mathcal U}_{v}} \zeta_{v-1}^{\frac{\m_v}{\m_{v-1}}-1}
\sum_{i_3=1}^{l} \lp \mE_{{\mathcal U}_1} Z_{i_3}^{\m_1}  \rp^{p-1}
  \nonumber \\
  & & \times
  \mE_{{\mathcal U}_1}\frac{(C^{(i_1)})^{s-1} A^{(i_1,i_2)} \y_j^{(i_2)}}{Z^{1-\m_1}} \Bigg ) \Bigg .   \Bigg. \Bigg)   \nonumber \\
      & = &  \frac{\sqrt{\p_{k_1} -\p_{k_1+1}}
}{\p_{k_1} -\p_{k_1+1}} T_{k_1+1,1,j} \nonumber \\
 & = &  \frac{\sqrt{\p_{k_1} -\p_{k_1+1}}
}{\p_{k_1} -\p_{k_1+1}}\lp \Pi_{k_1+1,1,j}^{(1)}
+  \Pi_{k_1+1,1,j}^{(2)}\rp,
 \end{eqnarray}
where all three $T_{k_1+1,1,j}$, $\Pi_{k_1+1,1,j}^{(1)}$, and $\Pi_{k_1+1,1,j}^{(1)}$ are determined in \cite{Stojnicnflldp25}. In particular,  \cite{Stojnicnflldp25} obtained
 \begin{eqnarray}\label{eq:rthlev2genanal19}
 \Pi_{k_1+1,1,j}^{(1)}=\frac{(\p_{k_1} -\p_{k_1+1})}{(\p_{k_1-1}-\p_{k_1})}T_{k_1,1,j},
 \end{eqnarray}
and
\begin{eqnarray}\label{eq:rthlev2genanal26}
 \sum_{i_1=1}^{l}\sum_{i_2=1}^{l} \sum_{j=1}^{m}
\beta_{i_1} \Pi_{k_1+1,1,j}^{(2)}
  & = & -\sqrt{1-t} s\beta^2(\p_{k_1}-\p_{k_1+1}) \lp \m_{k_1} - \m_{k_1+1} \rp \nonumber\\
 & & \times
  \mE_{G,{\mathcal U}_{r+1}} \left \langle \|\x^{(i_1)}\|_2\|\x^{(p_1)}\|_2 (\y^{(p_2)})^T \y^{(i_2)} \right \rangle_{\gamma_{k_1+1}^{(r)}}.
\end{eqnarray}
Following \cite{Stojnicsflgscompyx23}'s (77),  a combination of (\ref{eq:rthlev2genanal12}), (\ref{eq:rthlev2genanal19}), and (\ref{eq:rthlev2genanal26}), then gives
 \begin{eqnarray}\label{eq:rthlev2genanal27}
 \sum_{i_1=1}^{l}\sum_{i_2=1}^{l} \sum_{j=1}^{m}
 \frac{\beta_{i_1}\sqrt{1-t}}{\sqrt{\p_{k_1}-\p_{k_1+1}}} T_{k_1+1,1,j}^{\p}
& = &
 \sum_{i_1=1}^{l}\sum_{i_2=1}^{l} \sum_{j=1}^{m}
\beta_{i_1} \frac{\sqrt{1-t}}{\p_{k_1}-\p_{k_1+1}}     \lp \Pi_{k_1+1,1,j}^{(1)}
+  \Pi_{k_1+1,1,j}^{(2)}\rp \nonumber \\
& = &
 \sum_{i_1=1}^{l}\sum_{i_2=1}^{l} \sum_{j=1}^{m}
 \frac{\beta_{i_1}\sqrt{1-t}T_{k_1,1,j}}{\p_{k_1-1}-\p_{k_1}}  +   \sum_{i_1=1}^{l}\sum_{i_2=1}^{l} \sum_{j=1}^{m}
 \frac{\beta_{i_1}\sqrt{1-t} \Pi_{k_1+1,1,j}^{(2)}}{\p_{k_1}-\p_{k_1+1}} \nonumber \\
& = &
 \sum_{i_1=1}^{l}\sum_{i_2=1}^{l} \sum_{j=1}^{m}
 \frac{\beta_{i_1}\sqrt{1-t}T_{k_1,1,j}^{\p}}{\sqrt{\p_{k_1-1}-\p_{k_1}}}  +   \sum_{i_1=1}^{l}\sum_{i_2=1}^{l} \sum_{j=1}^{m}
 \frac{\beta_{i_1}\sqrt{1-t} \Pi_{k_1+1,1,j}^{(2)}}{\p_{k_1}-\p_{k_1+1}} \nonumber \\
 & = & \sum_{i_1=1}^{l}\sum_{i_2=1}^{l} \sum_{j=1}^{m}
 \frac{\beta_{i_1}\sqrt{1-t}T_{k_1,1,j}^{\p}}{\sqrt{\p_{k_1-1}-\p_{k_1}}}   \nonumber \\
 &  & - (1-t)\Bigg( \Bigg. s\beta^2 \lp \m_{k_1} -  \m_{k_1+1} \rp  \nonumber \\
 & & \times
  \mE_{G,{\mathcal U}_{r+1}} \left \langle \|\x^{(i_1)}\|_2\|\x^{(p_1)}\|_2 (\y^{(p_2)})^T \y^{(i_2)} \right \rangle_{\gamma_{k_1+1}^{(r)} } \Bigg. \Bigg).
\end{eqnarray}
Following into the footsteps of \cite{Stojnicsflgscompyx23} one now observes that a repetition of the above mechanism together with the first level results from Section \ref{sec:gencon} allow to  write in an analogous fashion for sequence $\lp T_{k,2}^{\q}\rp_{k=1:r+1}$
\begin{eqnarray}\label{eq:rthlev2genanal28}
 \sum_{i_1=1}^{l}\sum_{i_2=1}^{l}
 \frac{\beta_{i_1}\sqrt{1-t} \|\y^{(i_2)}\|_2}{\sqrt{\q_{k_1}-\q_{k_1+1}}} T_{k_1+1,2}^{\q}
 & = &  \sum_{i_1=1}^{l}\sum_{i_2=1}^{l} \sum_{j=1}^{m}
 \frac{\beta_{i_1}\sqrt{1-t} \|\y^{(i_2)}\|_2T_{k_1,2}^{\q}}{\sqrt{\q_{k_1-1}-\q_{k_1}}} \nonumber \\
 &  & - (1-t) \Bigg( \Bigg. s\beta^2 \lp \m_{k_1} - \m_{k_1+1} \rp \nonumber \\
 & &
\times  \mE_{G,{\mathcal U}_{r+1}} \left \langle   (\x^{(p_1)})^T \x^{(i_1)} \|\y^{(i_2)}\|_2\|\y^{(p_2)}\|_2  \right \rangle_{\gamma_{k_1+1}^{(r)} } \Bigg.\Bigg), \nonumber \\
\end{eqnarray}
and for sequence $\lp T_{k,3}^{(\p,\q)}\rp_{k=1:r+1}$
 \begin{eqnarray}\label{eq:rthlev2genanal29}
 \sum_{i_1=1}^{l}\sum_{i_2=1}^{l}
 \frac{\beta_{i_1}\sqrt{t} \|\y^{(i_2)}\|_2}{\sqrt{\p_{k_1}\q_{k_1}-\p_{k_1+1}\q_{k_1+1}}} T_{k_1+1,3}^{(\p,\q)}
 & = &   \sum_{i_1=1}^{l}\sum_{i_2=1}^{l}
\frac{\beta_{i_1} \sqrt{t} \|\y^{(i_2)}\|_2}{\sqrt{\p_{k_1-1}\q_{k_1-1}-\p_{k_1}\q_{k_1}}} T_{k_1,3}^{(\p,\q)}
\nonumber \\
 &  & - t\Bigg( \Bigg. s\beta^2 \lp  \m_{k_1} - \m_{k_1+1} \rp \nonumber \\
& & \times   \mE_{G,{\mathcal U}_{r+1}} \left \langle   \|\x^{(i_1)}\|_2\|\x^{(p_1)}\|_2 \|\y^{(i_2)}\|_2\|\y^{(p_2)}\|_2  \right \rangle_{\gamma_{k_1+1}^{(r)} } \Bigg. \Bigg).
\end{eqnarray}

%%%%%%%%%%%%%%%%%%%%%%%%%%%%%%%%%%%%%%%%%%%%%%%%%%%%%%%%%%%%%%%%%%%%%%%%%%%%%%%%%%
%%%%%%%%%%%%%%%%%%%%%%%%%%%%%%%%%%%%%%%%%%%%%%%%%%%%%%%%%%%%%%%%%%%%%%%%%%%%%%%%%%
%%%%%%%%%%%%%%%%%%%%%%%%%%%%%%%%%%%%%%%%%%%%%%%%%%%%%%%%%%%%%%%%%%%%%%%%%%%%%%%%%%
%%%%%%%%%%%%%%%%%%%%%%%%%%%%%%%%%%%%%%%%%%%%%%%%%%%%%%%%%%%%%%%%%%%%%%%%%%%%%%%%%%
\subsection{Connecting  everything  together -- $r$-th level}
\label{sec:rthall}
%%%%%%%%%%%%%%%%%%%%%%%%%%%%%%%%%%%%%%%%%%%%%%%%%%%%%%%%%%%%%%%%%%%%%%%%%%%%%%%%%%
%%%%%%%%%%%%%%%%%%%%%%%%%%%%%%%%%%%%%%%%%%%%%%%%%%%%%%%%%%%%%%%%%%%%%%%%%%%%%%%%%%
%%%%%%%%%%%%%%%%%%%%%%%%%%%%%%%%%%%%%%%%%%%%%%%%%%%%%%%%%%%%%%%%%%%%%%%%%%%%%%%%%%
%%%%%%%%%%%%%%%%%%%%%%%%%%%%%%%%%%%%%%%%%%%%%%%%%%%%%%%%%%%%%%%%%%%%%%%%%%%%%%%%%%

The above results are summarized in the following theorem.
\begin{theorem}
\label{thm:thm3}
For a positive integer $r\in\mN$ and $k\in\{1,2,\dots,r+1\}$ let vectors $\m=[\m_0,\m_1,\m_2,...,\m_r,\m_{r+1}]$,
$\p=[\p_0,\p_1,...,\p_r,\p_{r+1}]$, and $\q=[\q_0,\q_1,\q_2,\dots,\q_r,\q_{r+1}]$  be such that $\m_0=1$, $\m_{r+1}=0$, $1\geq\p_0\geq \p_1\geq \p_2\geq \dots \geq \p_r\geq \p_{r+1} =0$, and $1\geq\q_0\geq \q_1\geq \q_2\geq \dots \geq \q_r\geq \q_{r+1} = 0$. Assume that the elements of $G\in\mR^{m\times n}$, $u^{(4,k)}\in\mR$, $\u^{(2,k)}\in\mR^m$, and $\h^{(k)}\in\mR^n$ are independent standard normals and set $a_k=\sqrt{\p_{k-1}\q_{k-1}-\p_k\q_k}$, $b_k=\sqrt{\p_{k-1}-\p_{k}}$, $c_k=\sqrt{\q_{k-1}-\q_{k}}$, and ${\mathcal U}_k\triangleq [u^{(4,k)},\u^{(2,k)},\h^{(2k)}]$. Also, assume that sets ${\mathcal X}$, $\bar{{\mathcal X}}$, and ${\mathcal Y}$, scalars  $\beta$, $p$, and $s$, and function $f_{\bar{\x}^{(i_3)} } (\cdot)$ are as in Proposition \ref{thm:thm1} and consider the following function
\begin{equation}\label{eq:thm3eq1}
\psi(t)  =  \mE_{G,{\mathcal U}_{r+1}} \frac{1}{p|s|\sqrt{n}\m_r} \log
\lp \mE_{{\mathcal U}_{r}} \lp \dots \lp \mE_{{\mathcal U}_2}\lp\lp    \sum_{i_3=1}^{l} \lp \mE_{{\mathcal U}_1}  Z_{i_3}^{\m_1}\rp^p\rp^{\frac{\m_2}{\m_1}}\rp\rp^{\frac{\m_3}{\m_2}} \dots \rp^{\frac{\m_{r}}{\m_{r-1}}}\rp,
\end{equation}
where
\begin{eqnarray}\label{eq:thm3eq2}
Z_{i_3} & \triangleq & \sum_{i_1=1}^{l}\lp\sum_{i_2=1}^{l}e^{\beta D_0^{(i_1,i_2,i_3)}} \rp^{s} \nonumber \\
 D_0^{(i_1,i_2,i_3)} & \triangleq & \sqrt{t}(\y^{(i_2)})^T
 G\x^{(i_1)}+\sqrt{1-t}\|\x^{(i_2)}\|_2 (\y^{(i_2)})^T\lp\sum_{k=1}^{r+1}b_k\u^{(2,k)}\rp\nonumber \\
 & & +\sqrt{t}\|\x^{(i_1)}\|_2\|\y^{(i_2)}\|_2\lp\sum_{k=1}^{r+1}a_ku^{(4,k)}\rp +\sqrt{1-t}\|\y^{(i_2)}\|_2\lp\sum_{k=1}^{r+1}c_k\h^{(k)}\rp^T\x^{(i)}
 + f_{\bar{\x}^{(i_3)} } (\x^{(i_1)}). \nonumber \\
 \end{eqnarray}
For  $\zeta_k$ given in (\ref{eq:rthlev2genanal7a}) and (\ref{eq:rthlev2genanal7b})   set
\begin{eqnarray}\label{eq:thm3eq3}
 \Phi_{{\mathcal U}_1}^{(i_3)} & \triangleq &  \mE_{{\mathcal U}_{1}} \frac{  Z_{i_3}^{\m_1}  } {  \mE_{{\mathcal U}_{1}} Z_{i_3}^{\m_1}  } \nonumber \\
 \Phi_{{\mathcal U}_k} & \triangleq &  \mE_{{\mathcal U}_{k}} \frac{\zeta_{k-1}^{\frac{\m_k}{\m_{k-1}}}}{\zeta_k} ,k\geq 2 \nonumber \\
 \gamma_{00}(i_3) & = &
\frac{  \lp  \mE_{{\mathcal U}_{1}} Z_{i_3}^{\m_1} \rp^p   }  { \sum_{i_3=1}^{l} \lp  \mE_{{\mathcal U}_{1}} Z_{i_3}^{\m_1} \rp^p   } \nonumber \\
 \gamma_0(i_1,i_2;i_3) & = &
\frac{(C_{i_3}^{(i_1)})^{s}}{Z_{i_3}}  \frac{A_{i_3}^{(i_1,i_2)}}{C_{i_3}^{(i_1)}} \nonumber \\
\gamma_{01}^{(r)}  & = & \prod_{k=r}^{2}\Phi_{{\mathcal U}_k} \gamma_{00}(i_3) \Phi_{{\mathcal U}_1}^{(i_3)} (\gamma_0(i_1,i_2;i_3))  \nonumber \\
\gamma_{02}^{(r)}  & = & \prod_{k=r}^{2} \Phi_{{\mathcal U}_k}   \gamma_{00}(i_3)    \Phi_{{\mathcal U}_1}^{(i_3)}   (\gamma_0(i_1,i_2;i_3)\times \gamma_0(i_1,p_2;i_3)) \nonumber \\
 \gamma_{1}^{(r)}  & = & \prod_{k=r}^{2} \Phi_{{\mathcal U}_k}   \gamma_{00}(i_3)    \Phi_{{\mathcal U}_1}^{(i_3)}   (\gamma_0(i_1,i_2;i_3)\times \gamma_0(p_1,p_2;i_3)) \nonumber \\
\gamma_{2}^{(r)} \triangleq \gamma_{21}^{(r)}  & = & \prod_{k=r}^{2}\Phi_{{\mathcal U}_k} \lp
  \gamma_{00}(i_3)    \Phi_{{\mathcal U}_1}^{(i_3)}   \gamma_0(i_1,i_2;i_3)\times
    \gamma_{00}(p_3)    \Phi_{{\mathcal U}_1}^{(p_3)}
  \gamma_0(p_1,p_2;p_3) \rp  \nonumber \\
\gamma_{22}^{(r)}  & = & \prod_{k=r}^{2}\Phi_{{\mathcal U}_k} \lp
  \gamma_{00}(i_3)  \lp  \Phi_{{\mathcal U}_1}^{(i_3)}   \gamma_0(i_1,i_2;i_3)\times
     \Phi_{{\mathcal U}_1}^{(i_3)}
  \gamma_0(p_1,p_2;i_3) \rp \rp \nonumber \\
\gamma_{k_1+1}^{(r)}  & = & \prod_{k=r}^{k_1+1}\Phi_{{\mathcal U}_k} \lp \prod_{k=k_1}^{2}\Phi_{{\mathcal U}_k}
  \gamma_{00}(i_3)    \Phi_{{\mathcal U}_1}^{(i_3)}   \gamma_0(i_1,i_2;i_3)\times \prod_{k=k_1}^{2} \Phi_{{\mathcal U}_k}
    \gamma_{00}(p_3)    \Phi_{{\mathcal U}_1}^{(p_3)}
  \gamma_0(p_1,p_2;p_3) \rp, k_1\geq 2. \nonumber \\
 \end{eqnarray}
Also, let
\begin{eqnarray}\label{eq:thm3eq41aa0}
\phi^{(1,\p)} & = &
  -(1-t)\lp \m_{1}-\m_{2}\rp p \mE_{G,{\mathcal U}_{r+1}} \langle \|\x^{(i_1)}\|_2\|\x^{(p_1)}\|_2(\y^{(p_2)})^T\y^{(i_2)} \rangle_{\gamma_{21}^{(r)}} \nonumber \\
& &   - t \q_{k_1}
\lp \m_{1}-\m_{2}\rp p \mE_{G,{\mathcal U}_{r+1}} \langle\|\x^{(i_1)}\|_2\|\x^{(p_1)}\|_2\|\y^{(i_2)}\|_2\|\y^{(p_2)}\|_2\rangle_{\gamma_{21}^{(r)}}\nonumber \\
&  &
  +(1-t) \m_{1}  (p-1)\mE_{G,{\mathcal U}_{r+1}} \langle \|\x^{(i_1)}\|_2\|\x^{(p_1)}\|_2(\y^{(p_2)})^T\y^{(i_2)} \rangle_{\gamma_{22}^{(r)}} \nonumber \\
& &   + t \q_{k_1}
\m_1 (p-1) \mE_{G,{\mathcal U}_{r+1}} \langle\|\x^{(i_1)}\|_2\|\x^{(p_1)}\|_2\|\y^{(i_2)}\|_2\|\y^{(p_2)}\|_2\rangle_{\gamma_{22}^{(r)}}
\nonumber \\
\phi^{(1,\q)} & = &
   -(1-t)  \lp \m_{1}-\m_{2}\rp p \mE_{G,{\mathcal U}_{r+1}} \langle \|\y^{(i_2)}\|_2\|\y^{(p_2)}\|_2(\x^{(i_1)})^T\x^{(p_1)}\rangle_{\gamma_{21}^{(r)}} \nonumber \\
& &   - t \p_{k_1}
\lp \m_{1}-\m_{2}\rp p  \mE_{G,{\mathcal U}_{r+1}} \langle\|\x^{(i_1)}\|_2\|\x^{(p_1)}\|_2\|\y^{(i_2)}\|_2\|\y^{(p_2)}\|_2\rangle_{\gamma_{21}^{(r)}}
\nonumber \\
& &
   +(1-t)  \m_{1} (p-1) \mE_{G,{\mathcal U}_{r+1}} \langle \|\y^{(i_2)}\|_2\|\y^{(p_2)}\|_2(\x^{(i_1)})^T\x^{(p_1)}\rangle_{\gamma_{22}^{(r)}} \nonumber \\
& &   + t \p_{k_1} \m_{1}  (p-1)  \mE_{G,{\mathcal U}_{r+1}} \langle\|\x^{(i_1)}\|_2\|\x^{(p_1)}\|_2\|\y^{(i_2)}\|_2\|\y^{(p_2)}\|_2\rangle_{\gamma_{22}^{(r)}}.
\end{eqnarray}
and for $k_1\geq 2$
\begin{eqnarray}\label{eq:thm3eq41}
\phi^{(k_1,\p)} & = &
  -(1-t)\lp \m_{k_1}-\m_{k_1+1}\rp p \mE_{G,{\mathcal U}_{r+1}} \langle \|\x^{(i_1)}\|_2\|\x^{(p_1)}\|_2(\y^{(p_2)})^T\y^{(i_2)} \rangle_{\gamma_{k_1+1}^{(r)}} \nonumber \\
& &   - t \q_{k_1}
\lp \m_{k_1}-\m_{k_1+1}\rp p \mE_{G,{\mathcal U}_{r+1}} \langle\|\x^{(i_1)}\|_2\|\x^{(p_1)}\|_2\|\y^{(i_2)}\|_2\|\y^{(p_2)}\|_2\rangle_{\gamma_{k_1+1}^{(r)}}\nonumber \\
\phi^{(k_1,\q)} & = &
   -(1-t)  \lp \m_{k_1}-\m_{k_1+1}\rp  p \mE_{G,{\mathcal U}_{r+1}} \langle \|\y^{(i_2)}\|_2\|\y^{(p_2)}\|_2(\x^{(i_1)})^T\x^{(p_1)}\rangle_{\gamma_{k_1+1}^{(r)}} \nonumber \\
& &   - t \p_{k_1}
\lp \m_{k_1}-\m_{k_1+1}\rp  p  \mE_{G,{\mathcal U}_{r+1}} \langle\|\x^{(i_1)}\|_2\|\x^{(p_1)}\|_2\|\y^{(i_2)}\|_2\|\y^{(p_2)}\|_2\rangle_{\gamma_{k_1+1}^{(r)}}.
\end{eqnarray}
Then for $k_1\in\{1,2,\dots,r\}$
\begin{eqnarray}\label{eq:thm3eq42}
\frac{d\psi(t)}{d\p_{k_1}}  & = &       \frac{\mbox{sign}(s)s\beta^2}{2\sqrt{n}} \phi^{(k_1,\p)}\nonumber \\
\frac{d\psi(t)}{d\q_{k_1}}  & = &       \frac{\mbox{sign}(s)s\beta^2}{2\sqrt{n}} \phi^{(k_1,\q)}.
 \end{eqnarray}
 \end{theorem}
\begin{proof}  The proofs follow through the results presented  in earlier sections and \cite{Stojnicnflldp25}. We distinguish two cases:  (i) $k_1>1 $ and (ii) $k_1=1$.

\vspace{.1in}

\noindent \underline{\textbf{\emph{(i) $k_1 > 1$ case:}}}  It  follows automatically from the previous discussion after one combines  (\ref{eq:rthlev2genanal10e})-(\ref{eq:rthlev2genanal10f}), (\ref{eq:rthlev2genanal10e1})-(\ref{eq:rthlev2genanal10f1}), and (\ref{eq:rthlev2genanal27})-(\ref{eq:rthlev2genanal29}).

\vspace{.1in}

\noindent \underline{\textbf{\emph{(i) $k_1 = 1$ case:}}}  The proof follows through a direct extension of the first level  results  presented in Section \ref{sec:gencon}. Second level generalization  presented in  \cite{Stojnicnflldp25}'s Section 3 contains structurally identical forms that are needed for such an extension.  From (\ref{eq:rthlev2genanal10e})-(\ref{eq:rthlev2genanal10f}), (\ref{eq:rthlev2genanal10e1})-(\ref{eq:rthlev2genanal10f1}), and (\ref{eq:rthlev2genanal27})-(\ref{eq:rthlev2genanal29}), one has that for $k_1=1$ three  pairs ($T_{k,1,j}^{\p}, k\in\{1,2\}$, $T_{k,2}^{\q}, k\in\{1,2\}$, and $T_{k,3}^{\p}, k\in\{1,2\}$)  need to be determined. As the strategies are identical for each of them we show the details for  $T_{k,1,j}^{\p}, k\in\{1,2\}$ and $T_{k,3}^{\p}, k\in\{1,2\}$ which is sufficient to prove that $\frac{d\psi(t)}{d\p_1}$ indeed has the form given in (\ref{eq:thm3eq41aa0}). The proofs for $T_{k,2}^{\q}, k\in\{1,2\}$ and consequently $\frac{d\psi(t)}{d\q_1}$ follow in a completely analogous manner.

\vspace{.1in}

\noindent \underline{\textbf{\emph{Determining $T_{1,1,j}^{\p}$ on the $r$-th level}}}

We first note that both $T_{1,1,j}^{\p}$ and $T_{2,1,j}^{\p}$ are already determined on the first level. $T_{1,1,j}^{\p}$ is determined in (\ref{eq:liftgenAanal19i}) and  $T_{2,1,j}^{\p}$ in (\ref{eq:genDanal25}). We would like to show that the same forms remain in place at the $r$-th level with $\gamma$'s from the first level, $\gamma_{\cdot}^{(1)}$'s from (\ref{eq:genAanal19e}) and (\ref{eq:genAanal19e1}), replaced by the $\gamma$'s from the $r$-th level, $\gamma_{\cdot}^{(r)}$'s from (\ref{eq:thm3eq3}).

From (\ref{eq:rthlev2genanal10g}), we first recall on
 \begin{eqnarray}\label{eq:proofeq1}
 T_{1,1,j}^{\p} & = &   \mE_{G,{\mathcal U}_{r+1}} \lp
\zeta_r^{-1}\prod_{v=r}^{2}\mE_{{\mathcal U}_{v}} \zeta_{v-1}^{\frac{\m_v}{\m_{v-1}}-1}
\sum_{i_3=1}^{l} \lp \mE_{{\mathcal U}_1} Z_{i_3}^{\m_1}  \rp^{p-1}
  \mE_{{\mathcal U}_1}\frac{(C_{i_3}^{(i_1)})^{s-1} A_{i_3}^{(i_1,i_2)} \y_j^{(i_2)}\u_j^{(2,1)}}{Z_{i_3}^{1-\m_1}} \rp.
 \end{eqnarray}
We then follow the strategy of Section \ref{sec:hand1T11}. Applying Gaussian integration by parts we write analogously to (\ref{eq:liftgenAanal19}) (and to \cite{Stojnicnflldp25}'s (20))
 \begin{eqnarray}\label{eq:proofliftgenAanal19}
T_{1,1,j}^{\p} & = &
   \mE_{G,{\mathcal U}_{r+1}} \lp
\zeta_r^{-1}\prod_{v=r}^{2}\mE_{{\mathcal U}_{v}} \zeta_{v-1}^{\frac{\m_v}{\m_{v-1}}-1}
\sum_{i_3=1}^{l} \lp \mE_{{\mathcal U}_1} Z_{i_3}^{\m_1}  \rp^{p-1}
  \mE_{{\mathcal U}_1}\frac{(C_{i_3}^{(i_1)})^{s-1} A_{i_3}^{(i_1,i_2)} \y_j^{(i_2)}\u_j^{(2,1)}}{Z_{i_3}^{1-\m_1}} \rp
  \nonumber \\
 & = & \mE_{G,{\mathcal U}_{r+1}} \Bigg (\Bigg.
\zeta_r^{-1}\prod_{v=r}^{2}\mE_{{\mathcal U}_{v}} \zeta_{v-1}^{\frac{\m_v}{\m_{v-1}}-1}
\sum_{i_3=1}^{l} \lp \mE_{{\mathcal U}_1} Z_{i_3}^{\m_1}  \rp^{p-1}
\nonumber \\
& & \times
\mE_{{\mathcal U}_1}\lp \mE_{{\mathcal U}_1}(\u_j^{(2,1)}\u_j^{(2,1)}) \frac{d}{d\u_j^{(2,1)}}\lp\frac{(C_{i_3}^{(i_1)})^{s-1} A_{i_3}^{(i_1,i_2)}\y_j^{(i_2)}}{Z_{i_3}^{1-\m_1}}\rp\rp\Bigg ) \Bigg . \nonumber \\
& = & \mE_{G,{\mathcal U}_{r+1}}   \Bigg ( \Bigg .
\zeta_r^{-1}\prod_{v=r}^{2}\mE_{{\mathcal U}_{v}} \zeta_{v-1}^{\frac{\m_v}{\m_{v-1}}-1}
\sum_{i_3=1}^{l} \lp \mE_{{\mathcal U}_1} Z_{i_3}^{\m_1}  \rp^{p-1}
\nonumber \\
& & \times
\mE_{{\mathcal U}_1}(\u_j^{(2,1)}\u_j^{(2,1)})\mE_{{\mathcal U}_1}\lp  \frac{d}{d\u_j^{(2,1)}}\lp\frac{(C_{i_3}^{(i_1)})^{s-1} A_{i_3}^{(i_1,i_2)}\y_j^{(i_2)}}{Z_{i_3}^{1-\m_1}}\rp\rp\Bigg ) \Bigg. \nonumber \\
& = & \mE_{G,{\mathcal U}_{r+1}} \Bigg (\Bigg.
\zeta_r^{-1}\prod_{v=r}^{2}\mE_{{\mathcal U}_{v}} \zeta_{v-1}^{\frac{\m_v}{\m_{v-1}}-1}
\sum_{i_3=1}^{l} \lp \mE_{{\mathcal U}_1} Z_{i_3}^{\m_1}  \rp^{p-1}
 \mE_{{\mathcal U}_1}\lp  \frac{d}{d\u_j^{(2,1)}}\lp\frac{(C_{i_3}^{(i_1)})^{s-1} A_{i_3}^{(i_1,i_2)}\y_j^{(i_2)}}{Z_{i_3}^{1-\m_1}}\rp\rp\Bigg ) \Bigg..
 \nonumber \\
\end{eqnarray}
Similarly to (\ref{eq:liftgenAanal19a}), the above can be rewritten as
\begin{eqnarray}\label{eq:proofliftgenAanal19a}
T_{1,1,j}^{\p} & = &    \sqrt{\p_0-\p_1}\mE_{G,{\mathcal U}_{r+1}} \lp
\zeta_r^{-1}\prod_{v=r}^{2}\mE_{{\mathcal U}_{v}} \zeta_{v-1}^{\frac{\m_v}{\m_{v-1}}-1}
\sum_{i_3=1}^{l} \lp \mE_{{\mathcal U}_1} Z_{i_3}^{\m_1}  \rp^{p-1}
\lp \Theta_1+\Theta_2 \rp\rp,
\end{eqnarray}
with $\Theta_1$ and $\Theta_2$ as in (\ref{eq:liftgenAanal19c}) (and \cite{Stojnicnflldp25}'s (22))
{\small\begin{eqnarray}\label{eq:proofliftgenAanal19c}
\Theta_1 &  = &  \mE_{{\mathcal U}_1} \Bigg( \Bigg. \frac{\y_j^{(i_2)} \lp (C_{i_3}^{(i_1)})^{s-1}\beta_{i_1}A_{i_3}^{(i_1,i_2)}\y_j^{(i_2)}\sqrt{1-t} +A_{i_3}^{(i_1,i_2)}(s-1)(C_{i_3}^{(i_1)})^{s-2}\beta_{i_1}\sum_{p_2=1}^{l}A_{i_3}^{(i_1,p_2)}\y_j^{(p_2)}\sqrt{1-t}\rp}{Z_{i_3}^{1-\m_1}}\Bigg. \Bigg)\Bigg. \Bigg) \nonumber \\
\Theta_2 & = & -(1-\m_1)\mE_{{\mathcal U}_1} \lp\sum_{p_1=1}^{l}
\frac{(C_{i_3}^{(i_1)})^{s-1} A_{i_3}^{(i_1,i_2)}\y_j^{(i_2)}}{Z_{i_3}^{2-\m_1}}
s  (C_{i_3}^{(p_1)})^{s-1}\sum_{p_2=1}^{l}\beta_{p_1}A_{i_3}^{(p_1,p_2)}\y_j^{(p_2)}\sqrt{1-t}\rp\Bigg.\Bigg).
\end{eqnarray}}
One then observes the following
\begin{eqnarray}\label{eq:proofliftgenAanal19g}
\sum_{i_1=1}^{l}\sum_{i_2=1}^{l}\sum_{j=1}^{m}
\hspace{-.3in}
& &
\lp
 \zeta_r^{-1}
 \prod_{v=r}^{2}\mE_{{\mathcal U}_{v}} \zeta_{v-1}^{\frac{\m_v}{\m_{v-1}}-1}
\sum_{i_3=1}^{l} \lp \mE_{{\mathcal U}_1} Z_{i_3}^{\m_1}  \rp^{p-1}
 \frac{\beta_{i_1}\Theta_1}{\sqrt{1-t}}
 \rp
    =
    \nonumber \\
    &  = &
    \Bigg ( \Bigg.
 \prod_{v=r}^{2}\mE_{{\mathcal U}_{v}} \frac{\zeta_{v-1}^{\frac{\m_v}{\m_{v-1}}} } { \zeta_v}
  \sum_{i_3=1}^{l} \frac{
\lp \mE_{{\mathcal U}_1} Z_{i_3}^{\m_1} \rp^{p}    }{ \lp
\sum_{i_3=1}^{l}
\lp \mE_{{\mathcal U}_1} Z_{i_3}^{\m_1} \rp^p \rp   }
 \mE_{{\mathcal U}_1}\frac{Z_{i_3}^{\m_1}}{\mE_{{\mathcal U}_1} Z_{i_3}^{\m_1}}
  \sum_{i_1=1}^{l}\frac{(C_{i_3}^{(i_1)})^s}{Z_{i_3}}\sum_{i_2=1}^{l}\frac{A_{i_3}^{(i_1,i_2)}}{C_{i_3}^{(i_1)}}\beta_{i_1}^2\|\y^{(i_2)}\|_2^2
 \Bigg ) \Bigg.
  \nonumber\\
& & +  \Bigg ( \Bigg.
 \prod_{v=r}^{2}\mE_{{\mathcal U}_{v}} \frac{\zeta_{v-1}^{\frac{\m_v}{\m_{v-1}}} } { \zeta_v}
  \sum_{i_3=1}^{l} \frac{
\lp \mE_{{\mathcal U}_1} Z_{i_3}^{\m_1} \rp^{p}    }{ \lp
\sum_{i_3=1}^{l}
\lp \mE_{{\mathcal U}_1} Z_{i_3}^{\m_1} \rp^p \rp   }
 \mE_{{\mathcal U}_1}\frac{Z_{i_3}^{\m_1}}{\mE_{{\mathcal U}_1} Z_{i_3}^{\m_1}}
  \sum_{i_1=1}^{l}\frac{(s-1)(C_{i_3}^{(i_1)})^s}{Z_{i_3}}
  \nonumber \\
  & & \times
  \sum_{i_2=1}^{l}\sum_{p_2=1}^{l}\frac{A_{i_3}^{(i_1,i_2)}A_{i_3}^{(i_1,p_2)}}{(C_{i_3}^{(i_1)})^2}
   \beta_{i_1}^2(\y^{(p_2)})^T\y^{(i_2)}
 \Bigg ) \Bigg.
 \nonumber \\
 &  = & \beta^2 \Bigg ( \Bigg. \langle \|\x^{(i_1)}\|_2^2\|\y^{(i_2)}\|_2^2\rangle_{\gamma_{01}^{(r)}}
 +  (s-1) \langle \|\x^{(i_1)}\|_2^2(\y^{(p_2)})^T\y^{(i_2)}\rangle_{\gamma_{02}^{(r)}} \Bigg ) \Bigg. .
 \end{eqnarray}
 We then analogously have for the $\Theta_2$ counterpart
\begin{eqnarray}\label{eq:proofliftgenAanal19h}
\sum_{i_1=1}^{l}\sum_{i_2=1}^{l}\sum_{j=1}^{m}
\hspace{-.3in}
& &
\lp
 \zeta_r^{-1}
 \prod_{v=r}^{2}\mE_{{\mathcal U}_{v}} \zeta_{v-1}^{\frac{\m_v}{\m_{v-1}}-1}
\sum_{i_3=1}^{l} \lp \mE_{{\mathcal U}_1} Z_{i_3}^{\m_1}  \rp^{p-1}
 \frac{\beta_{i_1}\Theta_2}{\sqrt{1-t}}
 \rp
    =
    \nonumber \\
    &  = &
-s(1-\m_1) \mE_{G,{\mathcal U}_{r+1}} \Bigg( \Bigg.
 \prod_{v=r}^{2}\mE_{{\mathcal U}_{v}} \frac{\zeta_{v-1}^{\frac{\m_v}{\m_{v-1}}} } { \zeta_v}
 \sum_{i_3=1}^{l} \frac{
\lp \mE_{{\mathcal U}_1} Z_{i_3}^{\m_1} \rp^{p}    }{ \lp
\sum_{i_3=1}^{l}
\lp \mE_{{\mathcal U}_1} Z_{i_3}^{\m_1} \rp^p \rp   }
\nonumber \\
& &
 \hspace{-.05in} \times
 \frac{Z_{i_3}^{\m_1}}{\mE_{{\mathcal U}_1} Z_{i_3}^{\m_1}} \sum_{i_1=1}^{l}\frac{(C_{i_3}^{(i_1)})^s}{Z_{i_3}}\sum_{i_2=1}^{l}
\frac{A_{i_3}^{(i_1,i_2)}}{C_{i_3}^{(i_1)}}
 \sum_{p_1=1}^{l} \frac{(C_{i_3}^{(p_1)})^s}{Z_{i_3}}\sum_{p_2=1}^{l}\frac{A_{i_3}^{(p_1,p_2)}}{C_{i_3}^{(p_1)}} \beta_{i_1}\beta_{p_1}(\y^{(p_2)})^T\y^{(i_2)} \Bigg.\Bigg)\nonumber \\
& =&  \hspace{-.05in} -s\beta^2(1-\m_1) \mE_{G,{\mathcal U}_{r+1}} \langle \|\x^{(i_1)}\|_2\|\x^{(p_1)}\|_2(\y^{(p_2)})^T\y^{(i_2)} \rangle_{\gamma_{1}^{(r)}}.
\end{eqnarray}
An additional combination of  (\ref{eq:proofliftgenAanal19a}), (\ref{eq:proofliftgenAanal19g}), and (\ref{eq:proofliftgenAanal19h}) gives
\begin{eqnarray}\label{eq:proofliftgenAanal19i}
\sum_{i_1=1}^{l}\sum_{i_2=1}^{l}\sum_{j=1}^{m}
\hspace{-.3in}
& &
\beta_{i_1}\frac{\sqrt{1-t}}{\sqrt{\p_0-\p_1}}T_{1,1,j}^{\p}
=
\nonumber \\
& = &
 (1-t) \mE_{G,{\mathcal U}_{r+1}} \lp
\zeta_r^{-1}\prod_{v=r}^{2}\mE_{{\mathcal U}_{v}} \zeta_{v-1}^{\frac{\m_v}{\m_{v-1}}-1}
\sum_{i_3=1}^{l} \lp \mE_{{\mathcal U}_1} Z_{i_3}^{\m_1}  \rp^{p-1}
\lp \frac{\beta_{i_1}\Theta_1}{\sqrt{1-t}}+\frac{\beta_{i_1}\Theta_2}{\sqrt{1-t}} \rp\rp\nonumber \\
& = &  (1-t) \beta^2  \Bigg ( \Bigg .
\mE_{G,{\mathcal U}_{r+1}} \langle \|\x^{(i_1)}\|_2^2\|\y^{(i_2)}\|_2^2\rangle_{\gamma_{01}^{(r)}}
  +  (s-1)\mE_{G,{\mathcal U}_{r+1}}\langle \|\x^{(i_1)}\|_2^2(\y^{(p_2)})^T\y^{(i_2)}\rangle_{\gamma_{02}^{(r)}} \Bigg ) \Bigg . \nonumber \\
& & -  (1-t) s\beta^2(1-\m_1)\langle \|\x^{(i_1)}\|_2\|\x^{(p_1)}\|_2(\y^{(p_2)})^T\y^{(i_2)} \rangle_{\gamma_{1}^{(r)}}.
 \end{eqnarray}

 \vspace{.1in}

\noindent \underline{\textbf{\emph{Determining $T_{2,1,j}^{\p}$ on the $r$-th level}}}

From (\ref{eq:rthlev2genanal10g}), we first recall on
 \begin{eqnarray}\label{eq:proofxeq1}
 T_{2,1,j}^{\p} & = &   \mE_{G,{\mathcal U}_{r+1}} \lp
\zeta_r^{-1}\prod_{v=r}^{2}\mE_{{\mathcal U}_{v}} \zeta_{v-1}^{\frac{\m_v}{\m_{v-1}}-1}
\sum_{i_3=1}^{l} \lp \mE_{{\mathcal U}_1} Z_{i_3}^{\m_1}  \rp^{p-1}
  \mE_{{\mathcal U}_1}\frac{(C_{i_3}^{(i_1)})^{s-1} A_{i_3}^{(i_1,i_2)} \y_j^{(i_2)}\u_j^{(2,2)}}{Z_{i_3}^{1-\m_1}} \rp.
 \end{eqnarray}
We then follow the strategy of \cite{Stojnicnflldp25}'s Section 3.1.2 (part 1). Applying Gaussian integration by parts we have
\begin{eqnarray}\label{eq:proofxlev2genDanal19}
T_{2,1,j}^{\p} & = &   \mE_{G,{\mathcal U}_{r+1}} \lp
\zeta_r^{-1}\prod_{v=r}^{2}\mE_{{\mathcal U}_{v}} \zeta_{v-1}^{\frac{\m_v}{\m_{v-1}}-1}
\sum_{i_3=1}^{l} \lp \mE_{{\mathcal U}_1} Z_{i_3}^{\m_1}  \rp^{p-1}
  \mE_{{\mathcal U}_1}\frac{(C_{i_3}^{(i_1)})^{s-1} A_{i_3}^{(i_1,i_2)} \y_j^{(i_2)}\u_j^{(2,2)}}{Z_{i_3}^{1-\m_1}} \rp \nonumber \\
  & = &
\mE_{G,{\mathcal U}_{r+1}} \Bigg ( \Bigg.
\zeta_r^{-1}\prod_{v=r}^{3}\mE_{{\mathcal U}_{v}} \zeta_{v-1}^{\frac{\m_v}{\m_{v-1}}-1}
\sum_{i_3=1}^{l} \lp \mE_{{\mathcal U}_1} Z_{i_3}^{\m_1}  \rp^{p-1}
\nonumber \\
& &
\times
\mE_{{\mathcal U}_2}\lp\mE_{{\mathcal U}_2} (\u_j^{(2,2)}\u_j^{(2,2)})\frac{d}{d\u_j^{(2,2)}}\lp \frac{(C_{i_3}^{(i_1)})^{s-1} A_{i_3}^{(i_1,i_2)}\y_j^{(i_2)}}{Z_{i_3}^{1-\m_1}
 \zeta_1^{1-\frac{\m_2}{\m_1}} \lp
 \mE_{{\mathcal U}_1}  Z_{i_3}^{\m_1}
\rp^{1-p}   }
\rp\rp
\Bigg . \Bigg )
 \nonumber \\
& = &
 \sqrt{\p_1-\p_2} \mE_{G,{\mathcal U}_{r+1}} \Bigg ( \Bigg .
\zeta_r^{-1}\prod_{v=r}^{2}\mE_{{\mathcal U}_{v}} \zeta_{v-1}^{\frac{\m_v}{\m_{v-1}}-1}
\sum_{i_3=1}^{l} \lp \mE_{{\mathcal U}_1} Z_{i_3}^{\m_1}  \rp^{p-1}
\nonumber \\
& &
\times
 \frac{d}{d \lp \sqrt{\p_1-\p_2}\u_j^{(2,2)}\rp  }\lp \frac{(C_{i_3}^{(i_1)})^{s-1} A_{i_3}^{(i_1,i_2)}\y_j^{(i_2)}}{Z_{i_3}^{1-\m_1}}\rp\Bigg . \Bigg ) \nonumber \\
&  & +
 \sqrt{\p_1-\p_2} \mE_{G,{\mathcal U}_{r+1}} \Bigg ( \Bigg .
\zeta_r^{-1}\prod_{v=r}^{3}\mE_{{\mathcal U}_{v}} \zeta_{v-1}^{\frac{\m_v}{\m_{v-1}}-1}
\sum_{i_3=1}^{l} \lp \mE_{{\mathcal U}_1} Z_{i_3}^{\m_1}  \rp^{p-1}
\frac{(C_{i_3}^{(i_1)})^{s-1} A_{i_3}^{(i_1,i_2)}\y_j^{(i_2)}}{Z_{i_3}^{1-\m_1}}
\nonumber \\
& &
\times
 \frac{d}{d  \lp \sqrt{\p_1-\p_2}\u_j^{(2,2)}\rp  }\lp
\zeta_1^{\frac{\m_2}{\m_1} -1 }             \rp\Bigg . \Bigg ) \nonumber \\
&  & +
 \sqrt{\p_1-\p_2} \mE_{G,{\mathcal U}_{r+1}} \Bigg ( \Bigg .
\zeta_r^{-1}\prod_{v=r}^{2}\mE_{{\mathcal U}_{v}} \zeta_{v-1}^{\frac{\m_v}{\m_{v-1}}-1}
\sum_{i_3=1}^{l}
\frac{(C_{i_3}^{(i_1)})^{s-1} A_{i_3}^{(i_1,i_2)}\y_j^{(i_2)}}{Z_{i_3}^{1-\m_1}}
\nonumber \\
& &
\times
 \frac{d}{d  \lp \sqrt{\p_1-\p_2}\u_j^{(2,2)} \rp }\lp
\lp
 \mE_{{\mathcal U}_1}  Z_{i_3}^{\m_1}
\rp^{p-1}
           \rp\Bigg . \Bigg ) .
 \end{eqnarray}
We split $T_{2,1,j}^{\p}$ into three components
\begin{eqnarray}\label{eq:proofxlev2genDanal19a}
T_{2,1,j}^{\p}   =   T_{2,1,j}^{\p,c} +  T_{2,1,j}^{\p,d} +  T_{2,1,j}^{\p,e},
\end{eqnarray}
where
\begin{eqnarray}\label{eq:proofxlev2genDanal19b}
T_{2,1,j}^{\p,c}
& = &
 \sqrt{\p_1-\p_2} \mE_{G,{\mathcal U}_{r+1}} \Bigg ( \Bigg .
\zeta_r^{-1}\prod_{v=r}^{2}\mE_{{\mathcal U}_{v}} \zeta_{v-1}^{\frac{\m_v}{\m_{v-1}}-1}
\sum_{i_3=1}^{l} \lp \mE_{{\mathcal U}_1} Z_{i_3}^{\m_1}  \rp^{p-1}
\nonumber \\
& &
\times
 \frac{d}{d \lp \sqrt{\p_1-\p_2}\u_j^{(2,2)}\rp  }\lp \frac{(C_{i_3}^{(i_1)})^{s-1} A_{i_3}^{(i_1,i_2)}\y_j^{(i_2)}}{Z_{i_3}^{1-\m_1}}\rp\Bigg . \Bigg )
 \nonumber \\
 T_{2,1,j}^{\p,d} &  = &
 \sqrt{\p_1-\p_2} \mE_{G,{\mathcal U}_{r+1}} \Bigg ( \Bigg .
\zeta_r^{-1}\prod_{v=r}^{3}\mE_{{\mathcal U}_{v}} \zeta_{v-1}^{\frac{\m_v}{\m_{v-1}}-1}
\sum_{i_3=1}^{l} \lp \mE_{{\mathcal U}_1} Z_{i_3}^{\m_1}  \rp^{p-1}
\frac{(C_{i_3}^{(i_1)})^{s-1} A_{i_3}^{(i_1,i_2)}\y_j^{(i_2)}}{Z_{i_3}^{1-\m_1}}
\nonumber \\
& &
\times
 \frac{d}{d  \lp \sqrt{\p_1-\p_2}\u_j^{(2,2)}\rp  }\lp
\zeta_1^{\frac{\m_2}{\m_1} -1 }             \rp\Bigg . \Bigg ) \nonumber \\
T_{2,1,j}^{\p,e} &  = &
 \sqrt{\p_1-\p_2} \mE_{G,{\mathcal U}_{r+1}} \Bigg ( \Bigg .
\zeta_r^{-1}\prod_{v=r}^{2}\mE_{{\mathcal U}_{v}} \zeta_{v-1}^{\frac{\m_v}{\m_{v-1}}-1}
\sum_{i_3=1}^{l}
\frac{(C_{i_3}^{(i_1)})^{s-1} A_{i_3}^{(i_1,i_2)}\y_j^{(i_2)}}{Z_{i_3}^{1-\m_1}}
\nonumber \\
& &
\times
 \frac{d}{d  \lp \sqrt{\p_1-\p_2}\u_j^{(2,2)} \rp }\lp
\lp
 \mE_{{\mathcal U}_1}  Z_{i_3}^{\m_1}
\rp^{p-1}
           \rp\Bigg . \Bigg ).
 \end{eqnarray}
 Analogously to \cite{Stojnicnflldp25}'s (135)   (see also the first part of \cite{Stojnicnflldp25}'s Section 2.1.2), we have
\begin{eqnarray}\label{eq:proofxlev2genDanal19b1}
\sum_{i_1=1}^{l}\sum_{i_2=1}^{l}\sum_{j=1}^{m} \hspace{-.3in}
& &
\frac{  \beta_{i_1} \sqrt{1-t} }{  \sqrt{\p_1-\p_2} }  T_{2,1,j}^{\p,c}
=
\nonumber \\
 & = & (1-t) \beta^2  \lp \mE_{G,{\mathcal U}_{r+1}}\langle \|\x^{(i_1)}\|_2^2\|\y^{(i_2)}\|_2^2\rangle_{\gamma_{01}^{(r)}} +  (s-1)\mE_{G,{\mathcal U}_{r+1}}\langle \|\x^{(i_1)}\|_2^2(\y^{(p_2)})^T\y^{(i_2)}\rangle_{\gamma_{02}^{(r)}} \rp \nonumber \\
& & - (1-t)s\beta^2(1-\m_1)\mE_{G,{\mathcal U}_{r+1}}\langle \|\x^{(i_1)}\|_2\|\x^{(p_1)}\|_2(\y^{(p_2)})^T\y^{(i_2)} \rangle_{\gamma_{1}^{(r)}}.
\end{eqnarray}
From \cite{Stojnicnflldp25}'s (137) we first have
 \begin{eqnarray}\label{eq:proofxlev2genDanal23}
T_{2,1,j}^{\p,d}
  & = & -s\sqrt{1-t}\sqrt{\p_1-\p_2}(\m_1-\m_2) p \mE_{G,{\mathcal U}_3}\Bigg( \Bigg.
   \prod_{v=r}^{2}\mE_{{\mathcal U}_{v}} \frac{\zeta_{v-1}^{\frac{\m_v}{\m_{v-1}}} } { \zeta_v}
\nonumber \\
& &
\times
  \sum_{i_3=1}^{l}
  \frac{
\lp
 \mE_{{\mathcal U}_1}  Z_{i_3}^{\m_1}
\rp^{p}  }  {   \lp
\sum_{i_3=1}^{l}
\lp
 \mE_{{\mathcal U}_1}  Z_{i_3}^{\m_1}
\rp^p
\rp
 }
\mE_{{\mathcal U}_1}\frac{Z_{i_3}^{\m_1}}{\mE_{{\mathcal U}_1} Z_{i_3}^{\m_1}}
 \frac{(C_{i_3}^{(i_1)})^{s}}{Z_{i_3} }  \frac{A_{i_3}^{(i_1,i_2)}}{C_{i_3}^{(i_1)}}\y_j^{(i_2)} \nonumber \\
& & \times
  \sum_{p_3=1}^{l}
  \frac{
\lp
 \mE_{{\mathcal U}_1}  Z_{p_3}^{\m_1}
\rp^{p}  }  {   \lp
\sum_{p_3=1}^{l}
\lp
 \mE_{{\mathcal U}_1}  Z_{p_3}^{\m_1}
\rp^p
\rp
 }
\lp \mE_{{\mathcal U}_1}  \frac{Z_{p_3}^{\m_1}}{\mE_{{\mathcal U}_1} Z_{p_3}^{\m_1}} \sum_{p_1=1}^{l}  \frac{(C_{p_3}^{(p_1)})^s}{Z_{p_3}}\sum_{p_2=1}^{l}
\frac{A^{(p_1,p_2)}}{C_{p_3}^{(p_1)}}\beta_{p_1}\y_j^{(p_2)}  \rp\Bigg. \Bigg),
\end{eqnarray}
and then analogously to \cite{Stojnicnflldp25}'s (138)
\begin{eqnarray}\label{eq:proofxlev2genDanal24}
 \sum_{i_1=1}^{l}  \sum_{i_2=1}^{l} \sum_{j=1}^{m}  \frac{ \beta_{i_1} \sqrt{1-t} }{  \sqrt{\p_1-\p_2}  }  T_{2,1,j}^{\p,d}
 & = & -(1-t)s\beta^2 (\m_1-\m_2) p
 \mE_{G,{\mathcal U}_{r+1}} \langle \|\x^{(i_1)}\|_2\|\x^{(p_1)}\|_2(\y^{(p_2)})^T\y^{(i_2)} \rangle_{\gamma_{21}^{(r)}}.
 \nonumber \\
\end{eqnarray}
Along the same lines, from \cite{Stojnicnflldp25}'s (140) we first have
 \begin{eqnarray}\label{eq:proofxlev2genDanal23bb1}
T_{2,1,j}^{\p,e}
 & = & s\sqrt{1-t}\sqrt{\p_1-\p_2} \m_1 (p-1)  \mE_{G,{\mathcal U}_{r+1}}\Bigg( \Bigg.
    \prod_{v=r}^{2}\mE_{{\mathcal U}_{v}} \frac{\zeta_{v-1}^{\frac{\m_v}{\m_{v-1}}} } { \zeta_v}
\nonumber \\
& &
\times
  \sum_{i_3=1}^{l}
  \frac{
\lp
 \mE_{{\mathcal U}_1}  Z_{i_3}^{\m_1}
\rp^{p}  }  {   \lp
\sum_{i_3=1}^{l}
\lp
 \mE_{{\mathcal U}_1}  Z_{i_3}^{\m_1}
\rp^p
\rp
 }
\mE_{{\mathcal U}_1}\frac{Z_{i_3}^{\m_1}}{\mE_{{\mathcal U}_1} Z_{i_3}^{\m_1}}
 \frac{(C_{i_3}^{(i_1)})^{s}}{Z_{i_3}}  \frac{A_{i_3}^{(i_1,i_2)}}{C_{i_3}^{(i_1)}}\y_j^{(i_2)} \nonumber \\
& & \times
\lp \mE_{{\mathcal U}_1}  \frac{Z_{i_3}^{\m_1}}{\mE_{{\mathcal U}_1} Z_{i_3}^{\m_1}} \sum_{p_1=1}^{l}  \frac{(C_{i_3}^{(p_1)})^s}{Z_{i_3}}\sum_{p_2=1}^{l}
\frac{A_{i_3}^{(p_1,p_2)}}{C_{i_3}^{(p_1)}}\beta_{p_1}\y_j^{(p_2)}  \rp\Bigg. \Bigg),
\end{eqnarray}
and then analogously to  \cite{Stojnicnflldp25}'s (141)
\begin{eqnarray}\label{eq:proofxlev2genDanal24bb2}
 \sum_{i_1=1}^{l}  \sum_{i_2=1}^{l} \sum_{j=1}^{m} \frac{  \beta_{i_1}  \sqrt{1-t} }{  \sqrt{\p_1-\p_2}  }  T_{2,1,j}^{\p,e}
 & = & (1-t) s\beta^2 \m_1 (p-1)
 \mE_{G,{\mathcal U}_{r+1}} \langle \|\x^{(i_1)}\|_2\|\x^{(p_1)}\|_2(\y^{(p_2)})^T\y^{(i_2)} \rangle_{\gamma_{22}^{(r)}}.
\end{eqnarray}
Finally we then find
\begin{eqnarray}\label{eq:proofxlev2genDanal25}
 \sum_{i_1=1}^{l}  \sum_{i_2=1}^{l} \sum_{j=1}^{m}
 \hspace{-.3in}
 & &
  \frac{    \beta_{i_1} \sqrt{1-t}  }{ \sqrt{\p_1-\p_2 } } T_{2,1,j}^{\p}
=
\nonumber \\
 & = & (1-t)\beta^2
  \lp \mE_{G,{\mathcal U}_{r+1}}\langle \|\x^{(i_1)}\|_2^2\|\y^{(i_2)}\|_2^2\rangle_{\gamma_{01}^{(r)}} +   (s-1)\mE_{G,{\mathcal U}_{r+1}}\langle \|\x^{(i_1)}\|_2^2(\y^{(p_2)})^T\y^{(i_2)}\rangle_{\gamma_{02}^{(r)}} \rp\nonumber \\
& & - (1-t)s\beta^2(1-\m_1)\mE_{G,{\mathcal U}_{r+1}}\langle \|\x^{(i_1)}\|_2\|\x^{(p_1)}\|_2(\y^{(p_2)})^T\y^{(i_2)} \rangle_{\gamma_{1}^{(r)}}
\nonumber \\
 &   &
  -(1-t) s\beta^2(\m_1-\m_2) p \mE_{G,{\mathcal U}_{r+1}} \langle \|\x^{(i_1)}\|_2\|\x^{(p_1)}\|_2(\y^{(p_2)})^T\y^{(i_2)} \rangle_{\gamma_{21}^{(r)}}
\nonumber \\
   &   &
+ (1-t) s\beta^2 \m_1 (p-1) \mE_{G,{\mathcal U}_{r+1}} \langle \|\x^{(i_1)}\|_2\|\x^{(p_1)}\|_2(\y^{(p_2)})^T\y^{(i_2)} \rangle_{\gamma_{22}^{(r)}}.
\end{eqnarray}

\vspace{.1in}

\noindent \underline{\textbf{\emph{Determining $T_{1,3}^{\p,\q}$ on the $r$-th level}}}

From (\ref{eq:rthlev2genanal10g}), we recall on
 \begin{eqnarray}\label{eq:proofxxeq1}
 T_{1,3}^{\p,\q} & = &   \mE_{G,{\mathcal U}_{r+1}} \lp
\zeta_r^{-1}\prod_{v=r}^{2}\mE_{{\mathcal U}_{v}} \zeta_{v-1}^{\frac{\m_v}{\m_{v-1}}-1}
\sum_{i_3=1}^{l} \lp \mE_{{\mathcal U}_1} Z_{i_3}^{\m_1}  \rp^{p-1}
  \mE_{{\mathcal U}_1}\frac{(C_{i_3}^{(i_1)})^{s-1} A_{i_3}^{(i_1,i_2)}  u^{(4,1)}}{Z_{i_3}^{1-\m_1}} \rp.
 \end{eqnarray}
Relying again on the strategy of Section \ref{sec:hand1T11}, we  apply Gaussian integration by parts and write analogously ro (\ref{eq:liftgenCanal21}) (and to \cite{Stojnicnflldp25}'s (33))
 \begin{eqnarray}\label{eq:proofxxliftgenCanal21}
T_{1,3}^{(\p,\q)} \hspace{-.1in} & = & \mE_{G,{\mathcal U}_{r+1}} \lp
\zeta_r^{-1}\prod_{v=r}^{2}\mE_{{\mathcal U}_{v}} \zeta_{v-1}^{\frac{\m_v}{\m_{v-1}}-1}
\sum_{i_3=1}^{l} \lp \mE_{{\mathcal U}_1} Z_{i_3}^{\m_1}  \rp^{p-1}
\mE_{{\mathcal U}_1}  \frac{(C_{i_3}^{(i_1)})^{s-1} A_{i_3}^{(i_1,i_2)}u^{(4,1)}}{Z_{i_3}^{1-\m_1}} \rp \nonumber \\
 & = & \sqrt{\p_0\q_0-\p_1\q_1} \mE_{G,{\mathcal U}_2} \Bigg( \Bigg.
\zeta_r^{-1}\prod_{v=r}^{2}\mE_{{\mathcal U}_{v}} \zeta_{v-1}^{\frac{\m_v}{\m_{v-1}}-1}
\sum_{i_3=1}^{l} \lp \mE_{{\mathcal U}_1} Z_{i_3}^{\m_1}  \rp^{p-1}
\nonumber \\
 & & \times \mE_{{\mathcal U}_1} \lp\frac{d}{d \lp \sqrt{\p_0\q_0-\p_1\q_1} u^{(4,1)} \rp} \lp\frac{(C_{i_3}^{(i_1)})^{s-1} A_{i_3}^{(i_1,i_2)}  } {Z_{i_3}^{1-\m_1}}\rp\rp\Bigg.\Bigg) \nonumber \\
 & = & \sqrt{\p_0\q_0-\p_1\q_1} \mE_{G,{\mathcal U}_2} \Bigg( \Bigg.
  \prod_{v=r}^{2}\mE_{{\mathcal U}_{v}} \frac{\zeta_{v-1}^{\frac{\m_v}{\m_{v-1}}} } { \zeta_v}
   \sum_{i_3=1}^{l}
  \frac{
\lp
 \mE_{{\mathcal U}_1}  Z_{i_3}^{\m_1}
\rp^{p-1}  }  {   \lp
\sum_{i_3=1}^{l}
\lp
 \mE_{{\mathcal U}_1}  Z_{i_3}^{\m_1}
\rp^p
\rp
 }
\nonumber \\
 & & \times \mE_{{\mathcal U}_1} \lp\frac{d}{d \lp \sqrt{\p_0\q_0-\p_1\q_1} u^{(4,1)} \rp} \lp\frac{(C_{i_3}^{(i_1)})^{s-1} A_{i_3}^{(i_1,i_2)}  }{Z_{i_3}^{1-\m_1}}\rp\rp\Bigg.\Bigg).
 \end{eqnarray}
Combining (\ref{eq:proofxxliftgenCanal21}) with   \cite{Stojnicnflldp25}'s (34) gives
 \begin{eqnarray}\label{eq:proofxxliftgenCanal21}
T_{1,3}^{(\p,\q)} \hspace{-.14in} & = & \sqrt{\p_0\q_0-\p_1\q_1} \mE_{G,{\mathcal U}_{r+1}} \Bigg( \Bigg.
  \prod_{v=r}^{2}\mE_{{\mathcal U}_{v}} \frac{\zeta_{v-1}^{\frac{\m_v}{\m_{v-1}}} } { \zeta_v}
   \sum_{i_3=1}^{l}
  \frac{
\lp
 \mE_{{\mathcal U}_1}  Z_{i_3}^{\m_1}
\rp^{p-1}  }  {   \lp
\sum_{i_3=1}^{l}
\lp
 \mE_{{\mathcal U}_1}  Z_{i_3}^{\m_1}
\rp^p
\rp
 }
\nonumber \\
 & & \times \mE_{{\mathcal U}_1} \lp\frac{d}{d \lp \sqrt{\p_0\q_0-\p_1\q_1} u^{(4,1)} \rp} \lp\frac{(C_{i_3}^{(i_1)})^{s-1} A_{i_3}^{(i_1,i_2)}   } {Z_{i_3}^{1-\m_1}}\rp\rp\Bigg.\Bigg).
\nonumber \\
 & = & \sqrt{\p_0\q_0-\p_1\q_1} \mE_{G,{\mathcal U}_{r+1}} \Bigg( \Bigg.
  \prod_{v=r}^{2}\mE_{{\mathcal U}_{v}} \frac{\zeta_{v-1}^{\frac{\m_v}{\m_{v-1}}} } { \zeta_v}
   \sum_{i_3=1}^{l}
  \frac{
\lp
 \mE_{{\mathcal U}_1}  Z_{i_3}^{\m_1}
\rp^{p-1}  }  {   \lp
\sum_{i_3=1}^{l}
\lp
 \mE_{{\mathcal U}_1}  Z_{i_3}^{\m_1}
\rp^p
\rp
 }
\nonumber \\
 & & \times
 \Bigg (\Bigg .
\mE_{{\mathcal U}_1} \lp\frac{(C_{i_3}^{(i_1)})^{s-1}\beta_{i_1}A_{i_3}^{(i_1,i_2)}\|\y^{(i_2)}\|_2\sqrt{t}+A_{i_3}^{(i_1,i_2)}(s-1)(C_{i_3}^{(i_1)})^{s-2}\beta_{i_1}\sum_{p_2=1}^{l}A_{i_3}^{(i_1,p_2)}\|\y^{(p_2)}\|_2\sqrt{t}}{Z_{i_3}^{1-\m_1}} \rp
\nonumber\\
& &
- \mE_{{\mathcal U}_1} \lp
\frac{(C_{i_3}^{(i_1)})^{s-1} A_{i_3}^{(i_1,i_2)}}{Z_{i_3}^{2-\m_1}}
s  \sum_{p_1=1}^{l} (C_{i_3}^{(p_1)})^{s-1}\sum_{p_2=1}^{l}\beta_{p_1}A_{i_3}^{(p_1,p_2)}\|\y^{(p_2)}\|_2\sqrt{t}\rp \Bigg. \Bigg)
 \Bigg.\Bigg)
 \nonumber \\
  & = & \sqrt{\p_0\q_0-\p_1\q_1} \mE_{G,{\mathcal U}_{r+1}} \Bigg( \Bigg.
  \prod_{v=r}^{2}\mE_{{\mathcal U}_{v}} \frac{\zeta_{v-1}^{\frac{\m_v}{\m_{v-1}}} } { \zeta_v}
   \sum_{i_3=1}^{l}
  \frac{
\lp
 \mE_{{\mathcal U}_1}  Z_{i_3}^{\m_1}
\rp^{p-1}  }  {   \lp
\sum_{i_3=1}^{l}
\lp
 \mE_{{\mathcal U}_1}  Z_{i_3}^{\m_1}
\rp^p
\rp
 }
\nonumber \\
 & & \times
 \Bigg (\Bigg .
\mE_{{\mathcal U}_1} \lp\frac{(C_{i_3}^{(i_1)})^{s-1}\beta_{i_1}A_{i_3}^{(i_1,i_2)}\|\y^{(i_2)}\|_2\sqrt{t}+A_{i_3}^{(i_1,i_2)}(s-1)(C_{i_3}^{(i_1)})^{s-2}\beta_{i_1}\sum_{p_2=1}^{l}A_{i_3}^{(i_1,p_2)}\|\y^{(p_2)}\|_2\sqrt{t}}{Z_{i_3}^{1-\m_1}} \rp
\nonumber\\
& &
- \mE_{{\mathcal U}_1} \lp
\frac{(C_{i_3}^{(i_1)})^{s-1} A_{i_3}^{(i_1,i_2)}}{Z_{i_3}^{2-\m_1}}
s  \sum_{p_1=1}^{l} (C_{i_3}^{(p_1)})^{s-1}\sum_{p_2=1}^{l}\beta_{p_1}A_{i_3}^{(p_1,p_2)}\|\y^{(p_2)}\|_2\sqrt{t}\rp \Bigg. \Bigg)
 \Bigg.\Bigg).
 \end{eqnarray}
One then has the following analogue to (\ref{eq:liftgenCanal21b})
\begin{eqnarray}\label{eq:proofxxliftgenCanal21b}
\sum_{i_1=1}^{l}\sum_{i_2=1}^{l}
\hspace{-.3in} & &
 \frac{  \beta_{i_1}\|\y^{(i_2)}\|_2 \sqrt{t}}{\sqrt{\p_0\q_0-\p_1\q_1}}T_{1,3}^{(\p,\q)}
=
\nonumber \\
& = & t\beta^2 \Bigg( \Bigg. \mE_{G,{\mathcal U}_{r+1}}\langle \|\x^{(i_1)}\|_2^2\|\y^{(i_2)}\|_2^2\rangle_{\gamma_{01}^{(r)}}
 +   (s-1)\mE_{G,{\mathcal U}_{r+1}}\langle \|\x^{(i_1)}\|_2^2 \|\y^{(i_2)}\|_2\|\y^{(p_2)}\|_2\rangle_{\gamma_{02}^{(r)}}\Bigg.\Bigg) \nonumber \\
& & - ts\beta^2(1-\m_1)\mE_{G,{\mathcal U}_{r+1}}\langle \|\x^{(i_1)}\|_2\|\x^{(p_`)}\|_2\|\y^{(i_2)}\|_2\|\y^{(p_2)}\|_2 \rangle_{\gamma_{1}^{(r)}}.
\end{eqnarray}

\vspace{.1in}

\noindent \underline{\textbf{\emph{Determining $T_{2,3}^{\p,\q}$ on the $r$-th level}}}

From (\ref{eq:rthlev2genanal10g}), we recall on
 \begin{eqnarray}\label{eq:proofxxxeq1}
 T_{2,3}^{\p,\q} & = &   \mE_{G,{\mathcal U}_{r+1}} \lp
\zeta_r^{-1}\prod_{v=r}^{2}\mE_{{\mathcal U}_{v}} \zeta_{v-1}^{\frac{\m_v}{\m_{v-1}}-1}
\sum_{i_3=1}^{l} \lp \mE_{{\mathcal U}_1} Z_{i_3}^{\m_1}  \rp^{p-1}
  \mE_{{\mathcal U}_1}\frac{(C_{i_3}^{(i_1)})^{s-1} A_{i_3}^{(i_1,i_2)}  u^{(4,2)}}{Z_{i_3}^{1-\m_1}} \rp.
 \end{eqnarray}
We proceed by following the strategy of \cite{Stojnicnflldp25}'s Section 3.1.2 (part 3). Gaussian integration by parts gives
\begin{eqnarray}\label{eq:proofxxxlev2genFanal21}
T_{2,3}^{\p,\q} & = &   \mE_{G,{\mathcal U}_{r+1}} \lp
\zeta_r^{-1}\prod_{v=r}^{2}\mE_{{\mathcal U}_{v}} \zeta_{v-1}^{\frac{\m_v}{\m_{v-1}}-1}
\sum_{i_3=1}^{l} \lp \mE_{{\mathcal U}_1} Z_{i_3}^{\m_1}  \rp^{p-1}
\mE_{{\mathcal U}_1}\frac{(C_{i_3}^{(i_1)})^{s-1} A_{i_3}^{(i_1,i_2)} u^{(4,2)}}{Z_{i_3}^{1-\m_1}} \rp
\nonumber  \\
& = & \mE_{G,{\mathcal U}_{r+1}} \Bigg( \Bigg.
\zeta_r^{-1}\prod_{v=r}^{3}\mE_{{\mathcal U}_{v}} \zeta_{v-1}^{\frac{\m_v}{\m_{v-1}}-1}
\sum_{i_3=1}^{l} \lp \mE_{{\mathcal U}_1} Z_{i_3}^{\m_1}  \rp^{p-1}
\nonumber \\
& & \times \mE_{{\mathcal U}_2} \lp\mE_{{\mathcal U}_2} (u^{(4,2)}u^{(4,2)})\lp\frac{d}{du^{(4,2)}} \lp\frac{(C_{i_3}^{(i_1)})^{s-1} A_{i_3}^{(i_1,i_2)}}{Z_{i_3}^{1-\m_1}
\zeta_1^{1-\frac{\m_2}{\m_1}} \lp
 \mE_{{\mathcal U}_1}  Z_{i_3}^{\m_1}
\rp^{1-p}
 }
\rp \rp\rp\Bigg. \Bigg) \nonumber \\
& = & \sqrt{\p_1\q_1-\p_2\q_2}  \mE_{G,{\mathcal U}_{r+1}}  \Bigg ( \Bigg.
\zeta_r^{-1}\prod_{v=r}^{2}\mE_{{\mathcal U}_{v}} \zeta_{v-1}^{\frac{\m_v}{\m_{v-1}}-1}
\sum_{i_3=1}^{l} \lp \mE_{{\mathcal U}_1} Z_{i_3}^{\m_1}  \rp^{p-1}
\nonumber \\
& &
\times
\lp
\frac{d}{d \lp \sqrt{\p_1\q_1-\p_2\q_2}   u^{(4,2)} \rp  } \lp\frac{(C_{i_3}^{(i_1)})^{s-1} A_{i_3}^{(i_1,i_2)}}{Z_{i_3}^{1-\m_1}}
\rp
\rp
 \Bigg ) \Bigg.
\nonumber \\
& & + \sqrt{\p_1\q_1-\p_2\q_2}   \mE_{G,{\mathcal U}_{r+1}}  \Bigg ( \Bigg.
\zeta_r^{-1}\prod_{v=r}^{3}\mE_{{\mathcal U}_{v}} \zeta_{v-1}^{\frac{\m_v}{\m_{v-1}}-1}
\sum_{i_3=1}^{l} \lp \mE_{{\mathcal U}_1} Z_{i_3}^{\m_1}  \rp^{p-1}
\lp   \frac{(C_{i_3}^{(i_1)})^{s-1} A_{i_3}^{(i_1,i_2)}}{Z_{i_3}^{1-\m_1}}    \rp
\nonumber \\
& &
\times
 \Bigg ( \Bigg.
 \frac{d}{d \lp \sqrt{\p_1\q_1-\p_2\q_2}  u^{(4,2)} \rp  }
  \Bigg ( \Bigg.
\zeta_1^{\frac{\m_2}{\m_1} -1 }
 \Bigg ) \Bigg.
 \Bigg ) \Bigg.
 \Bigg ) \Bigg.
\nonumber \\
& & +  \sqrt{\p_1\q_1-\p_2\q_2}  \mE_{G,{\mathcal U}_{r+1}} \Bigg ( \Bigg.
\zeta_r^{-1}\prod_{v=r}^{3}\mE_{{\mathcal U}_{v}} \zeta_{v-1}^{\frac{\m_v}{\m_{v-1}}-1}
\sum_{i_3=1}^{l}
\lp   \frac{(C_{i_3}^{(i_1)})^{s-1} A_{i_3}^{(i_1,i_2)}}{Z_{i_3}^{1-\m_1}}    \rp
\nonumber \\
& &
\times
\lp\frac{d}{d \lp  \sqrt{\p_1\q_1-\p_2\q_2}   u^{(4,2)} \rp } \lp
 \lp
 \mE_{{\mathcal U}_1}  Z_{i_3}^{\m_1}
\rp^{p-1}
\rp
\rp
\Bigg .\Bigg ).\nonumber \\
\end{eqnarray}
Writing it in a more convenient way we also have
\begin{eqnarray}\label{eq:proofxxxlev2genFanal22}
T_{2,3}^{\p.\q} & = & T_{2,3}^{\p.\q,c} +T_{2,3}^{\p.\q,d} +T_{2,3}^{\p.\q,e},
\end{eqnarray}
where
\begin{eqnarray}\label{eq:lev2genFanal23}
T_{2,3}^{\p.\q,c} & = & \sqrt{\p_1\q_1-\p_2\q_2}  \mE_{G,{\mathcal U}_{r+1}}  \Bigg ( \Bigg.
\zeta_r^{-1}\prod_{v=r}^{2}\mE_{{\mathcal U}_{v}} \zeta_{v-1}^{\frac{\m_v}{\m_{v-1}}-1}
\sum_{i_3=1}^{l} \lp \mE_{{\mathcal U}_1} Z_{i_3}^{\m_1}  \rp^{p-1}
\nonumber \\
& &
\times
\lp
\frac{d}{d \lp \sqrt{\p_1\q_1-\p_2\q_2}   u^{(4,2)} \rp  } \lp\frac{(C_{i_3}^{(i_1)})^{s-1} A_{i_3}^{(i_1,i_2)}}{Z_{i_3}^{1-\m_1}}
\rp
\rp
 \Bigg ) \Bigg.
\nonumber \\
 T_{2,3}^{\p.\q,d} & = & \sqrt{\p_1\q_1-\p_2\q_2}   \mE_{G,{\mathcal U}_{r+1}}  \Bigg ( \Bigg.
\zeta_r^{-1}\prod_{v=r}^{3}\mE_{{\mathcal U}_{v}} \zeta_{v-1}^{\frac{\m_v}{\m_{v-1}}-1}
\sum_{i_3=1}^{l} \lp \mE_{{\mathcal U}_1} Z_{i_3}^{\m_1}  \rp^{p-1}
\lp   \frac{(C_{i_3}^{(i_1)})^{s-1} A_{i_3}^{(i_1,i_2)}}{Z_{i_3}^{1-\m_1}}    \rp
\nonumber \\
& &
\times
 \Bigg ( \Bigg.
 \frac{d}{d \lp \sqrt{\p_1\q_1-\p_2\q_2}  u^{(4,2)} \rp  }
  \Bigg ( \Bigg.
\zeta_1^{\frac{\m_2}{\m_1} -1 }
 \Bigg ) \Bigg.
 \Bigg ) \Bigg.
 \Bigg ) \Bigg.
\nonumber \\
 T_{2,3}^{\p.\q,e} & = &
\sqrt{\p_1\q_1-\p_2\q_2}  \mE_{G,{\mathcal U}_{r+1}} \Bigg ( \Bigg.
\zeta_r^{-1}\prod_{v=r}^{3}\mE_{{\mathcal U}_{v}} \zeta_{v-1}^{\frac{\m_v}{\m_{v-1}}-1}
\sum_{i_3=1}^{l}
\lp   \frac{(C_{i_3}^{(i_1)})^{s-1} A_{i_3}^{(i_1,i_2)}}{Z_{i_3}^{1-\m_1}}    \rp
\nonumber \\
& &
\times
\lp\frac{d}{d \lp  \sqrt{\p_1\q_1-\p_2\q_2}   u^{(4,2)} \rp } \lp
 \lp
 \mE_{{\mathcal U}_1}  Z_{i_3}^{\m_1}
\rp^{p-1}
\rp
\rp
\Bigg .\Bigg ). \nonumber \\
 \end{eqnarray}
We then have analogously to \cite{Stojnicnflldp25}'s (160) (see also the third part of \cite{Stojnicnflldp25}'s Section 2.1.2)
\begin{eqnarray}\label{eq:proofxxxlev2genFanal23b}
\sum_{i_1=1}^{l}\sum_{i_2=1}^{l}  \frac{ \beta_{i_1}\|\y^{(i_2)}\|_2 \sqrt{t} }{\sqrt{\p_1\q_1-\p_2\q_2}   }     T_{2,3}^{\p,\q,c}   & = &
 \beta^2 \Bigg(\Bigg. \mE_{G,{\mathcal U}_{r+1}}\langle \|\x^{(i_1)}\|_2^2\|\y^{(i_2)}\|_2^2\rangle_{\gamma_{01}^{(r)}} \nonumber \\
& & +   (s-1)\mE_{G,{\mathcal U}_{r+1}}\langle \|\x^{(i_1)}\|_2^2 \|\y^{(i_2)}\|_2\|\y^{(p_2)}\|_2\rangle_{\gamma_{02}^{(r)}} \Bigg.\Bigg) \nonumber \\
& & -   s\beta^2(1-\m_1)\mE_{G,{\mathcal U}_{r+1}}\langle \|\x^{(i_1)}\|_2\|\x^{(p_`)}\|_2\|\y^{(i_2)}\|_2\|\y^{(p_2)}\|_2 \rangle_{\gamma_{1}^{(r)}}.
\end{eqnarray}
Utilizing \cite{Stojnicnflldp25}'s (163) we further  have
  \begin{eqnarray}\label{eq:proofxxxlev2genFanal27}
T_{2,3}^{\p,\q,d}
& = & -s\sqrt{t}\sqrt{\p_1\q_1-\p_2\q_2}(\m_1-\m_2) p \mE_{G,{\mathcal U}_3} \Bigg( \Bigg.
  \prod_{v=r}^{2}\mE_{{\mathcal U}_{v}} \frac{\zeta_{v-1}^{\frac{\m_v}{\m_{v-1}}} } { \zeta_v}
\nonumber \\
& &
\times
\sum_{i_3=1}^{l}
  \frac{
\lp
 \mE_{{\mathcal U}_1}  Z_{i_3}^{\m_1}
\rp^{p}  }{   \lp
\sum_{i_3=1}^{l}
\lp
 \mE_{{\mathcal U}_1}  Z_{i_3}^{\m_1}
\rp^p
\rp
 }
\mE_{{\mathcal U}_1} \frac{Z_{i_3}^{\m_1}}{\mE_{{\mathcal U}_1} Z_{i_3}^{\m_1}} \frac{(C_{i_3}^{(i_1)})^s}{Z_{i_3}}\frac{A_{i_3}^{(i_1,i_2)}}{(C_{i_3}^{(i_1)})}\nonumber \\
& & \times
\sum_{p_3=1}^{l}
  \frac{
\lp
 \mE_{{\mathcal U}_1}  Z_{p_3}^{\m_1}
\rp^{p}  }{   \lp
\sum_{p_3=1}^{l}
\lp
 \mE_{{\mathcal U}_1}  Z_{p_3}^{\m_1}
\rp^p
\rp
 }
\mE_{{\mathcal U}_1} \frac{Z_{p_3}^{\m_1}}{\mE_{{\mathcal U}_1} Z_{p_3}^{\m_1}} \sum_{p_1=1}^{l} \frac{(C_{p_3}^{(p_1)})^{s}}{Z_{p_3}}\sum_{p_2=1}^{l}
\frac{A_{p_3}^{(p_1,p_2)}}{(C^{(p_1)})}\beta_{p_1}\|\y^{(p_2)}\|_2\Bigg. \Bigg),
\end{eqnarray}
and consequently
\begin{eqnarray}\label{eq:proofxxxlev2genFanal28}
\sum_{i_1=1}^{l}\sum_{i_2=1}^{l}\frac{  \beta_{i_1}\|\y^{(i_2)}\|_2 \sqrt{t}  }{\sqrt{\p_1\q_2-\p_2\q_2}  }   T_{2,3}^{\p,\q,d}
 & = & -ts\beta^2 (\m_1-\m_2) p \mE_{G,{\mathcal U}_{r+1}} \langle\|\x^{(i_2)}\|_2\|\x^{(p_2)}\|_2\|\y^{(i_2)}\|_2\|\y^{(p_2)}\rangle_{\gamma_{21}^{(r)}}.\nonumber \\
\end{eqnarray}
Utilizing \cite{Stojnicnflldp25}'s (166) we first find
\begin{eqnarray}\label{eq:proofxxxlev2genFanal27bb2}
T_{2,3}^{\p,\q,e}
& = & s\sqrt{t}\sqrt{ \p_1\q_1-\p_2\q_2 } \m_1 (p-1) \mE_{G,{\mathcal U}_{r+1}} \Bigg( \Bigg.
  \prod_{v=r}^{2}\mE_{{\mathcal U}_{v}} \frac{\zeta_{v-1}^{\frac{\m_v}{\m_{v-1}}} } { \zeta_v}
\nonumber \\
& &
\times
\sum_{i_3=1}^{l}
  \frac{
\lp
 \mE_{{\mathcal U}_1}  Z_{i_3}^{\m_1}
\rp^{p}  }{   \lp
\sum_{i_3=1}^{l}
\lp
 \mE_{{\mathcal U}_1}  Z_{i_3}^{\m_1}
\rp^p
\rp
 }
\mE_{{\mathcal U}_1} \frac{Z_{i_3}^{\m_1}}{\mE_{{\mathcal U}_1} Z_{i_3}^{\m_1}} \frac{(C_{i_3}^{(i_1)})^s}{Z_{i_3}}\frac{A_{i_3}^{(i_1,i_2)}}{(C_{i_3}^{(i_1)})}\nonumber \\
& & \times
\mE_{{\mathcal U}_1} \frac{Z_{i_3}^{\m_1}}{\mE_{{\mathcal U}_1} Z_{i_3}^{\m_1}} \sum_{p_1=1}^{l} \frac{(C_{i_3}^{(p_1)})^{s}}{Z_{i_3}}\sum_{p_2=1}^{l}
\frac{A_{i_3}^{(p_1,p_2)}}{(C_{i_3}^{(p_1)})}\beta_{p_1}\|\y^{(p_2)}\|_2\Bigg. \Bigg),
\end{eqnarray}
and then
\begin{eqnarray}\label{eq:proofxxxlev2genFanal28bb3}
\sum_{i_1=1}^{l}\sum_{i_2=1}^{l} \frac{ \beta_{i_1}\|\y^{(i_2)}\|_2 \sqrt{t} }{ \sqrt{\p_1\q_1-\p_2\q_2}  }    T_{2,3}^{\p,\q,e}
 & = & ts\beta^2 \m_1 (p-1) \mE_{G,{\mathcal U}_{r+1}} \langle\|\x^{(i_2)}\|_2\|\x^{(p_2)}\|_2\|\y^{(i_2)}\|_2\|\y^{(p_2)}\rangle_{\gamma_{22}^{(r)}}.
\end{eqnarray}
Combining  (\ref{eq:proofxxxlev2genFanal22}), (\ref{eq:proofxxxlev2genFanal23b}), (\ref{eq:proofxxxlev2genFanal28}), and (\ref{eq:proofxxxlev2genFanal28bb3}) we then obtain the following analogue to (\ref{eq:genFanal29})
 \begin{eqnarray}\label{eq:proofxxxlev2genFanal29}
\sum_{i_1=1}^{l}\sum_{i_2=1}^{l}  \frac{  \beta_{i_1}\|\y^{(i_2)}\|_2 \sqrt{t} } { \sqrt{\p_1\q_1-\p_2\q_2}  }    T_{2,3}^{\p,\q}
 & = &
 t\beta^2 \Bigg( \Bigg. \mE_{G,{\mathcal U}_{r+1}}\langle \|\x^{(i_1)}\|_2^2\|\y^{(i_2)}\|_2^2\rangle_{\gamma_{01}^{(r)}} \nonumber \\
& & +  (s-1)\mE_{G,{\mathcal U}_{r+1}}\langle \|\x^{(i_1)}\|_2^2 \|\y^{(i_2)}\|_2\|\y^{(p_2)}\|_2\rangle_{\gamma_{02}^{(r)}}\Bigg.\Bigg) \nonumber \\
& & - t s\beta^2(1-\m_1)\mE_{G,{\mathcal U}_{r+1}}\langle \|\x^{(i_1)}\|_2\|\x^{(p_`)}\|_2\|\y^{(i_2)}\|_2\|\y^{(p_2)}\|_2 \rangle_{\gamma_{1}^{(r)}} \nonumber \\
&  & -ts\beta^2 (\m_1-\m_2) p \mE_{G,{\mathcal U}_{r+1}} \langle\|\x^{(i_2)}\|_2\|\x^{(p_2)}\|_2\|\y^{(i_2)}\|_2\|\y^{(p_2)} \|_2 \rangle_{\gamma_{21}^{(r)}}\nonumber \\
&  & +ts\beta^2  \m_1 (p-1) \mE_{G,{\mathcal U}_{r+1}} \langle\|\x^{(i_2)}\|_2\|\x^{(p_2)}\|_2\|\y^{(i_2)}\|_2\|\y^{(p_2)} \|_2  \rangle_{\gamma_{22}^{(r)}}.
\end{eqnarray}
Finally, combining (\ref{eq:rthlev2genanal10e})-(\ref{eq:rthlev2genanal10f}) (for $k_1=1$) with  (\ref{eq:proofliftgenAanal19i}), (\ref{eq:proofxlev2genDanal25}), (\ref{eq:proofxxliftgenCanal21b}), and (\ref{eq:proofxxxlev2genFanal29}) we obtain  that $\frac{d\psi(t)}{d\p_1}$ is indeed as stated in (\ref{eq:thm3eq41aa0}) in Theorem \ref{thm:thm3}.

As mentioned earlier, handling $T_{1,2}^{\q}$ and $T_{2,2}^{\q}$ is done in a completely analogous manner and we skip redoing the whole procedure  once again. Instead we note that the final results are
\begin{eqnarray}\label{eq:proofxxxxliftgenBanal20b}
\sum_{i_1=1}^{l}\sum_{i_2=1}^{l}  \frac{\beta_{i_1}\|\y^{(i_2)}\|_2\sqrt{1-t}}{\sqrt{\q_0-\q_1}}T_{1,2}^{\q} & = & (1-t)\beta^2
\Bigg( \Bigg.\mE_{G,{\mathcal U}_{r+1}}\langle \|\x^{(i_1)}\|_2^2\|\y^{(i_2)}\|_2^2\rangle_{\gamma_{01}^{(r)}} \nonumber \\
& & +   (s-1)\mE_{G,{\mathcal U}_{r+1}}\langle \|\x^{(i_1)}\|_2^2 \|\y^{(i_2)}\|_2\|\y^{(p_2)}\|_2\rangle_{\gamma_{02}^{(r)}}\Bigg.\Bigg)  \nonumber \\
& & - (1-t)s\beta^2(1-\m_1)\mE_{G,{\mathcal U}_{r+1}}\langle (\x^{(p_1)})^T\x^{(i_1)}\|\y^{(i_2)}\|_2\|\y^{(p_2)}\|_2 \rangle_{\gamma_{1}^{(r)}},
\nonumber \\
\end{eqnarray}
and
 \begin{eqnarray}\label{eq:proofxxxxlev2genEanal25}
\sum_{i_1=1}^{l}\sum_{i_2=1}^{l}   \frac{  \beta_{i_1}\|\y^{(i_2)}\|_2 \sqrt{1-t} }  { \sqrt{\q_1 - \q_2} }   T_{2,2}^{\q}
& = &  (1-t) \beta^2 \Bigg( \Bigg. \mE_{G,{\mathcal U}_{r+1}}\langle \|\x^{(i_1)}\|_2^2\|\y^{(i_2)}\|_2^2\rangle_{\gamma_{01}^{(r)}} \nonumber \\
& & +   (s-1)\mE_{G,{\mathcal U}_{r+1}}\langle \|\x^{(i_1)}\|_2^2 \|\y^{(i_2)}\|_2\|\y^{(p_2)}\|_2\rangle_{\gamma_{02}^{(r)}}\Bigg.\Bigg) \nonumber \\
& & - (1-t)s\beta^2(1-\m_1)\mE_{G,{\mathcal U}_{r+1}}\langle (\x^{(p_1)})^T\x^{(i_1)}\|\y^{(i_2)}\|_2\|\y^{(p_2)}\|_2 \rangle_{\gamma_{1}^{(r)}} \nonumber \\
&  & -(1-t)s\beta^2 (\m_1-\m_2) p  \mE_{G,{\mathcal U}_{r+1}} \langle \|\y^{(i_2)}\|_2\|\y^{(p_2)}\|_2(\x^{(i_1)})^T\x^{(p_1)}\rangle_{\gamma_{21}^{(r)}}
\nonumber \\
&  &  + (1-t)s\beta^2  \m_1 (p-1) \mE_{G,{\mathcal U}_{r+1}} \langle \|\y^{(i_2)}\|_2\|\y^{(p_2)}\|_2(\x^{(i_1)})^T\x^{(p_1)}  \rangle_{\gamma_{22}^{(r)}}. \nonumber \\
\end{eqnarray}
Combining (\ref{eq:rthlev2genanal10e})-(\ref{eq:rthlev2genanal10f}) (for $k_1=1$) with  (\ref{eq:proofxxxxliftgenBanal20b}), (\ref{eq:proofxxxxlev2genEanal25}), (\ref{eq:proofxxliftgenCanal21b}), and (\ref{eq:proofxxxlev2genFanal29}) one then finds that $\frac{d\psi(t)}{d\q_1}$ is also as stated in (\ref{eq:thm3eq41aa0}) in Theorem \ref{thm:thm3}.
 \end{proof}

%%%%%%%%%%%%%%%%%%%%%%%%%%%%%%%%%%%%%%%%%%%%%%%%%%%%%%%%%%%%%%%%%%%%%%%%%%%%%%%%
\section{Interpolating path stationarization}
\label{sec:statalongpath}
%%%%%%%%%%%%%%%%%%%%%%%%%%%%%%%%%%%%%%%%%%%%%%%%%%%%%%%%%%%%%%%%%%%%%%%%%%%%%%%%

We first recall on a fundamental result from \cite{Stojnicnflldp25}.
\begin{theorem}
\label{thm:thm4}
Assume the setup of Theorem \ref{thm:thm3}. Let also
\begin{eqnarray}\label{eq:thm3eq5}
 \phi_{k_1}^{(r)} & = &
-s(\m_{k_1-1}-\m_{k_1}) \omega(k_1;p) \nonumber \\
&  & \times
\mE_{G,{\mathcal U}_{r+1}} \langle (\p_{k_1-1}\|\x^{(i_1)}\|_2\|\x^{(p_1)}\|_2 -(\x^{(p_1)})^T\x^{(i_1)})(\q_{k_1-1}\|\y^{(i_2)}\|_2\|\y^{(p_2)}\|_2 -(\y^{(p_2)})^T\y^{(i_2)})\rangle_{\gamma_{k_1}^{(r)}} \nonumber \\
 \phi_{22}^{(r)} & = &
-s\m_1 (1-p) \nonumber \\
&  & \times
\mE_{G,{\mathcal U}_{r+1}} \langle (\p_{1}\|\x^{(i_1)}\|_2\|\x^{(p_1)}\|_2 -(\x^{(p_1)})^T\x^{(i_1)})(\q_{1}\|\y^{(i_2)}\|_2\|\y^{(p_2)}\|_2 -(\y^{(p_2)})^T\y^{(i_2)})\rangle_{\gamma_{22}^{(r)}} \nonumber \\
 \phi_{01}^{(r)} & = & (1-\p_0)(1-\q_0)\mE_{G,{\mathcal U}_{r+1}}\langle \|\x^{(i_1)}\|_2^2\|\y^{(i_2)}\|_2^2\rangle_{\gamma_{01}^{(r)}} \nonumber\\
\phi_{02}^{(r)} & = & (s-1)(1-\p_0)\mE_{G,{\mathcal U}_{r+1}}\left\langle \|\x^{(i_1)}\|_2^2 \lp\q_0\|\y^{(i_2)}\|_2\|\y^{(p_2)}\|_2-(\y^{(p_2)})^T\y^{(i_2)}\rp\right\rangle_{\gamma_{02}^{(r)}}, \end{eqnarray}
where
\begin{eqnarray}\label{eq:thm3rthlev2genanal38aa0}
\omega(x;p) \triangleq \begin{cases}
                         1, & \mbox{if } x=1 \\
                         p, & \mbox{otherwise}.
                       \end{cases},
\end{eqnarray}
Then
\begin{eqnarray}\label{eq:thm3eq6}
\frac{d\psi(t)}{dt}  & = &       \frac{\mbox{sign}(s)\beta^2}{2\sqrt{n}} \lp  \lp\sum_{k_1=1}^{r+1} \phi_{k_1}^{(r)}\rp +\phi_{22}^{(r)}+\phi_{01}^{(r)}+\phi_{02}^{(r)}\rp.
 \end{eqnarray}
It particular, choosing $\p_0=\q_0=1$, one also has
\begin{eqnarray}\label{eq:rthlev2genanal43}
\frac{d\psi(t)}{dt}  & = &       \frac{\mbox{sign}(s)\beta^2}{2\sqrt{n}}  \lp \sum_{k_1=1}^{r+1} \phi_{k_1}^{(r)}  +\phi_{22}^{(r)}   \rp.
 \end{eqnarray}
 \end{theorem}
\begin{proof}
Presented in \cite{Stojnicnflldp25}.
\end{proof}

We now consider a slightly different function $\psi_1$
\begin{eqnarray}\label{eq:saip1}
\psi_1(t) & = & -\frac{\mbox{sign}(s) s \beta^2}{2\sqrt{n}} \nonumber \\
& & \times \sum_{k=1}^{r+1}\Bigg(\Bigg. \p_{k-1}\q_{k-1}\mE_{G,{\mathcal U}_{r+1}} \langle\|\x^{(i_1)}\|_2\|\x^{(p_1)}\|_2\|\y^{(i_2)}\|_2\|\y^{(p_2)}\|_2\rangle_{\gamma_{k}^{(r)}}\nonumber \\
& & -\p_{k}\q_{k}\mE_{G,{\mathcal U}_{r+1}} \langle\|\x^{(i_1)}\|_2\|\x^{(p_1)}\|_2\|\y^{(i_2)}\|_2\|\y^{(p_2)}\|_2\rangle_{\gamma_{k+1}^{(r)}}\Bigg.\Bigg)
\m_{k} \omega(k_1;p)  \nonumber \\
& & +\psi(t).
 \end{eqnarray}
 For the simplicity of writing we assume constant magnitudes setup, i.e., we assume that within each of the sets ${\mathcal X}$, $\bar{{\mathcal X}}$, and ${\mathcal Y}$ the elements have equal magnitudes.

For $k_1>1$ we first find
\begin{eqnarray}\label{eq:saip2}
\frac{d\psi_1( t)}{d\p_{k_1}}
& = &
\frac{\mbox{sign}(s)s\beta^2}{2\sqrt{n}}\q_{k_1}\lp \m_{k_1}\omega(k_1;p) -\m_{k_1+1}\omega(k_1+1;p)   \rp
\nonumber \\
& & \times
\mE_{G,{\mathcal U}_{r+1}} \langle\|\x^{(i_1)}\|_2\|\x^{(p_1)}\|_2\|\y^{(i_2)}\|_2\|\y^{(p_2)}\|_2\rangle_{\gamma_{k_1+1}^{(r)}}
 +\frac{\mbox{sign}(s)s\beta}{2\sqrt{n}} \phi^{(k_1,\p)} \nonumber \\
& = &
\frac{\mbox{sign}(s)s\beta^2}{2\sqrt{n}} \Bigg(\Bigg. \q_{k_1}\lp \m_{k_1} -\m_{k_1+1}\rp p \mE_{G,{\mathcal U}_{r+1}} \langle\|\x^{(i_1)}\|_2\|\x^{(p_1)}\|_2\|\y^{(i_2)}\|_2\|\y^{(p_2)}\|_2\rangle_{\gamma_{k_1+1}^{(r)}}\nonumber \\
& &    -(1-t)\lp \m_{k_1}-\m_{k_1+1}\rp p \mE_{G,{\mathcal U}_{r+1}} \langle \|\x^{(i_1)}\|_2\|\x^{(p_1)}\|_2(\y^{(p_2)})^T\y^{(i_2)} \rangle_{\gamma_{k_1+1}^{(r)}} \nonumber \\
& &   - t\q_{k_1}
\lp \m_{k_1}-\m_{k_1+1}\rp  p \mE_{G,{\mathcal U}_{r+1}} \langle\|\x^{(i_1)}\|_2\|\x^{(p_1)}\|_2\|\y^{(i_2)}\|_2\|\y^{(p_2)}\|_2\rangle_{\gamma_{k_1+1}^{(r)}} \Bigg.\Bigg)\nonumber \\
& = &
(1-t)\lp \m_{k_1}-\m_{k_1+1}\rp p \frac{\mbox{sign}(s)s\beta^2}{2\sqrt{n}} \nonumber \\
& & \times \Bigg(\Bigg.   \mE_{G,{\mathcal U}_{r+1}} \langle\|\x^{(i_1)}\|_2\|\x^{(p_1)}\|_2 \lp \q_{k_1} \|\y^{(i_2)}\|_2\|\y^{(p_2)}\|_2 -(\y^{(p_2)})^T\y^{(i_2)}\rp\rangle_{\gamma_{k_1+1}^{(r)}} \Bigg.\Bigg),
 \end{eqnarray}
and then
\begin{eqnarray}\label{eq:saip3}
\frac{d\psi_1(t)}{d\q_{k_1}}
 & = &
\frac{\mbox{sign}(s)s\beta^2}{2\sqrt{n}}\p_{k_1}\lp \m_{k_1}\omega(k_1;p) -\m_{k_1+1}\omega(k_1+1;p)
\rp
\nonumber \\
& & \times
\mE_{G,{\mathcal U}_{r+1}} \langle\|\x^{(i_1)}\|_2\|\x^{(p_1)}\|_2\|\y^{(i_2)}\|_2\|\y^{(p_2)}\|_2\rangle_{\gamma_{k_1+1}^{(r)}}  +\frac{\mbox{sign}(s)s\beta^2}{2\sqrt{n}} \phi^{(k_1,\q)} \nonumber \\
& = &
\frac{\mbox{sign}(s)s\beta^2}{2\sqrt{n}} \Bigg(\Bigg. \p_{k_1}\lp \m_{k_1} -\m_{k_1+1}\rp p \mE_{G,{\mathcal U}_{r+1}} \langle\|\x^{(i_1)}\|_2\|\x^{(p_1)}\|_2\|\y^{(i_2)}\|_2\|\y^{(p_2)}\|_2\rangle_{\gamma_{k_1+1}^{(r)}}\nonumber \\
& &    -(1-t)\lp \m_{k_1}-\m_{k_1+1}\rp p \mE_{G,{\mathcal U}_{r+1}} \langle \|\y^{(i_2)}\|_2\|\y^{(p_2)}\|_2(\x^{(p_1)})^T\x^{(i_1)} \rangle_{\gamma_{k_1+1}^{(r)}} \nonumber \\
& &   - t\p_{k_1}
\lp \m_{k_1}-\m_{k_1+1}\rp  p \mE_{G,{\mathcal U}_{r+1}} \langle\|\x^{(i_1)}\|_2\|\x^{(p_1)}\|_2\|\y^{(i_2)}\|_2\|\y^{(p_2)}\|_2\rangle_{\gamma_{k_1+1}^{(r)}} \Bigg.\Bigg)\nonumber \\
& = &
(1-t)\lp \m_{k_1}-\m_{k_1+1}\rp p \frac{\mbox{sign}(s)s\beta^2}{2\sqrt{n}} \nonumber \\
& & \times \Bigg(\Bigg.   \mE_{G,{\mathcal U}_{r+1}} \langle\|\y^{(i_2)}\|_2\|\y^{(p_2)}\|_2
\lp \p_{k_1} \|\x^{(i_1)}\|_2\|\x^{(p_1)}\|_2 -(\x^{(p_1)})^T\x^{(i_1)}\rp\rangle_{\gamma_{k_1+1}^{(r)}} \Bigg.\Bigg).
 \end{eqnarray}
For $k_1=1$ we further have
\begin{eqnarray}\label{eq:saip2aa0}
\frac{d\psi_1( t)}{d\p_{1}}
& = &
\frac{\mbox{sign}(s)s\beta^2}{2\sqrt{n}}\q_{1}\lp \m_{1}\omega(1;p) -\m_{2}\omega(2;p)   \rp
\mE_{G,{\mathcal U}_{r+1}} \langle\|\x^{(i_1)}\|_2\|\x^{(p_1)}\|_2\|\y^{(i_2)}\|_2\|\y^{(p_2)}\|_2\rangle_{\gamma_{2}^{(r)}}
\nonumber \\
& &
 +\frac{\mbox{sign}(s)s\beta^2}{2\sqrt{n}} \phi^{(1,\p)} \nonumber \\
& = &
\frac{\mbox{sign}(s)s\beta^2}{2\sqrt{n}} \Bigg(\Bigg. \q_{1}\lp \m_{1} -\m_{2} p\rp  \mE_{G,{\mathcal U}_{r+1}} \langle\|\x^{(i_1)}\|_2\|\x^{(p_1)}\|_2\|\y^{(i_2)}\|_2\|\y^{(p_2)}\|_2\rangle_{\gamma_{2}^{(r)}}\nonumber \\
& &    -(1-t)\lp \m_{1}-\m_{2}\rp p \mE_{G,{\mathcal U}_{r+1}} \langle \|\x^{(i_1)}\|_2\|\x^{(p_1)}\|_2(\y^{(p_2)})^T\y^{(i_2)} \rangle_{\gamma_{21}^{(r)}} \nonumber \\
& &   - t\q_{1}
\lp \m_{1}-\m_{2}\rp  p \mE_{G,{\mathcal U}_{r+1}} \langle\|\x^{(i_1)}\|_2\|\x^{(p_1)}\|_2\|\y^{(i_2)}\|_2\|\y^{(p_2)}\|_2\rangle_{\gamma_{21}^{(r)}}
 \nonumber \\
& &    +(1-t) \m_{1} (p-1) \mE_{G,{\mathcal U}_{r+1}} \langle \|\x^{(i_1)}\|_2\|\x^{(p_1)}\|_2(\y^{(p_2)})^T\y^{(i_2)} \rangle_{\gamma_{22}^{(r)}} \nonumber \\
& &   + t\q_{1}
 \m_{1}  (p-1) \mE_{G,{\mathcal U}_{r+1}} \langle\|\x^{(i_1)}\|_2\|\x^{(p_1)}\|_2\|\y^{(i_2)}\|_2\|\y^{(p_2)}\|_2\rangle_{\gamma_{22}^{(r)}} \Bigg.\Bigg)\nonumber \\
& = &
\frac{\mbox{sign}(s)s\beta^2}{2\sqrt{n}} \Bigg(\Bigg. \q_{1}\lp \m_{1} -\m_{2} \rp p \mE_{G,{\mathcal U}_{r+1}} \langle\|\x^{(i_1)}\|_2\|\x^{(p_1)}\|_2\|\y^{(i_2)}\|_2\|\y^{(p_2)}\|_2\rangle_{\gamma_{2}^{(r)}}   \nonumber \\
& &    -(1-t)\lp \m_{1}-\m_{2}\rp p \mE_{G,{\mathcal U}_{r+1}} \langle \|\x^{(i_1)}\|_2\|\x^{(p_1)}\|_2(\y^{(p_2)})^T\y^{(i_2)} \rangle_{\gamma_{21}^{(r)}} \nonumber \\
& &   - t\q_{1}
\lp \m_{1}-\m_{2}\rp  p \mE_{G,{\mathcal U}_{r+1}} \langle\|\x^{(i_1)}\|_2\|\x^{(p_1)}\|_2\|\y^{(i_2)}\|_2\|\y^{(p_2)}\|_2\rangle_{\gamma_{21}^{(r)}}
 \nonumber \\
 & &
 - \q_{1} \m_{1} ( p -1 )  \mE_{G,{\mathcal U}_{r+1}} \langle\|\x^{(i_1)}\|_2\|\x^{(p_1)}\|_2\|\y^{(i_2)}\|_2\|\y^{(p_2)}\|_2\rangle_{\gamma_{2}^{(r)}} \nonumber \\
& &    +(1-t) \m_{1} (p-1) \mE_{G,{\mathcal U}_{r+1}} \langle \|\x^{(i_1)}\|_2\|\x^{(p_1)}\|_2(\y^{(p_2)})^T\y^{(i_2)} \rangle_{\gamma_{22}^{(r)}} \nonumber \\
& &   + t\q_{1}
 \m_{1}  (p-1) \mE_{G,{\mathcal U}_{r+1}} \langle\|\x^{(i_1)}\|_2\|\x^{(p_1)}\|_2\|\y^{(i_2)}\|_2\|\y^{(p_2)}\|_2\rangle_{\gamma_{22}^{(r)}} \Bigg.\Bigg)
 \nonumber \\
& = &
(1-t)\lp \m_{1}-\m_{2}\rp p \frac{\mbox{sign}(s)s\beta^2}{2\sqrt{n}} \nonumber \\
& & \times \Bigg(\Bigg.   \mE_{G,{\mathcal U}_{r+1}} \langle\|\x^{(i_1)}\|_2\|\x^{(p_1)}\|_2 \lp \q_{k_1} \|\y^{(i_2)}\|_2\|\y^{(p_2)}\|_2 -(\y^{(p_2)})^T\y^{(i_2)}\rp\rangle_{\gamma_{21}^{(r)}} \Bigg.\Bigg)
\nonumber \\
& - &
(1-t)  \m_{1}  (p-1) \frac{\mbox{sign}(s)s\beta^2}{2\sqrt{n}} \nonumber \\
& & \times \Bigg(\Bigg.   \mE_{G,{\mathcal U}_{r+1}} \langle\|\x^{(i_1)}\|_2\|\x^{(p_1)}\|_2 \lp \q_{k_1} \|\y^{(i_2)}\|_2\|\y^{(p_2)}\|_2 -(\y^{(p_2)})^T\y^{(i_2)}\rp\rangle_{\gamma_{22}^{(r)}} \Bigg.\Bigg),
 \end{eqnarray}
 where the last equality holds by noting that for constant magnitudes
 \begin{eqnarray}\label{eq:sqmag1}
 \mE_{G,{\mathcal U}_{r+1}} \langle\|\x^{(i_1)}\|_2\|\x^{(p_1)}\|_2\|\y^{(i_2)}\|_2\|\y^{(p_2)}\|_2\rangle_{\gamma_{2}^{(r)}}
& = & \mE_{G,{\mathcal U}_{r+1}} \langle\|\x^{(i_1)}\|_2\|\x^{(p_1)}\|_2\|\y^{(i_2)}\|_2\|\y^{(p_2)}\|_2\rangle_{\gamma_{21}^{(r)}}
 \nonumber \\
& = & \mE_{G,{\mathcal U}_{r+1}} \langle\|\x^{(i_1)}\|_2\|\x^{(p_1)}\|_2\|\y^{(i_2)}\|_2\|\y^{(p_2)}\|_2\rangle_{\gamma_{22}^{(r)}} .
 \end{eqnarray}
 Analogously to (\ref{eq:saip2aa0}) we also find
\begin{eqnarray}\label{eq:saip3aa0}
\frac{d\psi_1(t)}{d\q_{1}}
& = &
\frac{\mbox{sign}(s)s\beta^2}{2\sqrt{n}}\p_{1}\lp \m_{1}\omega(1;p) -\m_{2}\omega(2;p)  \rp\mE_{G,{\mathcal U}_{r+1}} \langle\|\x^{(i_1)}\|_2\|\x^{(p_1)}\|_2\|\y^{(i_2)}\|_2\|\y^{(p_2)}\|_2\rangle_{\gamma_{2}^{(r)}} \nonumber \\
& & +\frac{\mbox{sign}(s)s\beta^2}{2\sqrt{n}} \phi^{(1,\q)} \nonumber \\
& = &
\frac{\mbox{sign}(s)s\beta^2}{2\sqrt{n}} \Bigg(\Bigg. \p_{1}\lp \m_{1} -\m_{2} p\rp  \mE_{G,{\mathcal U}_{r+1}} \langle\|\x^{(i_1)}\|_2\|\x^{(p_1)}\|_2\|\y^{(i_2)}\|_2\|\y^{(p_2)}\|_2\rangle_{\gamma_{2}^{(r)}}\nonumber \\
& &    -(1-t)\lp \m_{1}-\m_{2}\rp p \mE_{G,{\mathcal U}_{r+1}} \langle \|\y^{(i_2)}\|_2\|\y^{(p_2)}\|_2(\x^{(p_1)})^T\x^{(i_1)} \rangle_{\gamma_{21}^{(r)}} \nonumber \\
& &   - t\p_{1}
\lp \m_{1}-\m_{2}\rp  p \mE_{G,{\mathcal U}_{r+1}} \langle\|\x^{(i_1)}\|_2\|\x^{(p_1)}\|_2\|\y^{(i_2)}\|_2\|\y^{(p_2)}\|_2\rangle_{\gamma_{21}^{(r)}} \nonumber \\
& &    +(1-t)  \m_{1}  (p-1) \mE_{G,{\mathcal U}_{r+1}} \langle \|\y^{(i_2)}\|_2\|\y^{(p_2)}\|_2(\x^{(p_1)})^T\x^{(i_1)} \rangle_{\gamma_{22}^{(1)}} \nonumber \\
& &   + t\p_{1}
 \m_{1} (p-1) \mE_{G,{\mathcal U}_{r+1}} \langle\|\x^{(i_1)}\|_2\|\x^{(p_1)}\|_2\|\y^{(i_2)}\|_2\|\y^{(p_2)}\|_2\rangle_{\gamma_{22}^{(r)}} \Bigg.\Bigg)\nonumber \\
& = &
\frac{\mbox{sign}(s)s\beta^2}{2\sqrt{n}} \Bigg(\Bigg. \p_{1}\lp \m_{1} -\m_{2} \rp p \mE_{G,{\mathcal U}_{r+1}} \langle\|\x^{(i_1)}\|_2\|\x^{(p_1)}\|_2\|\y^{(i_2)}\|_2\|\y^{(p_2)}\|_2\rangle_{\gamma_{2}^{(r)}}  \nonumber \\
& &    -(1-t)\lp \m_{1}-\m_{2}\rp p \mE_{G,{\mathcal U}_{r+1}} \langle \|\y^{(i_2)}\|_2\|\y^{(p_2)}\|_2(\x^{(p_1)})^T\x^{(i_1)} \rangle_{\gamma_{21}^{(r)}} \nonumber \\
& &   - t\p_{1}
\lp \m_{1}-\m_{2}\rp  p \mE_{G,{\mathcal U}_{r+1}} \langle\|\x^{(i_1)}\|_2\|\x^{(p_1)}\|_2\|\y^{(i_2)}\|_2\|\y^{(p_2)}\|_2\rangle_{\gamma_{21}^{(r)}} \nonumber \\
& &
- \p_{1}  \m_{1}( p -1)  \mE_{G,{\mathcal U}_{r+1}} \langle\|\x^{(i_1)}\|_2\|\x^{(p_1)}\|_2\|\y^{(i_2)}\|_2\|\y^{(p_2)}\|_2\rangle_{\gamma_{2}^{(r)}}
 \nonumber \\
& &    +(1-t)  \m_{1}  (p-1) \mE_{G,{\mathcal U}_{r+1}} \langle \|\y^{(i_2)}\|_2\|\y^{(p_2)}\|_2(\x^{(p_1)})^T\x^{(i_1)} \rangle_{\gamma_{22}^{(1)}} \nonumber \\
& &   + t\p_{1}
 \m_{1} (p-1) \mE_{G,{\mathcal U}_{r+1}} \langle\|\x^{(i_1)}\|_2\|\x^{(p_1)}\|_2\|\y^{(i_2)}\|_2\|\y^{(p_2)}\|_2\rangle_{\gamma_{22}^{(r)}} \Bigg.\Bigg)
 \nonumber \\
& = &
(1-t)\lp \m_{1}-\m_{2}\rp p \frac{\mbox{sign}(s)s\beta^2}{2\sqrt{n}} \nonumber \\
& & \times \Bigg(\Bigg.   \mE_{G,{\mathcal U}_{r+1}} \langle\|\y^{(i_2)}\|_2\|\y^{(p_2)}\|_2
\lp \p_{k_1} \|\x^{(i_1)}\|_2\|\x^{(p_1)}\|_2 -(\x^{(p_1)})^T\x^{(i_1)}\rp\rangle_{\gamma_{21}^{(r)}} \Bigg.\Bigg)
\nonumber \\
&  &
 - (1-t) \m_{1} ( p -1) \frac{\mbox{sign}(s)s\beta^2}{2\sqrt{n}} \nonumber \\
& & \times \Bigg(\Bigg.   \mE_{G,{\mathcal U}_{r+1}} \langle\|\y^{(i_2)}\|_2\|\y^{(p_2)}\|_2
\lp \p_{k_1} \|\x^{(i_1)}\|_2\|\x^{(p_1)}\|_2 -(\x^{(p_1)})^T\x^{(i_1)}\rp\rangle_{\gamma_{22}^{(r)}} \Bigg.\Bigg),
 \end{eqnarray}
where, as above,  the last equality holds based on (\ref{eq:sqmag1}). The above holds for any $t$. Below our focus is on  $t$ dependent vectors $\p(t)$, $\q(t)$, and $\m(t)$  and along those lines we  assume that the following system of equations is satisfied
\begin{eqnarray}\label{eq:saip4}
\frac{d\psi_1(\p,\q,\m,t)}{d\p_{1}}
& = &
(1-t)\lp \m_{1}-\m_{2}  \rp p\frac{\mbox{sign}(s)s\beta^2}{2\sqrt{n}} \nonumber \\
& & \times \Bigg(\Bigg.   \mE_{G,{\mathcal U}_{r+1}} \langle\|\x^{(i_1)}\|_2\|\x^{(p_1)}\|_2 \lp \q_{1} \|\y^{(i_2)}\|_2\|\y^{(p_2)}\|_2 -(\y^{(p_2)})^T\y^{(i_2)}\rp\rangle_{\gamma_{21}^{(r)}} \Bigg.\Bigg)\nonumber \\
&  &
- (1-t)  \m_{1} (p-1) \frac{\mbox{sign}(s)s\beta^2}{2\sqrt{n}} \nonumber \\
& & \times \Bigg(\Bigg.   \mE_{G,{\mathcal U}_{r+1}} \langle\|\x^{(i_1)}\|_2\|\x^{(p_1)}\|_2 \lp \q_{1} \|\y^{(i_2)}\|_2\|\y^{(p_2)}\|_2 -(\y^{(p_2)})^T\y^{(i_2)}\rp\rangle_{\gamma_{22}^{(r)}} \Bigg.\Bigg)\nonumber \\
& = & 0,  \nonumber \\
\frac{d\psi_1(\p,\q,\m,t)}{d\p_{k_1}}
& = &
(1-t)\lp \m_{k_1}-\m_{k_1+1}\rp p \frac{\mbox{sign}(s)s\beta^2}{2\sqrt{n}} \nonumber \\
& & \times \Bigg(\Bigg.   \mE_{G,{\mathcal U}_{r+1}} \langle\|\x^{(i_1)}\|_2\|\x^{(p_1)}\|_2 \lp \q_{k_1} \|\y^{(i_2)}\|_2\|\y^{(p_2)}\|_2 -(\y^{(p_2)})^T\y^{(i_2)}\rp\rangle_{\gamma_{k_1+1}^{(r)}} \Bigg.\Bigg)\nonumber \\
& = & 0, k_1>1,  \nonumber \\
\frac{d\psi_1(\p,\q,\m,t)}{d\q_{1}}
& = &
(1-t)\lp \m_{1}-\m_{2}\rp p \frac{\mbox{sign}(s)s\beta^2}{2\sqrt{n}} \nonumber \\
& & \times \Bigg(\Bigg.   \mE_{G,{\mathcal U}_{r+1}} \langle\|\y^{(i_2)}\|_2\|\y^{(p_2)}\|_2
\lp \p_{1} \|\x^{(i_1)}\|_2\|\x^{(p_1)}\|_2 -(\x^{(p_1)})^T\x^{(i_1)}\rp\rangle_{\gamma_{21}^{(r)}} \Bigg.\Bigg) \nonumber \\
&  &
 - (1-t) \m_{1} ( p-1) \frac{\mbox{sign}(s)s\beta^2}{2\sqrt{n}} \nonumber \\
& & \times \Bigg(\Bigg.   \mE_{G,{\mathcal U}_{r+1}} \langle\|\y^{(i_2)}\|_2\|\y^{(p_2)}\|_2
\lp \p_{1} \|\x^{(i_1)}\|_2\|\x^{(p_1)}\|_2 -(\x^{(p_1)})^T\x^{(i_1)}\rp\rangle_{\gamma_{22}^{(r)}} \Bigg.\Bigg) \nonumber \\
 & = & 0, \nonumber \\
\frac{d\psi_1(\p,\q,\m,t)}{d\q_{k_1}}
& = &
(1-t)\lp \m_{k_1}-\m_{k_1+1}\rp p \frac{\mbox{sign}(s)s\beta^2}{2\sqrt{n}} \nonumber \\
& & \times \Bigg(\Bigg.   \mE_{G,{\mathcal U}_{r+1}} \langle\|\y^{(i_2)}\|_2\|\y^{(p_2)}\|_2
\lp \p_{k_1} \|\x^{(i_1)}\|_2\|\x^{(p_1)}\|_2 -(\x^{(p_1)})^T\x^{(i_1)}\rp\rangle_{\gamma_{k_1+1}^{(r)}} \Bigg.\Bigg) \nonumber \\
 & = & 0, k_1>1, \nonumber \\
\frac{d\psi_1(\p,\q,\m,t)}{d\m_{k_1}}
 & = & 0,
  \end{eqnarray}
where  $\p_0(t)=\q_0(t)=\m_0(t)=1$ and  $\p$, $\q$, and $\m$ are added as arguments of $\psi_1(\cdot)$ due to their dependence on $t$. We can then establish the following theorem.

\begin{theorem}
\label{thm:thm5}
Assume the setup of Theorem \ref{thm:thm3} with elements of sets ${\mathcal X}$, $\bar{{\mathcal X}}$, and ${\mathcal Y}$ having equal magnitudes.
Consider the following  stationirized  fully lifted large deviation random duality theory frame (\textbf{\emph{complete sfl LD RDT frame}}). Assume that there exist $\bar{\p}(t)$, $\bar{\q}(t)$, and $\bar{\m}(t)$ that satisfy (\ref{eq:saip4}) and that $\bar{\p}(t)$ and $\bar{\q}(t)$ are not only the  overlaps $(\x^{(p_1)})^T\x^{(i_1)}$ and $(\y^{(p_2)})^T\y^{(i_2)}$ (scaled) expected values (in the sense of (\ref{eq:saip4})) but also their concentrating points (or such that $\phi_{k_1+1}^{(r)}=0$). For $\bar{\p}_0(t)=\bar{\q}_0(t)=1$,
$\bar{\p}_{r+1}(t)=\bar{\q}_{r+1}(t)=\bar{\m}_{r+1}(t)=0$, and $\m_1(t)\rightarrow\m_0(t)=1$ one also has
$\frac{d\psi_1(\bar{\p}(t),\bar{\q}(t),\bar{\m}(t),t)}{dt}   =   0$ and
\begin{eqnarray}\label{eq:thm5eq1}
\lim_{n\rightarrow\infty}\psi_1(\bar{\p}(t),\bar{\q}(t),\bar{\m}(t),t)
 =
\lim_{n\rightarrow\infty}\psi_1(\bar{\p}(0),\bar{\q}(0),\bar{\m}(0),,0)
= \lim_{n\rightarrow\infty} \psi_1(\bar{\p}(1),\bar{\q}(1),\bar{\m}(1),1),
\end{eqnarray}
with $\psi_1(\cdot)$ and $\psi(\cdot)$  as in (\ref{eq:saip1}) and  (\ref{eq:thm3eq1}), respectively.
\end{theorem}

\begin{proof}
By the choice of $\m_0$ and $\m_1$ one has  $\phi_{1}^{(r)}\rightarrow 0$. Comparing the forms in (\ref{eq:saip4}) and (\ref{eq:thm3eq5}) one then recognizes $\phi_{k_1+1}^{(r)}\rightarrow 0$ for $k_1>1$
and $\phi_{2}^{(r)} + \phi_{22}^{(r)}\rightarrow 0$  (for $k_1=1$ one also keeps in mind that expectations and concentrations are taken in the sense of (quasi) measure $\frac{(\m_1-\m_2)p}{(\m_1-\m_2)p+\m_1(1-p)} \gamma_{21}^{(r)} + \frac{\m_1(1-p)}{(\m_1-\m_2)p+\m_1(1-p)}  \gamma_{22}^{(r)} $; of additional practical interest might be to note the rotation invariance role regarding $\gamma_{21}^{(r)} = \gamma_{22}^{(r)}$ and along the same lines the suitability of choice  $f_{\bar{\x}} (\x)\sim \bar{\x}^T\x $). Based on (\ref{eq:rthlev2genanal43}) and analogously to
\cite{Stojnicsflgscompyx23}'s (101)
\begin{eqnarray}\label{eq:saip7}
\frac{d\psi_1(\bar{\p}(t),\bar{\q}(t),\bar{\m}(t),t)}{dt}
& = &
\frac{\partial\psi_1(\bar{\p}(t),\bar{\q}(t),\bar{\m}(t),t)}{\partial t}\frac{dt}{dt}
+\frac{\partial \psi_1(\bar{\p}(t),\bar{\q}(t),\bar{\m}(t),t)}{\partial \bar{\p}(t)}\frac{d\bar{\p}(t)}{dt} \nonumber \\
& &
+\frac{\partial \psi_1(\bar{\p}(t),\bar{\q}(t),\bar{\m}(t),t)}{\partial \bar{\q}(t)}\frac{d\bar{\q}(t)}{dt}
+\frac{\partial \psi_1(\bar{\p}(t),\bar{\q}(t),\bar{\m}(t),t)}{\partial \bar{\m}(t)}\frac{d\bar{\m}(t)}{dt} \nonumber \\
& = &
\frac{\partial\psi_1(\bar{\p}(t),\bar{\q}(t),\bar{\m}(t),t)}{\partial t}
  =
\frac{\partial\psi(\bar{\p}(t),\bar{\q}(t),\bar{\m}(t),t)}{\partial t}
  =        \frac{\mbox{sign}(s)\beta^2}{2\sqrt{n}} \sum_{k_1=1}^{r+1} \phi_{k_1}^{(r)}\rightarrow 0, \nonumber \\
 \end{eqnarray}
where the second equality follows by having $\bar{\p}(t)$, $\bar{\q}(t)$, and $\bar{\m}(t)$ being the stationary points of $\psi_1(\cdot)$.
\end{proof}

All the parts of the discussion from \cite{Stojnicsflgscompyx23} regarding the properties and practical relevance of the above theorem remain in place here as well. The only difference is that now the range of applicability is much wider and precisely along the lines of what is discussed in \cite{Stojnicnflldp25}. We skip repeating all of these and instead state the following two corollaries as practically the most useful.

\begin{corollary}
\label{thm:thm6}
Assume the setup of Theorem \ref{thm:thm5} with  sets ${\mathcal X}$ and ${\mathcal Y}$ being such that $\|\x\|_2=\|\y\|_2=1$. Then
 \begin{eqnarray}\label{eq:thm6eq0}
\lim_{n\rightarrow\infty}\psi_1(\bar{\p}(t),\bar{\q}(t),\bar{\m}(t),t)
 =
\lim_{n\rightarrow\infty}\psi_1(\bar{\p}(0),\bar{\q}(0),\bar{\m}(0),0)  =  \lim_{n\rightarrow\infty} \psi_1(\bar{\p}(1),\bar{\q}(1),\bar{\m}(1),1),
\end{eqnarray}
 and
\begin{eqnarray}\label{eq:thm6eq1}
\psi_1(\bar{\p}(0),\bar{\q}(0),\bar{\m}(0),0) & = & -\frac{\mbox{sign}(s) s \beta^2}{2\sqrt{n}} \sum_{k=1}^{r+1}\Bigg(\Bigg. \bar{\p}_{k-1}(0)\bar{\q}_{k-1}(0)  -\bar{\p}_{k}(0)\bar{\q}_{k}(0)   \Bigg.\Bigg)
\bar{\m}_k(0) \omega(k;p) \nonumber \\
& &  +\psi(\bar{\p}(0),\bar{\q}(0),\bar{\m}(0),0) \nonumber \\
\psi_1(\bar{\p}(1),\bar{\q}(1),\bar{\m}(1),1) & = & -\frac{\mbox{sign}(s) s \beta}{2\sqrt{n}} \sum_{k=1}^{r+1}\Bigg(\Bigg. \bar{\p}_{k-1}(1)\bar{\q}_{k-1}(1)  -\bar{\p}_{k}(1)\bar{\q}_{k}(1)   \Bigg.\Bigg)
\bar{\m}_k(1)  \omega(k;p) \nonumber \\
& &  +\psi(\bar{\p}(1),\bar{\q}(1),\bar{\m}(1),1).
 \end{eqnarray}
 Moreover, let
 \begin{equation}\label{eq:thm6eq2}
\psi_S(\p,\q,\m,t)  =  \mE_{G,{\mathcal U}_{r+1}} \frac{1}{p|s|\sqrt{n}\m_r} \log
\lp \mE_{{\mathcal U}_{r}} \lp \dots \lp \mE_{{\mathcal U}_2}\lp\lp
\lp \sum_{i_3=1}^{l} \mE_{{\mathcal U}_1}  Z_{i_3,S}^{\m_1}\rp^p \rp^{\frac{\m_2}{\m_1}}\rp\rp^{\frac{\m_3}{\m_2}} \dots \rp^{\frac{\m_{r}}{\m_{r-1}}}\rp,
\end{equation}
where,  analogously to (\ref{eq:thm3eq1}) and (\ref{eq:thm3eq2}),
\begin{equation}\label{eq:thm6eq3}
Z_{i_3,S}  \triangleq  \sum_{i_1=1}^{l}\lp\sum_{i_2=1}^{l}e^{\beta D_{0,S}^{(i_1,i_2,i_3)}} \rp^{s}, \nonumber \\
\end{equation}
with
\begin{eqnarray}\label{eq:thm6eq4}
 D_{0,S}^{(i_1,i_2,i_3)}
 & \triangleq &
 \sqrt{t}(\y^{(i_2)})^T
 G\x^{(i_1)}+\sqrt{1-t}\|\x^{(i_2)}\|_2 (\y^{(i_2)})^T\lp\sum_{k=1}^{r+1}b_k\u^{(2,k)}\rp  \nonumber \\
& &
 +\sqrt{1-t}\|\y^{(i_2)}\|_2\lp\sum_{k=1}^{r+1}c_k\h^{(k)}\rp^T\x^{(i)}
 + f_{\bar{\x}^{(i_3)} } (\x^{(i_1)}).
 \end{eqnarray}
Then
\begin{eqnarray}\label{eq:thm6eq5}
\lim_{n\rightarrow\infty} \psi_S(\bar{\p}(1),\bar{\q}(1),\bar{\m}(1),1) & = &
 -\lim_{n\rightarrow\infty}\frac{\mbox{sign}(s) s \beta^2}{2\sqrt{n}} \sum_{k=1}^{r+1}\Bigg(\Bigg. \bar{\p}_{k-1}(0)\bar{\q}_{k-1}(0)  -\bar{\p}_{k}(0)\bar{\q}_{k}(0)   \Bigg.\Bigg)
\bar{\m}_k(t) \omega(k;p)
 \nonumber \\
& &  +\lim_{n\rightarrow\infty} \psi_S(\bar{\p}(0),\bar{\q}(0),\bar{\m}(0),0). \nonumber \\
 \end{eqnarray}
\end{corollary}
\begin{proof}
One first notes that integrating out $u^{(2,k)}$ in $\psi(\cdot)$ gives
\begin{eqnarray}\label{eq:thm6eq6}
 \psi_S(\bar{\p}(1),\bar{\q}(1),\bar{\m}(1),1) & = & -\frac{\mbox{sign}(s) s \beta^2}{2\sqrt{n}} \sum_{k=1}^{r+1}\Bigg(\Bigg. \bar{\p}_{k-1}(1)\bar{\q}_{k-1}(1)  -\bar{\p}_{k}(1)\bar{\q}_{k}(1)   \Bigg.\Bigg)
\bar{\m}_k(1) \omega(k;p)
\nonumber \\
& &  +\psi(\bar{\p}(1),\bar{\q}(1),\bar{\m}(1),1). \nonumber \\
& = &  \psi_1(\bar{\p}(1),\bar{\q}(1),\bar{\m}(1),1).
 \end{eqnarray}
Moreover,
 \begin{eqnarray}\label{eq:thm6eq7}
  \psi_1(\bar{\p}(0),\bar{\q}(0),\bar{\m}(0),0) & = & -\frac{\mbox{sign}(s) s \beta^2}{2\sqrt{n}} \sum_{k=1}^{r+1}\Bigg(\Bigg. \bar{\p}_{k-1}(0)\bar{\q}_{k-1}(0)  -\bar{\p}_{k}(0)\bar{\q}_{k}(0)  \Bigg.\Bigg)
\bar{\m}_k(0)  \omega(k;p) \nonumber \\
& &  + \psi(\bar{\p}(0),\bar{\q}(0),\bar{\m}(0),0) \nonumber \\
 & = &  -\frac{\mbox{sign}(s) s \beta^2}{2\sqrt{n}} \sum_{k=1}^{r+1}\Bigg(\Bigg. \bar{\p}_{k-1}(0)\bar{\q}_{k-1}(0)  -\bar{\p}_{k}(0)\bar{\q}_{k}(0)   \Bigg.\Bigg)
\bar{\m}_k(0)   \omega(k;p)  \nonumber \\
& &  +\psi_S(\bar{\p}(0),\bar{\q}(0),\bar{\m}(0),0). \nonumber \\
 \end{eqnarray}
Keeping in mind that   $\lim_{n\rightarrow\infty}\psi_1(\bar{\p}(0),\bar{\q}(0),\bar{\m}(0),0)
=\lim_{n\rightarrow\infty}\psi_1(\bar{\p}(1),\bar{\q}(1),\bar{\m}(1),1)$, is implied by (\ref{eq:thm5eq1}) and (\ref{eq:thm6eq0}), one then has
\begin{eqnarray}\label{eq:thm6eq8}
\lim_{n\rightarrow\infty} \psi_S(\bar{\p}(1),\bar{\q}(1),\bar{\m}(1),1) & = & -\lim_{n\rightarrow\infty}\frac{\mbox{sign}(s) s \beta^2}{2\sqrt{n}} \sum_{k=1}^{r+1}\Bigg(\Bigg. \bar{\p}_{k-1}(0)\bar{\q}_{k-1}(0)  -\bar{\p}_{k}(0)\bar{\q}_{k}(0)   \Bigg.\Bigg)
\bar{\m}_k(0)  \omega(k;p) \nonumber \\
& &   +\lim_{n\rightarrow\infty}\psi_S(\bar{\p}(0),\bar{\q}(0),\bar{\m}(0),0), \nonumber \\
 \end{eqnarray}
which matches (\ref{eq:thm6eq5}).
\end{proof}

The following corollary is then directly implied by Theorem \ref{thm:thm5} and Corollary \ref{thm:thm6}.
\begin{corollary}
\label{thm:thm7}
Consider the following \textbf{\emph{modulo-$\m$ sfl LD RDT frame}}. Assume that there is an $\m$ such that $\bar{\p}(t)$ and $\bar{\q}(t)$ that fit the first four sets of equations in (\ref{eq:saip4})  are not only the overlaps $(\x^{(p_1)})^T\x^{(i_1)}$ and $(\y^{(p_2)})^T\y^{(i_2)}$
scaled expected values (in the sense of (\ref{eq:saip4})) but also their concentrating points (or such that $\phi_{k_1+1}^{(r)}=0$).  One then has
$\frac{d\psi_1(\bar{\p}(t),\bar{\q}(t),\m,t)}{dt}   =   0$,
 \begin{eqnarray}\label{eq:thm7eq0}
\lim_{n\rightarrow\infty}\psi_1(\bar{\p}(t),\bar{\q}(t),\bar{\m}(t),t)
 =
\lim_{n\rightarrow\infty}\psi_1(\bar{\p}(0),\bar{\q}(0),\bar{\m}(0),0) =  \lim_{n\rightarrow\infty} \psi_1(\bar{\p}(1),\bar{\q}(1),\bar{\m}(1),1),
\end{eqnarray}
and
\begin{eqnarray}\label{eq:thm7eq1}
\lim_{n\rightarrow\infty} \psi_S(\bar{\p}(1),\bar{\q}(1),\m,1) & \geq  &  \lim_{n\rightarrow\infty}  \inf_{\m} \Bigg(\Bigg.
-\frac{\mbox{sign}(s) s \beta^2}{2\sqrt{n}} \sum_{k=1}^{r+1}\Bigg(\Bigg. \bar{\p}_{k-1}(0)\bar{\q}_{k-1}(0)  -\bar{\p}_{k}(0)\bar{\q}_{k}(0)   \Bigg.\Bigg)
\m_{k}\omega(k;p)  \nonumber \\
& &   +   \psi_S(\bar{\p}(0),\bar{\q}(0),\m,0)\Bigg.\Bigg). \nonumber \\
 \end{eqnarray}
\end{corollary}
\begin{proof}
   Follows immediately through simple modifications of Theorem \ref{thm:thm5} and Corollary \ref{thm:thm6} proofs.
\end{proof}

\section{Conclusion}
\label{sec:lev2x3lev2liftconc}
%%%%%%%%%%%%%%%%%%%%%%%%%%%%%%%%%%%%%%%%%%%%%%%%%%%%%%%%%%%%%%%%%%%%%%%%%%%%%%%%

After \cite{Stojnicnflgscompyx23} introduced a powerful \emph{fully lifted} (fl)  blirp interpolating comparative  mechanism, \cite{Stojnicsflgscompyx23} followed up with its a particular stationarized  realization.  Companion paper \cite{Stojnicnflldp25} presents a large deviation fl upgrade thereby allowing to substantially extend the range of \cite{Stojnicnflgscompyx23}'s applicability. In particular, in addition to handling \emph{typical} random structures features (which \cite{Stojnicnflgscompyx23} can do), the machinery of \cite{Stojnicnflldp25} allows move to analytically much harder \emph{atypical} ones as  well. The same way \cite{Stojnicnflgscompyx23}  was followed by its a stationarization in \cite{Stojnicsflgscompyx23}, we here follow up on \cite{Stojnicnflldp25}  by introducing its own stationarized realization. A collection of very fundamental interpolating parameters relations is first uncovered and then shown to conveniently simplify to forms easily usable in practice.

Along the lines of the discussion in \cite{Stojnicnflldp25}, presented results directly apply to studying the so-called \emph{computational gaps} appearing in NP problems. Examples ranging from classical perceptrons to Hopfield models of prevalent interest in modern machine learning are just a few illustrative ones. The introduced concepts are very generic and their applications extend to various other random structures. All key generic principles needed for such applications are present here. To obtain concrete results for any particular application one usually needs to properly adapt the introduced methodology which typically requires a few additional technical adjustments. As these are problem specific, we present them for several most attractive application examples in separate companion papers.

%\newpage1
%\setcounter{page}{1}
\begin{singlespace}
\bibliographystyle{plain}
\bibliography{nflgscompyxRefs}
\end{singlespace}

\end{document}